\documentclass{guthesis}
\usepackage[UKenglish]{babel}
\usepackage{amsmath,amsthm,amssymb}
\usepackage{amsfonts}
\usepackage{enumerate}
\usepackage{graphicx}
\usepackage{color}
\usepackage[all]{xy}

\newtheorem{thm}{Theorem}[chapter]
\newtheorem{cor}[thm]{Corollary}
\newtheorem{lem}[thm]{Lemma}
\newtheorem{prop}[thm]{Proposition}

\newtheorem{innercustomthm}{Theorem}
\newenvironment{introthm}[1]
  {\renewcommand\theinnercustomthm{#1}\innercustomthm}
  {\endinnercustomthm}

\newtheorem{innercustomcor}{Corollary}
\newenvironment{introcor}[1]
  {\renewcommand\theinnercustomcor{#1}\innercustomcor}
  {\endinnercustomcor}

\newtheorem*{claim}{Claim}

\newtheorem{step}{Step}

\newtheoremstyle{named}{}{}{\itshape}{}{\bfseries}{.}{.5em}{#3}
\theoremstyle{named}
\newtheorem*{namedtheorem}{Theorem}

\theoremstyle{definition}
\newtheorem{defn}[thm]{Definition}
\newtheorem{ques}[thm]{Question}

\theoremstyle{remark}
\newtheorem{rem}[thm]{Remark}
\newtheorem{example}[thm]{Example}

\numberwithin{equation}{chapter}

\newcommand{\norm}[1]{\lVert #1 \rVert^2}

\newcommand{\spin}{\ifmmode{\rm Spin}\else{${\rm spin}$\ }\fi}
\newcommand{\spinc}{\ifmmode{{\rm Spin}^c}\else{${\rm spin}^c$}\fi}
\newcommand{\spinct}{\mathfrak t}
\newcommand{\spincs}{\mathfrak s}
\newcommand{\Char}{{\rm Char}}
\newcommand{\rk}{{\rm rk }}

\newcommand{\Z}{\mathbb{Z}}

\newcommand{\Q}{\mathbb{Q}}
\newcommand{\disc}{{\rm disc}}
\DeclareMathOperator*{\tors}{Tor}
\title{Alternating surgeries}		

\author{Duncan McCoy}
\submissionmonth{June}
\submissionyear{2016}
\copyrightto{D. McCoy}

\begin{document}
\maketitle
\frontmatter
\chapter*{Abstract}
This thesis is concerned with the question of when the double branched cover of an alternating knot can arise by Dehn surgery on a knot in $S^3$. We approach this problem using a surgery obstruction, first developed by Greene, which combines Donaldson's Diagonalization Theorem with the $d$-invariants of Ozsv{\'a}th and Szab{\'o}'s Heegaard Floer homology. This obstruction shows that if the double branched cover of an alternating knot or link $L$ arises by surgery on $S^3$, then for any alternating diagram the lattice associated to the Goeritz matrix takes the form of a changemaker lattice. By analyzing the structure of changemaker lattices, we show that the double branched cover of $L$ arises by non-integer surgery on $S^3$ if and only if $L$ has an alternating diagram which can be obtained by rational tangle replacement on an almost-alternating diagram of the unknot. When one considers half-integer surgery the resulting tangle replacement is simply a crossing change. This allows us to show that an alternating knot has unknotting number one if and only if it has an unknotting crossing in every alternating diagram.

These techniques also produce several other interesting results: they have applications to characterizing slopes of torus knots; they produce a new proof for a theorem of Tsukamoto on the structure of almost-alternating diagrams of the unknot; and they provide several bounds on surgeries producing the double branched covers of alternating knots which are direct generalizations of results previously known for lens space surgeries. Here, a rational number $p/q$ is said to be characterizing slope for $K\subseteq S^3$ if the oriented homeomorphism type of the manifold obtained by $p/q$-surgery on $K$ determines $K$ uniquely.

The thesis begins with an exposition of the changemaker surgery obstruction, giving an amalgamation of results due to Gibbons, Greene and the author. It then gives background material on alternating knots and changemaker lattices. The latter part of the thesis is then taken up with the applications of this theory. 
\chapter*{Acknowledgements}
I am grateful to Brendan Owens for providing a wealth of patient guidance throughout the course of my PhD studies. This thesis has greatly benefitted from the working environment provided by the maths department of Glasgow and interactions with the many excellent people therein. Of these, I would particularly like to thank Liam Watson, with whom I have had many helpful conversations. I also wish to thank Joshua Greene and acknowledge the influential role of his work, which inspired and motivated many ideas in this thesis. I would also like to thank the examining committee, Andr\'{a}s Juh\'{a}sz and Liam Watson, for their suggestions. My thanks also go to my parents for many years of support and to \r{A}sa~Lind for putting up with me.

\chapter*{Declaration}
I declare that, except where explicit reference is made to the contribution of others, this
thesis is the result of my own work and has not been submitted for any other degree at the
University of Glasgow or any other institution.

\tableofcontents
\nocite{Paddington}

\mainmatter
\chapter{Introduction}
One of the simplest ways of constructing 3-manifolds is through Dehn surgery. Given a knot $K$ in $S^3$, we perform surgery on it by cutting out a tubular neighbourhood of $K$ and gluing back in another solid torus. These gluings are naturally indexed by rational numbers. The gluing corresponding to $p/q\in \Q$ attaches the solid torus so that a curve representing $p[\mu]+q[\lambda]\in H_1(\partial( S^3\setminus \mathring{\nu}K);\Z)$ bounds a disk, where $\mu$ is meridian of $K$ and $\lambda$ a null-homologous longitude of $K$. For each $p/q\in \Q$ this allows us to define $S_{p/q}^3(K)$, the manifold obtained by $p/q$-surgery on $K$. It is known that any closed orientable 3-manifold can be obtained by surgery on a {\em link} in $S^3$, where surgery on a link means to surger each component separately \cite{Lickorish62dehnsurgery, wallace60modifications}. However, it is not hard to provide examples where this link necessarily has more than one component. For example, $\mathbb{R}P^3 \# \mathbb{R}P^3$ cannot arise by surgery on any knot in $S^3$, as for any $K$ in $S^3$, the homology group $H_1(S_{p/q}^3(K);\Z)$ is cyclic of order $p$. Given a collection of 3-manifolds, $\mathcal{M}$, there are two general questions about Dehn surgery one can ask:
\begin{enumerate}[(1)]
\item which manifolds in $\mathcal{M}$ arise by surgery on a knot $K$ in $S^3$?
\item for which $K$ and $p/q$ do we have $S_{p/q}^3(K)\in \mathcal{M}$?
\end{enumerate}
Questions of both these forms have motivated a great deal of research in low-dimensional topology and, in general, are very challenging. For example, when $\mathcal{M}$ is taken to be the set of lens spaces, which are, in a certain sense, the simplest 3-manifolds, both these questions have attracted substantial attention. For non-integer surgeries giving lens spaces, the Cyclic Surgery Theorem of Culler, Gordon, Luecke and Shalen answers (2) by showing that $K$ must be a torus knot \cite{cglscyclic}. As the lens space surgeries on torus knots were classified by Moser \cite{Moser71elementary}, this also answers $(1)$ for non-integer lens space surgeries. More recently the lens spaces arising by integer surgeries on knots in $S^3$ were determined by Greene \cite{GreeneLRP}. However, for integer surgeries (2) remains as a major open problem, known as the Berge conjecture.

\paragraph{} Given a link $L$ in $S^3$ or, more generally, in any 3-manifold, we can construct the {\em double branched cover} $\Sigma(L)$ by taking a double cover of the link complement $S^3 \setminus \nu L$ so that the cover is non-trivial on the meridian of each component, and then gluing in solid tori so that each lifted meridian bounds a disk. This construction can frequently be used to reformulate knot-theoretic questions in terms of 3-manifolds.

\paragraph{}Recall that an alternating link is one that possesses an alternating diagram, where a diagram is alternating if, following any component of the diagram round, we meet the crossings in such a way that an under-crossing is always followed by an over-crossing and vice versa; that is, the under/over nature of the crossings alternate as we traverse each component. An example of an alternating diagram is given in Figure~\ref{intro:fig:810}. The majority of this thesis is spent addressing the above Dehn surgery questions when $\mathcal{M}$ is the set of 3-manifolds obtained by taking double-branched cover of alternating knots and links.
\begin{defn}Given a knot $K\subseteq S^3$, we say that $S_{p/q}^3(K)$ is an {\em alternating surgery} if it is the double branched cover of a non-split alternating knot or link.
\end{defn}
As the lens spaces arise as the double branched covers of 2-bridge links, alternating surgeries provide a generalization of lens space surgeries. We will give a criterion in terms of alternating diagrams which characterizes precisely when the double branched cover of an alternating knot or link can arise by non-integer surgery on a knot in $S^3$. This provides an answer to (1) for non-integer alternating surgeries. This also allows us to construct a class of knots which realize all possible non-integer alternating surgeries. This class of knots potentially answers (2) for alternating surgeries. Several of our results on alternating surgeries are direct generalizations of results previously known for lens space surgeries.

The study of alternating surgeries also has applications to knot theory. The most striking of these is the classification of alternating knots with unknotting number one. We are also able to obtain a new proof for a theorem of Tsukamoto on the structure of almost-alternating diagrams of the unknot \cite{tsukamoto2009almost}. The techniques in use also have applications to the problem of determining the characterizing slopes of torus knots.

\section{Unknotting number one}
Given a knot $K\subseteq S^3$, its {\em unknotting number}, $u(K)$, is a classical knot invariant going back to the work of Tait in the 19th century \cite{tait1876knots}. It is defined to be the minimal number of crossing changes required in any diagram of $K$ to obtain the unknot. Upper bounds for the unknotting number are easy to obtain, since one can take any diagram and find a sequence of crossing changes giving the unknot. It is far harder to establish effective lower bounds for the unknotting number, as it is not generally known which diagrams will exhibit the actual unknotting number \cite{bernhard1994unknotting, bleiler1984note, jablan1998unknotting}. The quintessential example which illustrates the difficulty of the unknotting number is the knot $8_{10}$, show in Figure~\ref{intro:fig:810}; its unknotting number was long-believed to be two, but this remained unproven until the work of Ozsv{\'a}th and Szab{\'o} in 2005 \cite{ozsvath2005knots}.
\begin{figure}[h]
\centering
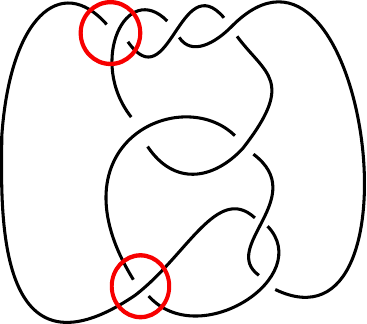
\label{intro:fig:810}
\caption{\label{intro:fig:810}The knot $8_{10}$. This diagram can be unknotted by changing the two circled crossings, but not by any single crossing. It follows from Theorem~\ref{Intro:thm:unknotting} that $u(8_{10})=2$.}
\end{figure}

One classical lower bound is the signature of a knot as defined by Trotter \cite{trotter1962homology}, which satisfies $|\sigma(K)|\leq 2 u(K)$ \cite{murasugi1965certain}. This is a particularly useful bound, since it may be computed in a variety of ways \cite{trotter1962homology, gordon1978signature}. Other bounds and obstructions have been constructed through the use of various knot-theoretic and topological invariants, including, among others, the Alexander module \cite[Theorem~7.10]{lickorish1997introduction} and the intersection form of 4-manifolds \cite{cochran1986unknotting, owens2008unknotting}.

The case of unknotting number one has been particularly well-studied. Recall that a {\em minimal diagram} for a knot is one containing the minimal possible number of crossings. Kohn made the following conjecture regarding unknotting number one knots and their minimal diagrams \cite[Conjecture~12]{kohn1991two}.

\begin{namedtheorem}[Kohn's Conjecture]
If $K$ is a knot with $u(K)=1$, then it has an unknotting crossing in a minimal diagram.
\end{namedtheorem}

This has been resolved in a number of cases. The two-bridge knots with unknotting number one were classified by Kanenobu and Murakami \cite{Unknottwobridge}, using the Cyclic Surgery Theorem \cite{cglscyclic}. For alternating large algebraic knots, the conjecture was settled by Gordon and Luecke \cite{gordon2006knots}. Most recently, the conjecture was proved for alternating 3-braid knots by Greene \cite{Greene3Braid}. The following theorem addresses Kohn's conjecture for all alternating knots \cite{mccoy2013alternating}.
\begin{introthm}{\ref{Intro:thm:unknotting}}
For an alternating knot $K$ the following are equivalent:
\begin{enumerate}[(i)]
\item $K$ has unknotting number one;
\item The branched double cover $\Sigma(K)$ can be obtained by half-integer surgery on a knot in $S^3$;
\item $K$ has an unknotting crossing in every alternating diagram.
\end{enumerate}
\end{introthm}
Recall that an alternating diagram is said to be {\em reduced} if it contains no nugatory crossings (see Figure~\ref{fig:nugatory}). Since the minimal diagrams of alternating knots are precisely the reduced alternating diagrams \cite{kauffman87state, murasugi86jones, thistlethwaite88alternating}, this resolves Kohn's conjecture for alternating knots.
\begin{figure}[h]
  \centering
  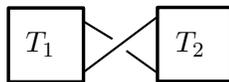
 \caption{A nugatory crossing.}
 \label{fig:nugatory}
\end{figure}

\begin{rem} In general, it is not known if the question of whether a knot has unknotting number one is decidable. Theorem~\ref{Intro:thm:unknotting} shows that for alternating knots this is decidable.
\end{rem}

In general, Kohn's Conjecture seems somewhat optimistic. For example, there are 14-crossing knots with unknotting number one and minimal diagrams not containing an unknotting crossing \cite{stoimenow2001examples}. However these examples are not sufficient to disprove the conjecture, since they still all possess at least one minimal diagram with an unknotting crossing.

The equivalence between the double branched cover arising by half-integer surgery and unknotting number one is not true in general.
\begin{example}\label{intro:exam:montyconverse}
Consider the torus knot $T_{r,s}$ for $r,s>0$. Its unknotting number is given by the formula \cite{kronheimer93gauge, rasmussen10khovanov}:
\[u(T_{r,s})=\frac{1}{2}(r-1)(s-1).\]
So the only torus knot with unknotting number one is $T_{3,2}$. However, for odd $q>1$, we have
$S^3_{\pm 1/2}(T_{2,q})\cong \Sigma (T_{q,4q\mp 1})$ \cite[Proposition 2]{Watson10Remark}.
\end{example}

\section{Almost-alternating diagrams of the unknot}
Theorem~\ref{Intro:thm:unknotting} can be interpreted as follows: showing that understanding alternating knots with unknotting number one is equivalent to understanding almost-alternating diagrams of the unknot, where an {\em almost-alternating diagram} is one which is obtained from an alternating diagram by a single crossing change. A result of Tsukamoto shows that any reduced almost-alternating diagram of the unknot can be built up using only certain types of isotopies: {\em flypes}, which are illustrated in Figure~\ref{fig:flypedef}; and {\em tongue} and {\em twirl} moves, which are the inverses of the {\em untongue} and {\em untwirl} moves depicted in Figure~\ref{fig:unswirluntongue}.

The methods in this paper also provide a new proof of this result \cite[Corollary~1.1]{tsukamoto2009almost}.
\begin{introthm}{\ref{intro:thm:tsukamoto}}[Tsukamoto]
Any reduced almost-alternating diagram of the unknot can be obtained from $\mathcal{C}_m$, for some non-zero integer $m$, by a sequence of flypes, tongue moves and twirl moves. For each $m$, $\mathcal{C}_m$ is the almost-alternating diagram shown in Figure~\ref{fig:claspdiagram}.
\end{introthm}

The proof given in this thesis is of very different flavour to the original which employed spanning surfaces and geometric arguments, rather than the machinery of Heegaard Floer homology and the topology of 4-manifolds. Together Theorem~\ref{Intro:thm:unknotting} and Theorem~\ref{intro:thm:tsukamoto} may be viewed as a complete description of alternating knots with unknotting number one.

\begin{figure}[p!]
  \centering
  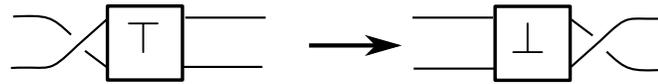
 \caption[A flype]{A flype. It is known that any two reduced alternating diagrams of a knot can be related by a sequence of flypes \cite{Menasco93classification}.}
 \label{fig:flypedef}
\end{figure}

\begin{figure}[p!]
  \centering
  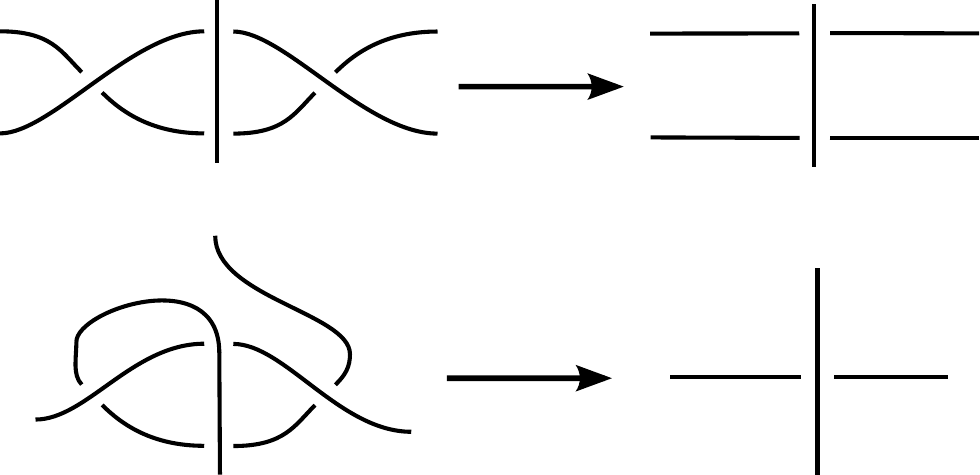
 \caption[An untongue and an untwirl move]{An untongue move (top) and an untwirl (bottom). Their inverses are known as tongue and twirl moves respectively.}
 \label{fig:unswirluntongue}
\end{figure}

\begin{figure}[p!]
  \centering
  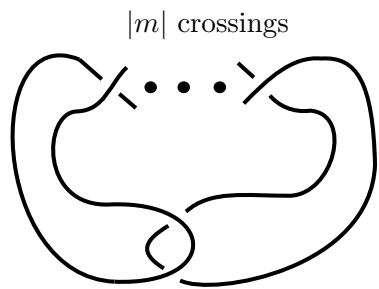
 \caption[The diagram $\mathcal{C}_m$]{The diagram $\mathcal{C}_m$, where the $|m|$ crossings are positive if $m>0$ and negative otherwise.}
 \label{fig:claspdiagram}
\end{figure}

\section{Alternating surgeries}
A {\em rational tangle} $T \subseteq B^3$ is a tangle consisting of two properly embedded arcs, such that $(B^3,T)$ is diffeomorphic as a pair to the standard unknotted tangle in $B^3$ with two arcs (this diffeomorphism is not required to fix the boundary).
Given a knot or link $L$ in $S^3$, one can obtain a new link by {\em rational tangle replacement}, that is by replacing one rational tangle in $L$ with some other rational tangle. For example, changing or resolving a crossing are both rational tangle replacements. As the double branched cover of a rational tangle is a solid torus, if $L'$ is obtained from $L$ by rational tangle replacement, then the branched double cover $\Sigma(L')$ can be obtained by surgery on some knot in $\Sigma(L)$ \cite{montesinos1973variedades}. This correspondence between tangle replacement and surgery is sometimes referred to as the {\em Montesinos trick}.
We get a criterion for when the double branched cover of an alternating link can arise by non-integer surgery in terms of rational tangle replacement in an alternating diagram \cite{mccoy2014noninteger}.

\begin{introthm}{\ref{intro:thm:nonint}}
Let $L$ be an alternating knot or link. For any $p/q$ with $|p|>q>1$, the double branched cover $\Sigma(L)$ arises as $p/q$-surgery on a knot in $S^3$ if and only if (1) $L$ possesses an alternating diagram obtained by rational tangle replacement from an almost-alternating diagram of the unknot and (2) the corresponding surgery slope is $p/q$.
\end{introthm}
\begin{rem}
It turns out that if $S_{p/q}^3(\kappa)$ is an alternating surgery for some $|p/q|<1$, then $\kappa$ is the unknot. Since surgeries on the unknot always satisfy $S_{p/q}^3(U)\cong S_{p/q'}^3(U)$ for some $p/q'\geq 1$, the assumption that $|p|>q$ in Theorem~\ref{intro:thm:nonint} is not a serious loss of generality.
\end{rem}

Not only does Theorem~\ref{intro:thm:nonint} show which double branched covers of alternating links arise by non-integer surgery on a knot in $S^3$, but the Montesinos trick also allows us to construct knots on which these alternating surgeries occur.

Let $D'$ be an almost-alternating diagram of the unknot. Consider a small ball $B$ containing the non-alternating crossing. If we take the double cover of $S^3$ branched over $D'$, then $B$ lifts to a solid torus. Since the double branched cover of the unknot is $S^3$, the core of this solid torus is a knot $\kappa_{D'}$ in $S^3$. Let $D$ be a diagram obtained by replacing the single crossing in the interior of $B$ with some other rational tangle. We see that $\Sigma(D)$ arises by surgery on $\kappa_{D'}$. As $D'$ is almost-alternating, there are tangle replacements for which the resulting diagram $D$ is alternating, showing that $\kappa_{D'}$ admits alternating surgeries.
\begin{defn}\label{intro:def:D}
Let $\mathcal{D}$ denote the collection of all knots arising in this way:
\[\mathcal{D}= \{ \kappa_{D'} \subseteq S^3 \,|\, \text{$D'$ is an almost-almost alternating diagram of the unknot} \}.\]
\end{defn}

Theorem~\ref{intro:thm:nonint} shows that $\mathcal{D}$ accounts for all possible non-integer alternating surgeries in the following sense.
\begin{introthm}{\ref{intro:thm:montyconverse}}
If $S_{p/q}^3(\kappa)$ is an alternating surgery for $q>1$, then there is $\kappa'\in \mathcal{D}$ with
\[S_{p/q}^3(\kappa)\cong S_{p/q}^3(\kappa')
\quad \text{and} \quad
\Delta_\kappa(t)=\Delta_{\kappa'}(t).\]
\end{introthm}
Here $\Delta_K(t)$ denotes the Alexander polynomial of $K$. In fact, the only known knots with alternating surgeries are those in $\mathcal{D}$. This raises the following question.
\begin{ques}
Suppose that $\kappa$ admits an alternating surgery, is $\kappa$ in $\mathcal{D}$?
\end{ques}

For a given knot, it is also possible to restrict the number of alternating surgeries that it admits \cite{mccoy2014bounds}.
\begin{introthm}{\ref{intro:thm:widthbound}}
Let $\kappa$ be a non-trivial knot admitting alternating surgeries. There is an integer $N$, which can be calculated from the Alexander polynomial of $\kappa$, such that if $S_{p/q}^3(\kappa)$ is an alternating surgery, then
\[N-1\leq p/q \leq N+1.\]
\end{introthm}
Since it can be shown that the $N$ in Theorem~\ref{intro:thm:widthbound} satisfies $|N|\leq 4g(\kappa)+2$, we obtain a generalization of a bound on lens space surgeries originally due to Rasmussen \cite{Rasmussen04Goda}.
\begin{introcor}{\ref{intro:cor:Rasmusbound}}
Let $\kappa$ be a non-trivial knot. If $S_{p/q}^3(\kappa)$ is an alternating surgery, then
\[|p/q| \leq 4g(\kappa)+3.\]
\end{introcor}

\section{Heegaard Floer homology and $d$-invariants}
Heegaard Floer homology forms a package of invariants, defined by Ozsv{\'a}th and Szab{\'o}, which can be associated with any closed oriented 3-manifold $Y$ \cite{Ozsvath04threemanifoldinvariants}. In its simplest form, Heegaard Floer homology (with coefficients in $\Z/2\Z$) defines a finite-dimensional $\Z/2\Z$-vector space $\widehat{HF}(Y)$. There are related Heegaard Floer homology invariants $HF^+(Y)$, $HF^-(Y)$ and $HF^\infty(Y)$ associated with $Y$, along with various exact sequences connecting these invariants.

The Heegaard Floer homology of any 3-manifold $Y$ splits as a direct sum over its \spinc-structures:
\[\widehat{HF}(Y)\cong \bigoplus_{\spinct \in \spinc(Y)}\widehat{HF}(Y,\spinct).\]
When $Y$ is a rational homology sphere each of these summands satisfies $\chi(\widehat{HF}(Y,\spinct))=1$ and, in particular, is non-trivial \cite{Ozsvath04applications}. Since $\spinc(Y)$ is in bijection with $H^2(Y;\Z)\cong H_1(Y;\Z)$, we have
\[\dim_{\Z/2\Z} \widehat{HF}(Y) \geq |H_1(Y; \Z)|,\]
for any rational homology sphere $Y$. When this minimum is realized we say that $Y$ is an {\em $L$-space}.

For rational homology spheres there are also associated numerical invariants, $d(Y,\spinct)\in \mathbb{Q}$, called the {\em $d$-invariants} \cite{Ozsvath03Absolutely}. These have the property that for all $\spinct \in \spinc(Y)$,
\[d(-Y,\spinct)=-d(Y,\spinct),\]
where $-Y$ denotes $Y$ with its orientation reversed, and
\[d(Y,\spinct)=d(Y,\overline\spinct),\]
where $\overline\spinct$ denotes the conjugate of $\spinct$.

\subsection{The knot Floer complex}\label{intro:sec:knotfloer}
Given a null-homologous knot $K$ in $Y$, one can also define knot Floer homology, as developed independently by Rasmussen \cite{Rasmussen03Thesis} and by Ozsv{\'a}th and Szab{\'o} \cite{Ozsvath04knotinvariants}. In particular, to any knot in $K \subseteq S^3$ one can associate a knot Floer complex, $CFK^\infty(K)$, which takes the form of a bifiltered chain complex
\[CFK^\infty(K)=\bigoplus_{i,j\in \Z}C\{(i,j)\},\]
whose chain homotopy type (as a bifiltered complex) is an invariant of $K$. This comes equipped with a chain complex isomorphism
\[U: CFK^\infty(K)\longrightarrow CFK^\infty(K),\]
which maps $C\{(i,j)\}$ isomorphically to $C\{(i-1,j-1)\}$.

The complex $CFK^\infty(K)$ has quotient complexes defined by
\[A_k^+=C\{i\geq 0 \text{ or } j\geq k\},\]
for each $k\in \Z$ and
\[B^+=C\{i\geq 0\}.\]
These complexes admit chain maps
\[v_k,h_k \colon A_k^+ \longrightarrow B^+,\]
where $v_k$ is the obvious vertical projection, and $h_k$ consists of the composition of a horizontal projection onto $C\{j\geq k\}$ and a chain homotopy equivalence. When restricted to sufficiently large gradings $v_k$ and $h_k$ induce $U$-equivariant isomorphisms on homology, meaning that they are modelled on multiplication by $U^{V_k}$ and $U^{H_k}$ respectively, for some non-negative integers $V_k$ and $H_k$ \cite{ni2010cosmetic}. These numbers are invariants of the chain homotopy type of $CFK^\infty(K)$ and hence are invariants of $K$.
The following proposition summarizes the properties of the $V_k$ and $H_k$ that we will use in this thesis.
\begin{prop}[Ni-Wu, \cite{ni2010cosmetic}]\label{intro:prop:Viproperties}
For any $K\subseteq S^3$ the $V_k$ and $H_k$ are a sequence of non-negative integers satisfying \[V_{k-1}\geq V_k\geq V_{k-1}-1 \quad\text{and}\quad H_k=V_{-k}\]
for all $k$.
\end{prop}
The sequence of $V_k$ is eventually zero and we will use $\nu^{+}=\nu^{+}(K)$ to denote the minimal integer such that $V_{\nu^{+}}=0$. This always satisfies $\nu^{+}\leq g_4(K)$, where $g_4(K)$ denotes the smooth slice genus of $K$ \cite[Theorem~2.3]{Rasmussen04Goda}.

\subsection{Dehn surgery and $d$-invariants}
We will primarily be interested in the $d$-invariants of $S^3_{p/q}(K)$, the manifold obtained by $p/q$-surgery on the knot $K\subseteq S^3$. These $d$-invariants can be calculated in terms of the integers $V_k$ and $H_k$ arising from $CFK^\infty(K)$. By using relative \spinc-structures on $S^3 \setminus \mathring{\nu} K$, Ozsv{\'a}th and Szab{\'o} establish identifications \cite{ozsvath2011rationalsurgery}
\begin{equation*}
\spinc(S_{p/q}^3(K)) \longleftrightarrow \Z/p\Z \longleftrightarrow \spinc(S_{p/q}^3(U)).
\end{equation*}
Using this identification, Ni and Wu derived a formula for the $d$-invariants of $S_{p/q}^3(K)$.
\begin{prop}[Proposition 1.6, \cite{ni2010cosmetic}]\label{intro:prop:NiWuformula}
If $p/q>0$, then for any $0\leq i \leq p-1$ the $d$-invariants of $S_{p/q}^3(K)$ satisfy
\[d(S_{p/q}^3(U),i)-d(S_{p/q}^3(K),i)=2\min \{ V_{\lfloor \frac{i}{q} \rfloor} , H_{\lfloor \frac{i-p}{q} \rfloor}\}.\]
\end{prop}
Given the properties of the $H_k$ and $V_k$ in Proposition~\ref{intro:prop:Viproperties}, we see that the formula for the $d$-invariants in Proposition~\ref{intro:prop:NiWuformula} can be rewritten purely in terms of the $V_k$:
\begin{equation*}
d(S_{p/q}^3(U),i)-d(S_{p/q}^3(K),i)=2V_{\min \{\lfloor \frac{i}{q} \rfloor,\lceil \frac{p-i}{q} \rceil\}}.
\end{equation*}

\subsection{$L$-space knots}
A knot $K$ is said to be an {\em $L$-space knot} if $S^3_{p/q}(K)$ is an $L$-space for some $p/q > 0$. It turns out that if $K$ is an $L$-space knot, then $S^3_{p/q}(K)$ is an $L$-space if and only if $p/q \geq 2g(K)-1$. The knot Floer homology of an $L$-space knot is known to be determined by its Alexander polynomial, which can be written in the form
\[\Delta_K(t)=a_0 +\sum_{i=1}^g a_i(t^i+t^{-i}),\]
where $g=g(K)$ and the non-zero values of $a_i$ alternate in sign and assume values in $\{\pm 1\}$ with $a_g=1$ \cite{Ozsvath04genusbounds, Ozsvath05Lensspace}. Given an Alexander polynomial in this form, we define its {\em torsion coefficients} by the formula
\[t_i(K) = \sum_{j\geq 1}ja_{|i|+j}.\]

When $K$ is an $L$-space knot, the $V_i$ appearing in Section~\ref{intro:sec:knotfloer} satisfy $V_i=t_i(K)$ for $i\geq 0$ \cite{ozsvath2011rationalsurgery}. In particular, this shows that $\nu^{+}(K)=g_4(K)=g(K)$.
\begin{rem}\label{Ch1:rem:torsiondeterminespoly}
Observe that the torsion coefficients uniquely determine the Alexander polynomial. For $j\geq 1$, $a_j$ is determined by the relation
\[a_{j}=t_{j-1}(K) -2t_{j}(K)+t_{j+1}(K).\]
Since the Alexander polynomial is normalized so that $\Delta_K(1)=1$, this is also sufficient to determine the coefficient $a_0$.
\end{rem}

\subsection{4-manifolds}
Other properties of $d$-invariants arise from their interaction with 4-manifolds. A compact, oriented 4-manifold has an {\em intersection form}, which is a symmetric bilinear pairing
\[Q_X\colon H_2(X; \Z) \times H_2(X; \Z) \rightarrow \Z.\]
Note that $Q_X(x,y)=0$ whenever $x$ or $y$ is a torsion element in $H_2(X; \Z)$, so this pairing descends to a pairing on $H_2(X; \Z)/{\rm Tors}$, which we will also call $Q_X$. If $\partial X$ is a disjoint union of rational homology spheres, then this pairing is non-degenerate. In this case we can identify $H^2(X;\Z)/{\rm Tors}$ with the dual group ${\rm Hom}(H_2(X;\Z),\Z)\subseteq H_2(X)\otimes \Q$. This gives a $\Q$-valued pairing on $H^2(X;\Z)$ (and also on $H^2(X;\Z)\times H^2(X;\Z)$). We will denote the pairings on both $H_2(X;\Z)$ and $H^2(X;\Z)$ by $(x,y) \mapsto x \cdot y$.

\paragraph{}Let $X$ be a negative-definite smooth 4-manifold $X$ with boundary $Y$ a union of disjoint rational homology spheres. For any $\spincs \in \spinc(X)$ which restricts to $\spinct \in \spinc(Y)$ there is a bound on the $d$-invariant \cite[Theorem~9.6]{Ozsvath03Absolutely}:
\begin{equation}\label{Ch1:eq:dinvinequality}
\frac{c_1(\spincs)^2+b_2(X)}{4}\leq d(Y,\spinct).
\end{equation}
Recall that for any $\spincs \in \spinc(X)$, the corresponding first Chern class $c_1(\spincs)\in H^2(X;\Z)$ is a {\em characteristic covector}, meaning that it satisfies
\[c_1(\spincs)\cdot x \equiv x \cdot x \bmod{2},\]
for all $x\in H_2(X; \Z)$.
We will be interested in the special case where $X$ is sharp, meaning that its intersection form determines the $d$-invariants of its boundary.
\begin{defn}
The negative-definite manifold $X$ with $\partial X=Y$ is {\em sharp} if for every $\spinct \in \spinc(Y)$ there is some $\spincs \in \spinc(X)$ which restricts to $\spinct$ and attains equality in \eqref{Ch1:eq:dinvinequality}. That is, $X$ is sharp if and only if
\[4d(Y,\spinct)=\max_{ \substack{\spincs\in \spinc(X) \\ \spincs|_Y= \spinct}}c_1(\spincs)^2+b_2(X),\]
for every $\spinct \in \spinc(Y)$.
\end{defn}
In this thesis, an integer lattice is taken to mean a free abelian group with an integer-valued bilinear pairing. As a sharp manifold $X$ is negative-definite, the pairing $-Q_X$ is positive-definite, making the group $H_2(X; \Z)/{\rm Tors}$ into an integer lattice. We use $-Q_X$ to denote this lattice.
\begin{rem}
Note that if $X$ is sharp, then $-X$ is positive-definite, has boundary $-Y$, and satisfies
\[-4d(Y,\spinct)=\min_{ \substack{\spincs\in \spinc(-X) \\ \spincs|_{-Y}= \spinct}}c_1(\spincs)^2-b_2(X),\]
for every $\spinct \in \spinc(Y)$.
\end{rem}

\section{The Changemaker Theorem}
The results of thesis rest on an obstruction to Dehn surgery obtained by combining the $d$-invariants of Heegaard Floer homology with Donaldson's Diagonalization Theorem. This obstruction was first developed by Greene for integer surgeries in his work on the lens space realization problem \cite{GreeneLRP} and the cabling conjecture \cite{greene2010space}, and for half-integer surgeries in his work on 3-braid knots with unknotting number one \cite{Greene3Braid}. It was extended, with extra hypotheses, to general non-integer slopes by Gibbons \cite{gibbons2013deficiency}.

\begin{defn} We say that a tuple of increasing positive integers $(\sigma_1, \dots, \sigma_t)$ satisfies the {\em changemaker condition}, if for every
\[1\leq n \leq \sigma_1 + \dots + \sigma_t,\]
there is $A\subseteq \{ 1, \dots , t\}$ such that $n=\sum_{i\in A} \sigma_i$.
\end{defn}
The changemaker has the following equivalent formulation which will sometimes be useful.
\begin{prop}[Brown, \cite{brown1961note}]\label{Ch3:prop:CMcondition}
A tuple $(\sigma_1, \dots , \sigma_t)$ of increasing positive integers satisfy the changemaker condition if and only if
\[\sigma_1 = 1, \quad\text{and}\quad \sigma_i \leq \sigma_1 + \dots + \sigma_{i-1} +1,\text{ for } 1<i\leq t.\]
\end{prop}
The key definition we will need is that of a changemaker lattice.
\begin{defn}\label{Intro:def:CMlattice}
Let $p/q>0$ be given by the continued fraction,
\[p/q=[a_0, \dots, a_l]^-=
a_0 -
    \cfrac{1}{a_1
        - \cfrac{1}{\ddots
            - \cfrac{1}{a_l} } },\]
where $a_0\geq 1$ and $a_i\geq 2$ for $i\geq 1$. Let $w_0, \dots, w_l\in \Z^N$ be vectors satisfying,
\[
w_i\cdot w_j =
\begin{cases}
a_i &\text{if $i=j$,}\\
-1 &\text{if $|i-j|=1$,}\\
0 &\text{if $|i-j|>1$.}
\end{cases}
\]
Suppose further that there is an orthonormal basis $\{e_1, \dots, e_N\}$ of $\Z^N$ for such that the following hold:
\begin{enumerate}[(I)]
\item there is a tuple $(\sigma_1, \dots, \sigma_t)$ satisfying the changemaker condition such that
\[
w_0=
\begin{cases}
\sigma_1 e_1 + \dots + \sigma_t e_t &\text{if $l=0$}\\
e_{t+1} + \sigma_1 e_1 + \dots + \sigma_t e_t &\text{if $l>0$;}
\end{cases}
\]
\item $|w_k\cdot e_i|\leq 1$ for all $1\leq k\leq l$ and all $e_i$;
\item for any $0\leq i < j\leq l$,
\[w_i\cdot w_j = -|I_j \cap I_i|=
\begin{cases}
-1 &\text{if $j=i+1$}\\
0 &\text{if $j>i+1$};
\end{cases}\]
where $I_k$ is the set $I_k=\{e_j\, | \, w_k\cdot e_j \ne 0\}$; and
\item for any $e_i$ there is $w_k$ such that $e_i\cdot w_k\ne 0$.
\end{enumerate}
Then we say that the orthogonal complement
\[L=\langle w_0, \dots, w_l \rangle^\bot \subseteq \Z^N\]
is a {\em $p/q$-changemaker lattice}.

Moreover, we say that the $\sigma_i$ are the {\em changemaker coefficients} of $L$ and that the $\sigma_i$ satisfying $\sigma_i>1$ are the {\em stable coefficients} of $L$.
\end{defn}
\begin{rem}
Note that a $p/q$-changemaker lattice is determined up to isomorphism by its stable coefficients. Given the stable coefficients, the remaining changemaker coefficients are determined by the requirement that $a_0=\norm{w_0}=\lceil p/q \rceil$. Up to automorphisms of $\Z^N$, the remaining $w_i$ are determined by Conditions II and III. Condition IV then determines $N$.
\end{rem}

This allows us to state the changemaker surgery obstruction.
\begin{introthm}{\ref{intro:thm:CM}}
Let $K\subseteq S^3$ such that $S_{p/q}^3(K)$ bounds a sharp manifold $X$ for $p/q>0$. Then the intersection form $Q_X$ satisfies
\[-Q_X\cong L \oplus \Z^S,\]
where $S\geq 0$ is an integer and
\[L=\langle w_0, \dots, w_l \rangle^\bot\subseteq \Z^N,\]
is a $p/q$-changemaker lattice.
\end{introthm}
\begin{rem} Condition IV in Definition~\ref{Intro:def:CMlattice} guarantees that a changemaker lattice does not include any vectors of norm 1. In almost all applications of Theorem~\ref{intro:thm:CM} we will know {\em a priori} that $-Q_X$ contains no elements of norm 1, and thus be able assume that $S=0$.
\end{rem} 
\chapter{Surgeries and sharp manifolds}\label{chap:sharpman}
The first versions of the changemaker theorem were established by Greene for integer and half-integer $L$-space surgeries in his work on the lens space realization problem \cite{GreeneLRP}, the cabling conjecture \cite{greene2010space} and alternating 3-braid knots with unknotting number one \cite{Greene3Braid}. These results were generalized by Gibbons who proved a changemaker theorem that held for surgeries of arbitrary rational slope, and, because of the developments in the theory of $d$-invariants due to Ni and Wu \cite{ni2010cosmetic}, no longer required the surgeries to be $L$-space surgeries \cite{gibbons2013deficiency}. Gibbons's result did, however, have additional conditions on the $d$-invariants of the manifold obtained by surgery. By adapting ideas from Greene's work \cite{greene2010space}, one can show that these conditions are redundant, giving a changemaker theorem valid for all positive rational surgeries bounding a sharp 4-manifold \cite{mccoy2014noninteger}.

The main purpose of this chapter is to establish the changemaker theorem in the greatest possible generality.
\begin{thm}\label{intro:thm:CM}
Let $K\subseteq S^3$ such that $S_{p/q}^3(K)$ bounds a sharp manifold $X$ for $p/q>0$. Then the intersection form $Q_X$ satisfies
\[-Q_X\cong L \oplus \Z^S,\]
where $S\geq 0$ is an integer and
\[L=\langle w_0, \dots, w_l \rangle^\bot\subseteq \Z^N,\]
is a $p/q$-changemaker lattice such that for all $0\leq i \leq n/2$,
\begin{equation}\label{Ch1:eq:w0formula2}
8V_{i} = \min_{\substack{ |c\cdot w_0|= n-2i \\ c \in \Char(\Z^{N})}} \norm{c} - N,
\end{equation}
where $n=\lceil p/q \rceil$.
\end{thm}
The $V_i$ appearing in \eqref{Ch1:eq:w0formula2} are the integers derived from the knot Floer complex of $K$ as described in Section~\ref{intro:sec:knotfloer}. The relationship in \eqref{Ch1:eq:w0formula2} was first established by Greene for integer surgeries \cite{greene2010space} and later observed to hold for rational surgeries as well \cite{mccoy2014sharp}.

\begin{rem}\label{Ch1:rem:CMgivesallVi}
As there are elements of norm $N$ in $\Char(\Z^N)$, \eqref{Ch1:eq:w0formula2} shows that $V_k=0$ for some $0\leq k \leq n/2$. As the $V_i$ are a non-negative decreasing sequence, this shows that $V_i=0$ for all $i\geq n/2$. This means that if $S_{p/q}^3(K)$ bounds a sharp manifold then knowing $w_0$ is sufficient to determine the $V_i$ for all $i\geq 0$.
\end{rem}
We will also show that \eqref{Ch1:eq:w0formula2} is sufficient to determine the vector $w_0$. It turns out that the stable coefficients of $L$ are determined by the $V_i$, making them invariants of $K$. In particular, the stable coefficients are independent of the surgery slope $p/q$. This is the sense in which the stable coefficients are stable.
\begin{thm}\label{Ch1:thm:stabledependence}
Let $K\subseteq S^3$ such that $S_{p/q}^3(K)$ bounds a sharp manifold $X$ for $p/q>0$. Then the intersection form $Q_X$ satisfies
\[-Q_X\cong L \oplus \Z^S,\]
where $L$ is a $p/q$-changemaker lattice and the stable coefficients are an invariant of $K$. Therefore, the intersection form of $X$ is determined by the knot $K$, the surgery slope $p/q$ and $b_2(X)$.
\end{thm}
If $X$ is a sharp manifold with boundary $Y$, then we can obtain another sharp manifold by taking a connect sum of $X$ with $\overline{\mathbb{C}P}^2$. It follows from Theorem~\ref{Ch1:thm:stabledependence} that if $Y\cong S^3_{p/q}(K)$, then at the level of intersection forms this is the only possibility.
\begin{cor}\label{Ch1:cor:interfromunique}
Let $K \subseteq S^3$ be a knot such that for some $p/q>0$, the 3-manifold $S^3_{p/q}(K)$ bounds negative-definite sharp manifolds $X$ and $X'$, with $b_2(X')= b_2(X)+k$ for $k\geq 0$, then
\[ Q_{X'} \cong Q_{X} \oplus (-\mathbb{Z}^{k})\cong Q_{X\#_{k}\overline{\mathbb{C}P}^2}.\]
\end{cor}
We will also prove the following theorem, which provides further examples of sharp manifolds to which Theorem~\ref{intro:thm:CM} can be applied \cite{mccoy2014sharp}.
\begin{thm}\label{Ch1:thm:sharpextension}
If $S^3_{p'/q'}(K)$ bounds a sharp manifold for some $p'/q'>0$, then $S^3_{p/q}(K)$ bounds a sharp manifold for all $p/q\geq p'/q'$.
\end{thm}
Owens and Strle show that if $S^3_{p'/q'}(K)$ bounds a negative-definite manifold $X$, then $S^3_{p/q}(K)$ bounds a negative-definite manifold for any $p/q\geq p'/q'$ by taking a certain negative-definite cobordism from $S^3_{p'/q'}(K)$ to $S^3_{p/q}(K)$ and gluing it to $X$ \cite{Owens12negdef}. We prove Theorem~\ref{Ch1:thm:sharpextension} by showing that if $X$ is sharp, then this construction results in a sharp manifold.

\section{A manifold bounding $S^3_{p/q}(K)$}
\begin{figure}
  \centering
  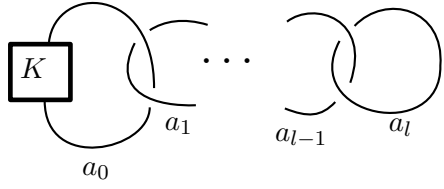
 \caption[An integer surgery diagram for $S^3_{p/q}(K)$]{A Kirby diagram for $W(K)$ and a surgery diagram for $Y \cong S^3_{p/q}(K)$.}
 \label{Ch1:fig:surgerytrace}
\end{figure}
Let $K \subseteq S^3$ be a knot. For fixed $p/q>0$, with a continued fraction $p/q= [a_0, \dots, a_l]^-$, where
\[
[a_0, \dots, a_l]^-
= a_0 -
    \cfrac{1}{a_1
        - \cfrac{1}{\ddots
            - \cfrac{1}{a_l} } },
\]
and the $a_i$ satisfy
\begin{equation}\label{Ch1:eq:aiconditions}
a_i\geq
\begin{cases}
1 & \text{for } i = 0 \text{ or }l\\
2 & \text{for } 0<i <l,
\end{cases}
\end{equation}
one can construct a 4-manifold $W=W(K)$ bounding $Y\cong S^3_{p/q}(K)$ by attaching 2-handles to $B^4$ according to the Kirby diagram given in Figure~\ref{Ch1:fig:surgerytrace}. Since $W$ is constructed by attaching 2-handles to a 0-handle, the first homology group $H_1(W)$ is trivial and $H_2(W)$ is free with a basis $\{h_0, \dots , h_l\}$ given by the 2-handles. With respect to this basis, the intersection form $Q_W\colon H_2(W) \times H_2(W) \rightarrow \mathbb{Z}$ is given by the matrix:
\[M =
  \begin{pmatrix}
   a_0  & -1   &        &       \\
   -1   & a_1  & -1     &       \\
        & -1   & \ddots & -1    \\
        &      & -1     & a_l
  \end{pmatrix}.\]
The conditions on the $a_i$ given by \eqref{Ch1:eq:aiconditions} imply that $M$ is positive-definite\footnote{For any $0\leq k \leq l$, the determinant of the top left $(k+1)\times(k+1)$-minor of $M$ is given by the numerator of the continued fraction $[a_0,\dots, a_k]^-$ and hence positive.} and hence that the intersection form on $W$ is positive-definite.
The corresponding $\mathbb{Q}$-valued pairing on $H^2(W) \cong {\rm Hom}(H_2(W),\mathbb{Z})$ is given by the matrix $M^{-1}$ with respect to the dual basis $\{h_0^*, \dots, h_l^*\}$. By considering the long exact sequence of the pair $(W, Y)$, we obtain the short exact sequence:
\[0 \rightarrow H^2(W, Y)\rightarrow H^2(W) \rightarrow H^2(Y) \rightarrow 0.\]
We may identify $H^2(W, Y)$ with $H_2(W)$ via Poincar\'{e} duality. With respect to the bases $\{h_i\}$ and $\{h_i^*\}$ the resulting map $H_2(W)\rightarrow H^2(W)$ is given by the matrix $M$. This gives an isomorphism
\begin{equation}\label{Ch1:eq:H2YcongM}
H^2(Y) \cong \frac{H^2(W)}{PD(H_2(W))} \cong {\rm coker}(M).
\end{equation}
Since $H^2(W)$ is torsion-free, the first Chern class defines an injective map
\[c_1:\spinc(W) \rightarrow H^2(W),\]
where the image is the set of characteristic covectors $\Char(W) \subseteq H^2(W)$. This allows us to identify $\spinc(W)$ with the set
\[\Char (W) = \{ (c_0, \dots , c_l)\in \mathbb{Z}^{l+1} \,|\, c_i \equiv a_i \bmod 2 \text{ for all } 0\leq i \leq l\}.\]
We will use this identification throughout this chapter.

\paragraph{}Using the map in \eqref{Ch1:eq:H2YcongM}, which arises from restriction, allows us to identify the set $\spinc(Y)$ with elements of the quotient
\[\frac{\Char(W)}{2PD(H_2(W))}.\]
Given $\spincs \in \Char(W)$ we will use $[\spincs]$ to denote its equivalence class modulo $2PD(H_2(W))$ and the corresponding \spinc-structure on $Y$.

\subsection{Representatives for ${{\rm Spin}^c}(S^3_{p/q}(K))$}\label{sec:representatives}
We will be interested in certain subsets of $\spinc(W)=\Char(W)$. In particular we wish to find a set of representatives for $\spinc(S^3_{p/q}(K))$ in $\spinc(W)$. Following Gibbons, we make the following definitions \cite{gibbons2013deficiency}.
\begin{defn}
Given $\spincs =(c_0, \dots, c_l)\in \Char(W)$, we say that it contains a {\em full tank} if there is $0\leq i <j\leq l$, such that $c_i=a_i$, $c_j=a_j$ and $c_k=a_k-2$ for all $i<k<j$.
We say that $\spincs$ is {\em left-full}, if there is $k>0$, such that $c_k=a_k$ and $c_j=a_j-2$ for all $0< j<k$.
\end{defn}
Observe that our definition of left-full does not impose any conditions on $c_0$, and that if $l=0$, then $\Char(W)$ contains no left-full elements.
\begin{defn} Let $\mathcal{M}$ denote the set of elements $\spincs =(c_0, \dots, c_l)\in \Char(W)$ satisfying $|c_i| \leq a_i,$ for all $0\leq i \leq l$, and such that neither $\spincs$ nor $-\spincs$ contain any full tanks. Let $\mathcal{C}\subseteq \mathcal{M}$ denote the set of elements $\spincs =(c_0, \dots, c_l)\in \mathcal{M}$ satisfying
\[2-a_i\leq c_i \leq a_i, \text{ for all } 0\leq i \leq l.\]
\end{defn}
The set $\mathcal{C}$ will turn out to form a complete set of representatives for $\spinc(Y)$.
\begin{lem}\label{Ch1:lem:countspincinC}
Write $p/q$ in the form $p/q=a_0-r/q$, where $q/r=[a_1,\dots, a_l]^-$. We have $|\mathcal{C}|=p$, and for each $c\equiv a_0 \bmod 2$, we have
\[|\{(c_0,\dots, c_l)\in \mathcal{C}\,|\, c_0=c\}|=
\begin{cases}
q   &\text{if $-a_0<c<a_0$}\\
q-r &\text{if $c=a_0$}
\end{cases}
\]
and
\[|\{\spincs=(c_0,\dots, c_l)\in \mathcal{C}\,|\, c_0=c \text{ and $\spincs$ is left full}\}|=
\begin{cases}
r   &\text{if $-a_0<c<a_0$}\\
0   &\text{if $c=a_0$.}
\end{cases}
\]
\end{lem}
\begin{proof}
We prove this by induction on the length of the continued fraction $[a_0, \dots , a_l]$. When $l=0$, we have $p=a_0$, $q=1$ and $r=0$. In this case,
\[\mathcal{C}=\{-a_0<c\leq a_0 \,|\, c \equiv a_0 \bmod 2\},\]
which clearly has the required properties. Suppose that $l>0$, and let $\mathcal{C'}$ denote the set
\[\mathcal{C'}=\{(c_1, \dots , c_l)\,|\, c_i \equiv a_i \bmod 2, -a_i<c_i\leq a_i, \text{for all $i$ and no full tanks}\}.\]
As $q/r=[a_1,\dots, a_l]^-$, we can assume that we have $|\mathcal{C'}|=q$, and for each $c\equiv a_1 \bmod 2$, we have
\[|\{(c_1,\dots, c_l)\in \mathcal{C'}\,|\, c_1=c\}|=
\begin{cases}
r   &\text{if $-a_1<c<a_1$}\\
r-r' &\text{if $c=a_1$}
\end{cases}
\]
and
\[|\{\spincs'=(c_1,\dots, c_l)\in \mathcal{C'} \,|\, c_1=c \text{ and $\spincs'$ is left full}\}|=
\begin{cases}
r'   &\text{if $-a_1<c<a_1$}\\
0   &\text{if $c=a_1$,}
\end{cases}
\]
where $r/r'=[a_2,\dots, a_l]^-$.
For $c\equiv a_0 \bmod 2$ in the range $-a_0< c\leq a_0$, take $\spincs=(c, c_1,\dots, c_l)$. If $c<a_0$, then $\spincs \in \mathcal{C}$ if and only if $(c_1,\dots, c_l)\in \mathcal{C'}$, and $\spincs$ is left-full if and only if $c_1=a_1$ or $c_1=a_1-2$ and $(c_1,\dots, c_l)\in \mathcal{C'}$ is left full. Therefore,
\[|\{(c_0,\dots, c_l)\in \mathcal{C} \,|\, c_0=c\}|=|\mathcal{C'}|=q\]
and
\[|\{\spincs=(c_0,\dots, c_l)\in \mathcal{C} \,|\, c_0=c \text{ and $\spincs$ is left full}\}|= (r-r')+r'=r,\]
for $c<a_0$. If $c=a_0$, then $\spincs \in \mathcal{C}$ if and only if $(c_1,\dots, c_l)\in \mathcal{C'}$ and $\spincs$ contains no full tanks. Equivalently, $\spincs \in \mathcal{C}$ if and only if $(c_1,\dots, c_l)\in \mathcal{C'}$ and $\spincs$ is not left-full. As above, we see that there are $r-r'$ choices of $\spincs'=(c_1,\dots, c_l)\in \mathcal{C'}$ with $c_1=a_1$ and $r'$ choices with $c_1=a_1-2$ and $\spincs'$ is left-full. This shows that
\[|\{(c_0,\dots, c_l)\in \mathcal{C} \,|\, c_0=a_0\}|=q-r\]
and
\[|\{\spincs=(c_0,\dots, c_l)\in \mathcal{C} \,|\, c_0=a_0 \text{ and $\spincs$ is left full}\}|=0,\]
as required. It follows that $|\mathcal{C}|=(a_0-1)q+q-r=p$, as required.
\end{proof}
\begin{defn}
We say that $\spincs \in \Char(W)$ is {\em short} if it satisfies $\norm{\spincs}\leq \norm{\spincs'}$ for all $\spincs' \in \Char(W)$ with $[\spincs']=[\spincs]$. Here $\norm{\spincs}=\spincs M^{-1} \spincs^T$ denotes the norm with respect to the pairing given by $M^{-1}$.
\end{defn}
 We will show that the set of short elements of $\Char(W)$ is precisely $\mathcal{M}$.

The following calculation will be useful in what follows. Take $\spincs=(c_0, \dots, c_l)\in \mathcal{M}$. In terms of the basis given by $\{h_i^*\}$, the cohomology class $PD(h_i)$ is given by the $i$th row of $M$, and we have
\begin{align}\begin{split}\label{Ch1:eq:pushdownnorm}
\norm{\spincs \pm 2PD(h_i)}&=\norm{\spincs} \pm 4PD(h_i)\cdot \spincs + 4\norm{PD(h_i)}\\
                           &= \norm{\spincs} \pm 4c_i + 4a_i.
\end{split}\end{align}
\begin{defn}
If $c_i= a_i$, then we define $\spincs'=\spincs - 2PD(h_i)$ to be the {\em push-down} of $\spincs$ at $a_i$. Similarly, if $c_i= -a_i$, then we define $\spincs'=\spincs + 2PD(h_i)$ to be the {\em push-up} of $\spincs$ at $a_i$.
\end{defn}
By \eqref{Ch1:eq:pushdownnorm}, we see that if $\spincs$ and $\spincs'$ are related by a push-up or push-down, then $\norm{\spincs}=\norm{\spincs'}$. We also have $[\spincs]=[\spincs']$. Note also that pushing up or pushing down cannot create a full tank, so $\spincs\in \mathcal{M}$ implies $\spincs'\in \mathcal{M}$.

\begin{lem}\label{Ch1:lem:fromMtoC}
For any $\spincs=(c_0, \dots, c_l)\in \mathcal{M}$, there is $\spincs'=(c_0', \dots, c_l') \in \mathcal{C}$, with $\norm{\spincs}=\norm{\spincs'}$ and $[\spincs]=[\spincs']$. If $c_0>-a_0$ and $-\spincs$ is not left-full, then $c_0'=c_0$ and $\spincs'$ is left-full if and only if $\spincs$ is left-full.
\end{lem}

\begin{proof}
Take $\spincs=(c_0, \dots, c_l)\in \mathcal{M}$. Observe that $\spincs\in \mathcal{C}$ if and only if $c_i>-a_i$ for all $i$. We define a sequence of elements in $\mathcal{M}$ as follows. First, set $\spincs_0=\spincs$. If $\spincs_{n}\not \in \mathcal{C}$ we set $\spincs_{n+1}=\spincs_{n}+2PD(h_k)$, where $k$ satisfies $c_k=-a_k$. Since $\spincs_{n+1}$ is obtained by a push-up from $\spincs_n$, we have $[\spincs_n]=[\spincs]$ and $\norm{\spincs_n}=\norm{\spincs}$ for all $n$. We will show that this sequence eventually results in an element of $\mathcal{C}$.

For any $\spincs'\in \mathcal{M}$ with $[\spincs']=[\spincs]$ we have $\spincs'-\spincs= \sum_{i=0}^l 2b_i PD(h_i)$ for some integers $b_i$. Since the $PD(h_i)$ are linearly independent, it makes sense to define
\[f(\spincs'):= \sum_{i=0}^l |b_i|.\]
As $\spincs_{n+1}$ is always obtained by adding $2PD(h_k)$ to $\spincs_{n}$ for some $k$, we must have $f(\spincs_{n})=n$ for all $n$. In particular, the $\spincs_n$ form a sequence of distinct elements in $\mathcal{M}$. As $\mathcal{M}$ is finite, the sequence must terminate with some $\spincs'\in \mathcal{C}$ satisfying $\norm{\spincs}=\norm{\spincs'}$ and $[\spincs]=[\spincs']$.

Now we verify the other stated properties of $\spincs'$. Suppose that $\spincs_{n+1}=(d_0', \dots, d_l')$ is obtained from $\spincs_{n}=(d_0, \dots, d_l)$ by pushing-up at some $k$. If $d_0>-a_0$ and $-\spincs_n$ is not left-full, then we necessarily have $k>1$. Thus $\spincs_{n+1}$ takes the form
\[(d_0', \dots, d_l')=(d_0, \dots ,d_{k-2}, d_{k-1}-2,a_k,d_{k+1}-2, d_{k+2}, \dots, d_l').\]
In particular $d_0=d_0'$ and $-\spincs_{n+1}$ is not left-full. Since we are pushing up at $k\geq 2$, $\spincs_{n+1}$ is left-full if and only if $\spincs_{n}$ is left-full. Thus if $-\spincs$ is not left-full and $c_0>-a_0$, then we can prove inductively that $\spincs'$ has the required properties.
\end{proof}

This allows us to show that $\mathcal{C}$ forms a complete set of representatives for $\spinc(Y)$ and to characterize the short elements of $\Char(W)$.
\begin{lem}\label{Ch1:lem:minimisers}
For every $\spinct\in \spinc(Y)$, there is a unique $\spincs \in \mathcal{C}$ with $[\spincs]=\spinct$. Any $\spincs \in \Char(W)$ is short if and only if $\spincs \in \mathcal{M}$.
\end{lem}
\begin{proof}
Consider $\spincs=(c_0, \dots, c_l)\in \spinc(W)$. It follows from \eqref{Ch1:eq:pushdownnorm}, that if $\pm c_i> a_i$, then $\norm{\spincs \mp 2PD(h_i)}=\norm{\spincs}\mp c_i+a_i<\norm{\spincs}$. Thus any short element must satisfy $|c_i|\leq a_i$ for all $i$. Now suppose $\spincs$ contains a full tank, say $c_j=a_j$, $c_i=a_i$ and $c_k=a_k-2$ for all $i<k<j$. Consider $\spincs'=\spincs - 2\sum_{k=i}^{j-1}PD(h_k)$, which can be obtained from $\spincs$ by a sequence of push-downs. If we write $\spincs'=(c_0', \dots, c_l')$, then $c'_j=a_j+2$, showing that $\spincs'$ is not short. Since $\norm{\spincs}=\norm{\spincs'}$, this shows $\norm{\spincs}$ is not short. Since $\spincs$ is short if and only if $-\spincs$ is short, this shows that $\spincs$ is short only if $\spincs\in \mathcal{M}$.

We prove the converse by a counting argument. Since every $\spinct\in \spinc(S^3_{p/q}(K))$ has a short representative $\spincs$, which is necessarily in $\mathcal{M}$, Lemma~\ref{Ch1:lem:fromMtoC} shows that it also has a short representative $\spincs' \in \mathcal{C}$. However Lemma~\ref{Ch1:lem:countspincinC} shows $|\spinc(Y)|=|\mathcal{C}|=p$, so for every $\spinct\in\spinc(Y)$, there is a unique $\spincs'\in \mathcal{C}$ with $[\spincs']=\spinct$ and $\spincs'$ is short. It then follows from Lemma~\ref{Ch1:lem:fromMtoC} that every $\spincs\in\mathcal{M}$ is short.
\end{proof}

\section{Calculating $d$-invariants of $S^3_{p/q}(K)$}\label{sec:calcdinvariants}
Now we use the representatives for $\spinc(S^3_{p/q}(K))$ studied in the previous section to calculate the $d$-invariants of $S^3_{p/q}(K)$.
Since the intersection form on $H^2(W)$ is independent of the choice of the knot $K$, it gives a choice of correspondences,
\[\spinc(W(K)) \leftrightarrow \Char(W(K)) \leftrightarrow \Char(W(U)) \leftrightarrow \spinc(W(U)),\]
and hence also a choice of correspondence
\begin{equation}\label{Ch1:eq:surgerytracespinccorrespond}
\spinc(S_{p/q}^3(K))\leftrightarrow \spinc(S_{p/q}^3(U)).
\end{equation}
Using this we can define $D^{p/q}_K:\spinc(S_{p/q}^3(K)) \rightarrow \mathbb{Q}$ by
\begin{equation*}
D^{p/q}_K(\spinct)=d(S_{p/q}^3(U),\spinct)-d(S_{p/q}^3(K),\spinct).
\end{equation*}

On the other hand, recall that there are identifications \cite{ozsvath2011rationalsurgery}
\begin{equation}\label{Ch1:eq:OzSzspinccorrespond}
\spinc(S_{p/q}^3(K)) \leftrightarrow \mathbb{Z}/p\mathbb{Z} \leftrightarrow \spinc(S_{p/q}^3(U)),
\end{equation}
for which the $d$-invariants of $S_{p/q}^3(K)$ satisfy the formula \cite{ni2010cosmetic}
\begin{equation}\label{Ch1:eq:NiWuformula1}
d(S_{p/q}^3(U),i)-d(S_{p/q}^3(K),i)=2V_{\min \{\lfloor \frac{i}{q} \rfloor,\lceil \frac{p-i}{q} \rceil\}},
\end{equation}
for $0\leq i \leq p-1$.

When $p/q=n$ is an integer and $W$ is obtained by attaching a single $n$-framed 2-handle to $B^4$ the correspondence \eqref{Ch1:eq:surgerytracespinccorrespond} can be easily reconciled with the one in \eqref{Ch1:eq:OzSzspinccorrespond}. In this case, the \spinc-structure $c\in\Char(W)=\{(c)\,|\, c\equiv n \bmod 2\}$, is labeled by $i \bmod n$, when $n+c \equiv 2i \bmod{2n}$ \cite{Ozsvath08integersurgery}. It is clear that in this case the correspondences in \eqref{Ch1:eq:surgerytracespinccorrespond} and \eqref{Ch1:eq:OzSzspinccorrespond} are the same. Hence for $c\equiv n \bmod 2$ satisfying $-n\leq c \leq n$, we have
\begin{equation}\label{Ch1:eq:integersurgDformula}
D^n_K([c])=d(S_{n}^3(U),[c])-d(S_{n}^3(K),[c])=2V_{\min \{\frac{n+c}{2},\frac{n-c}{2}\}}
=2V_{\frac{n-|c|}{2}}.
\end{equation}
\begin{rem}
It turns out that the correspondences between $\spinc(S_{p/q}^3(K))$ and $\spinc(S_{p/q}^3(U))$ used in \eqref{Ch1:eq:surgerytracespinccorrespond} and \eqref{Ch1:eq:OzSzspinccorrespond} coincide in general. However, we will evaluate $D^{p/q}_K([\spincs])$ without using this.
\end{rem}
Ozsv{\'a}th and Szab{\'o} have shown that the manifold $-W(U)$ is sharp \cite{Ozsvath03Absolutely, Ozsvath03plumbed} (or alternatively \cite{ozsvath2005heegaard}).
\begin{thm}[Ozsv{\'a}th-Szab{\'o}]\label{Ch1:thm:WUissharp}
The manifold $-W(U)$ is a sharp manifold bounding $S^3_{-p/q}(U)$.  In particular, for any $\spincs \in \mathcal{M}$,
\begin{equation}\label{Ch1:eq:lensspaced}
d(S_{p/q}^3(U), [\spincs])= \frac{\norm{\spincs}-b_2(W)}{4}=\frac{\norm{\spincs}-l-1}{4}.
\end{equation}
\end{thm}

By generalizing the ideas in \cite[Proof of Theorem~1.1]{gibbons2013deficiency}, this allows us to calculate $D^{p/q}_K([\spincs])$ for $\spincs \in \mathcal{M}$.
\begin{lem}\label{Ch1:lem:evalD}
For any $\spincs = (c_0, \dots, c_l)\in \mathcal{M}$, we have
\[D^{p/q}_K([\spincs])=
\begin{cases}
2V_{\frac{a_0-|c_0|-2}{2}}   &\text{if $\pm\spincs$ is left full and $\pm c_0\geq 0$,}\\
2V_{\frac{a_0-|c_0|}{2}}     &\text{otherwise.}
\end{cases}
\]
\end{lem}

\begin{proof}
Let $W'$ be the manifold obtained by attaching a single $a_0$ framed 2-handle to $B^4$ along $K$. Since $W$ is obtained by attaching 2-handles to $W'$, we can write $W=W' \cup Z$, where $\partial Z = -S_{a_0}^3(K) \cup S_{p/q}^3(K)$ and $b_2(Z)=l$. For any $\spincs \in \spinc(W)$, we have
\[c_1(\spincs)^2= c_1(\spincs|_{W'})^2+c_1(\spincs|_{Z})^2.\]
Thus for any $\spincs \in \mathcal{M}$, \eqref{Ch1:eq:lensspaced} shows that we have
\[
\frac{c_1(\spincs|_Z)^2 - l}{4}=
d(S_{p/q}^3(U),[\spincs])-d(S_{a_0}^3(U),[c_0]).
\]
However, from \eqref{Ch1:eq:dinvinequality} we have
\[
\frac{c_1(\spincs|_Z)^2 - l}{4}\geq d(S_{p/q}^3(K),[\spincs])-d(S_{a_0}^3(K),[c_0]).
\]
Combining these, we obtain
\begin{equation}\label{Ch1:eq:integralDbound}
D^{p/q}_K([\spincs])\geq D^{a_0}_K([c_0])=2V_{\frac{a_0-|c_0|}{2}},
\end{equation}
for any $\spincs \in \mathcal{M}$.

Now consider $\spincs \in \mathcal{C}$. If $\spincs$ is left-full, then it takes the form
\[\spincs=(c_0, a_1-2, \dots, a_{k-1}-2,a_k, c_{k+1},\dots,c_l),\]
for some $k$. Thus, by a sequence of push-downs performed at consecutively at $k, k-1, \dots, 2$ and $1$, we arrive at
\[\spincs'=(c_0', \dots, c_l')=\spincs-\sum_{i=1}^k PD(h_i)\in \mathcal{M},\]
which satisfies $[\spincs']=[\spincs]$ and $c_0'=c_0+2$. Thus by \eqref{Ch1:eq:integralDbound}, we have the bound
\begin{equation}\label{Ch1:eq:CintDbound}
D^{p/q}_K([\spincs])\geq
\begin{cases}
D^{a_0}_K([c_0+2])=2V_{\frac{a_0-|c_0|-2}{2}} &\text{if $c_0\geq 0$ and $\spincs$ is left-full}\\
D^{a_0}_K([c_0])=2V_{\frac{a_0-|c_0|}{2}} &\text{otherwise,}
\end{cases}
\end{equation}
for all $\spincs \in \mathcal{C}$. The following claim establishes the lemma for $\spincs \in \mathcal{C}$.
\begin{claim}
For all $\spincs \in \mathcal{C}$, we have equality in \eqref{Ch1:eq:CintDbound}.
\end{claim}
\begin{proof}[Proof of Claim]
Let $a_0=2k+\epsilon$ where $\epsilon\in \{0,1\}$.
Using the inequalities given in \eqref{Ch1:eq:CintDbound} and Lemma~\ref{Ch1:lem:countspincinC} to count the the elements of $\mathcal{C}$ which are left full or otherwise, we obtain
\begin{align}\begin{split}\label{Ch1:eq:sumofDs}
\sum_{\spinct \in \spinc(S_{p/q}^3(K))} D^{p/q}_K(\spinct)&\geq
\sum_{\substack{\spincs \in \mathcal{C},\\ c_0<0}} 2V_{\frac{a_0-|c_0|}{2}} + \sum_{\substack{\spincs \in \mathcal{C},\\ c_0\geq 0,\\ \text{$\spincs$ not left-full}}} 2V_{\frac{a_0-|c_0|}{2}} +\sum_{\substack{\spincs \in \mathcal{C},\\ c_0\geq 0,\\ \text{$\spincs$ left-full}}} 2V_{\frac{a_0-|c_0|-2}{2}}\\
&=2\epsilon q V_k + 2q\sum_{i=1}^{k-1}V_i + 2(q-r)\sum_{i=0}^{k}V_i + 2r\sum_{i=1}^{k}V_{i-1}\\
&=2(q+\epsilon q -r)V_k +2qV_0+4q\sum_{i=1}^{k-1}V_i.
\end{split}\end{align}
On the other hand, using \eqref{Ch1:eq:NiWuformula1}, we obtain
\begin{align*}
\sum_{\spinct \in \spinc(S_{p/q}^3(K))} D^{p/q}_K(\spinct)&=
\sum_{i=0}^{p-1}d(S_{p/q}^3(U),i)-d(S_{p/q}^3(K),i)\\
&=2\sum_{i=0}^{p-1}V_{\min \{\lfloor \frac{i}{q} \rfloor,\lceil \frac{p-i}{q} \rceil\}}\\
&=2(q+\epsilon q -r)V_k +2qV_0+4q\sum_{i=1}^{k-1}V_i.
\end{align*}
This shows that we have equality in \eqref{Ch1:eq:sumofDs}. Consequently we must also have equality in \eqref{Ch1:eq:CintDbound} for any $\spincs\in \mathcal{C}$.
\end{proof}
Now we deduce the lemma for arbitrary $\spincs=(c_0, \dots, c_l)\in \mathcal{M}$.

Suppose first that $|c_0|=a_0$. By \eqref{Ch1:eq:integralDbound}, we have $D_K^{p/q}([\spincs])\geq 2V_0$.
Since the $\mathcal{C}$ form a complete set of representatives for $\spinc(S_{p/q}^3(K))$, the above claim shows there is $i\geq 0$ such that
\[D_K^{p/q}([\spincs])=V_i.\]
Since $V_i \leq V_0$, it follows that $D_K^{p/q}([\spincs])= 2V_0$. As $\pm\spincs\in \mathcal{M}$ cannot be left-full if $c_0=\pm a_0$, this proves the lemma, when $a_0=|c_0|$.

We can now assume $|c_0|<a_0$. At most one of $\spincs$ or $-\spincs$ is left-full. If $-\spincs$ is not left-full, then Lemma~\ref{Ch1:lem:evalD} shows that there is $\spincs'=(c_0, c_1', \dots, c_l')\in \mathcal{C}$, with $[\spincs']=[\spincs]$ and $\spincs'$ is left-full if and only if $\spincs$ is left-full. This gives
\[D^{p/q}_K([\spincs])=D^{p/q}_K([\spincs'])=
\begin{cases}
2V_{\frac{a_0-|c_0|-2}{2}}   &\text{if $\spincs$ is left-full and $c_0\geq 0$,}\\
2V_{\frac{a_0-|c_0|}{2}}     &\text{otherwise,}
\end{cases}
\]
as required. Since $d$-invariants are invariant under conjugation and $-\spincs$ represents the conjugate of $\spincs$ under our identification $\spinc(W) \leftrightarrow \Char(W)$, we have
\begin{equation*}
D_{K}^{p/q}([\spincs])= D_{K}^{p/q}([-\spincs]).
\end{equation*}
Thus, if $-\spincs$ is left-full, then we get
\[D^{p/q}_K([\spincs])=D^{p/q}_K([-\spincs])=
\begin{cases}
2V_{\frac{a_0-|c_0|-2}{2}}   &\text{if $-c_0\geq 0$,}\\
2V_{\frac{a_0-|c_0|}{2}}     &\text{if $-c_0< 0$,}
\end{cases}
\]
completing the proof.
\end{proof}

\section{The Changemaker Theorem}
Let $W$ be the positive-definite manifold bounding $Y\cong S_{p/q}^3(K)$ obtained by attaching 2-handles $h_0$, \dots, $h_l$ to the 4-ball, as in Figure~\ref{Ch1:fig:surgerytrace}, according to the continued fraction expansion $p/q=[a_0, \dots, a_l]^-$, where $a_0\geq 1$ and $a_i\geq 2$ for $i\geq 1$ (i.e we are no longer allowing $a_l=1$). Note that a unique continued fraction of this form exists for all $p/q>0$. Suppose that $Y$ bounds a sharp negative-definite manifold $X$. Using this we can form the closed smooth manifold
\[Z:=W\cup_Y -X.\]

Since $W$ and $-X$ are positive-definite $Z$ is positive-definite with
\[b_2(Z)=B=b_2(X)+l+1.\]
By Donaldson's Diagonalization Theorem, the intersection form
\[Q_Z: H_2(Z)/{\rm Tors}\times H_2(Z)/{\rm Tors} \rightarrow \Z\]
is isomorphic to the diagonal form on $\Z^B$ \cite{donaldson87orientation}. From now on we will identify $H_2(Z)/{\rm Tors}$ with $\Z^B$ by choosing an orthonormal basis $\{e_1, \dots, e_B\}$. Under such an identification the first Chern class gives a surjective map
\[c_1\colon \spinc(Z)\rightarrow \Char(\Z^B).\]
Recall also that we have identified $\spinc(W)$ with the set
\[\Char (W) = \{(c_0, \dots , c_l)\in \mathbb{Z}^{l+1} \,|\, c_i \equiv a_i \bmod 2 \text{ for all } 0\leq i \leq l\}.\]
For each 2-handle $h_i$ of $W$, let $w_i\in \Z^B$ be the image of its homology class under the inclusion induced by the inclusion $W\subseteq Z$. Under these identifications, if $\spincs\in \spinc(Z)$ has $c_1(\spincs)=c\in \Char(\Z^B)$ then its restriction $\spincs|_W$ corresponds to
\[(c\cdot w_0, \dots, c\cdot w_l) \in \Char (W).\]
This following lemma is the natural generalization of \cite[Lemma~2.5]{greene2010space} to rational surgeries.
\begin{lem}\label{Ch1:lem:spincZ}
The following are true.
\begin{enumerate}[(i)]
\item For any $\spincs \in \mathcal{M}$, $D^{p/q}_K([\spincs])=0$ if and only if there is $c\in \{\pm 1\}^B$, such that \[\spincs=(c\cdot w_0, \dots, c\cdot w_l).\]
\item For all $i$ in the range $0\leq i \leq a_0/2$, the vector $w_0$ satisfies
\begin{equation}\label{Ch1:eq:w0formula}
8V_{i} = \min_{ \substack{|c\cdot w_0| =a_0-2i \\ c \in \Char(\mathbb{Z}^{B})}} \norm{c} - B.
\end{equation}
\end{enumerate}
\end{lem}
\begin{proof}
Observe that for any $\mathfrak{r} \in \spinc(Z)$, we have
$c_1(\mathfrak{r})^2=c_1(\mathfrak{r}|_W)^2 + c_1(\mathfrak{r}|_{-X})^2$. Thus if $\mathfrak{r}$ to restricts $\spincs=\mathfrak{r}|_{W}$ on $W$, then \eqref{Ch1:eq:dinvinequality} and Theorem~\ref{Ch1:thm:WUissharp} show
\begin{align}\label{Ch1:eq:ZDbound}
\begin{split}
c_1(\mathfrak{r})^2 -b_2(Z)&= c_1(\spincs)^2 - b_2(W) + c_1(\mathfrak{r}|_{-X})^2 - b_2(X)\\
    &\geq 4d(S_{p/q}^3(U), [\spincs]) - 4d(Y, [\spincs])\\
        &=4D^{p/q}_K([\spincs]).
\end{split}
\end{align}
By Lemma~\ref{Ch1:lem:minimisers}, equality can occur in \eqref{Ch1:eq:ZDbound} only if $\spincs \in \mathcal{M}$.

Now as $X$ is sharp, for any $\spincs\in \mathcal{M}$, there is $\spincs' \in \spinc(-X)$ with $\spincs'|_{-Y}=[\spincs]$ and
\[-4d(Y,[\spincs])=c_1(\spincs')^2 - b_2(X).\]
In particular, if we form $\mathfrak{r}\in \spinc(Z)$ by gluing $\spincs\in \mathcal{M}$ to $\spincs'$, then $\mathfrak{r}$ realizes the lower bound in \eqref{Ch1:eq:ZDbound}. This shows that for any $\spincs\in \mathcal{M}$, we have
\begin{equation*}
\min_{\substack{\mathfrak{r} \in \spinc(Z),\\ \mathfrak{r}|_W=\spincs}}c_1(\mathfrak{r})^2 -b_2(Z) = 4D^{p/q}_K([\spincs]).
\end{equation*}
Under our identification of $Q_Z$ with the diagonal form this give
\begin{equation}\label{Ch1:eq:spincZ}
\min_{\substack{c \in \Char(\Z^B),\\ \spincs=(c\cdot w_0, \dots, c\cdot w_l)}}\norm{c} - B = 4D^{p/q}_K([\spincs]),
\quad\text{for any $\spincs \in \mathcal{M}$.}
\end{equation}
Since any $c\in \Char(\Z^B)$ satisfies $\norm{c}\geq B$ with equality if and only if $c\in\{\pm 1\}^B$, statement $(i)$ follows from \eqref{Ch1:eq:spincZ}.

Now fix $i$ in the range $0\leq i \leq a_0/2$. From \eqref{Ch1:eq:ZDbound} we see that any $c\in \Char(\Z^B)$ with $(c\cdot w_0, \dots, c\cdot w_l)=\spincs\in \spinc(W)$ and $|c\cdot w_0| =a_0-2i$ satisfies
\[\norm{c}-B \geq 4D^{p/q}_K([\spincs]),\]
with equality only if $\spincs \in \mathcal{M}$.
This shows that
\begin{equation}\label{Ch1:eq:w0lowerboundformula}
\min_{ \substack{|c\cdot w_0| =a_0-2i \\ c \in \Char(\mathbb{Z}^{B})}} \norm{c} - B=
\min_{ \substack{|c\cdot w_0| =a_0-2i \\
 (c\cdot w_0, \dots, c\cdot w_l) \in \mathcal{M}\\
 c \in \Char(\mathbb{Z}^{B})}} \norm{c} - B
 \geq 8V_i,
\end{equation}
where the final inequality follows from Lemma~\ref{Ch1:lem:evalD} and \eqref{Ch1:eq:spincZ}.
 However, Lemma~\ref{Ch1:lem:evalD} also shows that $\spincs'=(a_0-2i, a_1-2, \dots, a_l-2)\in \mathcal{M}$, satisfies $D^{p/q}_K([\spincs'])=V_i$. This shows the lower bound in \eqref{Ch1:eq:w0lowerboundformula} is attained, proving statement $(ii)$.
\end{proof}
The following lemma is Gibbon's result \cite[Theorem~1.2]{gibbons2013deficiency} with extra hypotheses on the $d$-invariants of $Y$ removed. The proof given is a streamlined version of the one found by Gibbons.
\begin{lem}\label{Ch1:lem:wigiveCMlattice}
The vectors $w_0, \dots, w_l\in \Z^B$ satisfy Conditions I, II and III in the definition of a $p/q$-changemaker lattice and we have the bound
\begin{equation}\label{Ch1:eq:sharpnuplusbound}
\frac{p}{q}> 2\nu^{+}(K).
\end{equation}
\end{lem}
\begin{proof}
Let $e_1, \dots, e_B$ be an orthonormal basis for $\Z^B$. For $0\leq i\leq l$ and $1\leq j \leq B$, set $w_{i,j}=w_i\cdot e_j$. Without loss of generality, we will assume that this basis has been chosen so that $w_{0,j}\geq 0$ for all $j$. It will be useful to consider the quantity
\[S=\sum_{i=1}^B w_{0,i}.\]
As in Definition~\ref{Intro:def:CMlattice}, we will consider the sets
\[I_i=\{k\,|\, w_{i,k}\ne 0\}.\]
We break the proof of Lemma~\ref{Ch1:lem:wigiveCMlattice} into a number of steps.
\begin{step}\label{Ch1:Step:Vivanishing}
We have $V_i=0$ if and only if $2i\geq a_0-S$. In particular, we have
\begin{equation*}
2\nu^{+}(K)=a_0-S
\end{equation*}
and \eqref{Ch1:eq:sharpnuplusbound} holds.
\end{step}
\begin{proof}
It follows from Lemma~\ref{Ch1:lem:spincZ} that for $i$ in the range $0\leq i \leq a_0/2$, we have
$V_{i}=0$ if and only if there is $c\in \{\pm1\}^B$, such that $|c\cdot w_0| =a_0-2i$. The maximal value that $|c\cdot w_0|$ can attain for any $c\in \{\pm 1\}^B$ is $|c\cdot w_0|=S$ -- obtained by choosing $\pm c=(1,\dots,1)$. This shows $V_i=0$ if and only if $2i\geq a_0-S$. By definition, $\nu^+(K)$ is the minimal integer such that $V_{\nu^+(K)}=0$, so $2\nu^{+}(K)=a_0-S$. As $a_0\geq 1$, we must have $S\geq 1$, so
\[\frac{p}{q}>a_0-1 =2\nu^+(K)+S-1\geq 2\nu^{+}(K),\]
proving \eqref{Ch1:eq:sharpnuplusbound}.
\end{proof}

For the remainder of the proof we set $\epsilon_j =2-a_j$.

\begin{step}\label{Ch1:Step:wijvalues}
For all $i\geq 1$ and all $j$, we have $|w_{i,j}|\leq 1$. That is, Condition II of Definition~\ref{Intro:def:CMlattice} is fulfilled.
\end{step}
\begin{proof}
Consider $\spincs\in\mathcal{M}$ given by
\[\spincs=(-S,\epsilon_1, \dots, \epsilon_{i-1},a_{i},\epsilon_{i+1}, \dots, \epsilon_l).\]
By Lemma~\ref{Ch1:lem:evalD} and Step~\ref{Ch1:Step:Vivanishing}, we have $D^{p/q}_K([\spincs])=0$. So by Lemma~\ref{Ch1:lem:spincZ}$(i)$ there is $c=(c_1, \dots, c_B)\in \{\pm 1\}^B$ such that
\[0=a_i-c\cdot w_i=\sum_{j=1}^B w_{i,j}(w_{i,j}- c_j).\]
As $|c_j|=1$, we have $w_{i,j}(w_{i,j}- c_j)\geq 0$ for each $j$. Consequently, we have \[w_{i,j}(w_{i,j}- c_j)= 0\]
for all $j$. However this can only be achieved if $w_{i,j}=c_j$ or $w_{i,j}=0$. In either case $|w_{i,j}|\leq 1$ as required.
\end{proof}
\begin{step}\label{Ch1:Step:overlap}
For $0\leq i < j\leq l$ we have
\[w_{i}\cdot w_{j}= -|I_i \cap I_j|=
\begin{cases}
-1 &\text{if $j=i+1$}\\
0 &\text{if $j>i+1$}.
 \end{cases}\]
That is, Condition III of Definition~\ref{Intro:def:CMlattice} is fulfilled.
\end{step}
\begin{proof}
For $k\geq 1$, it will be useful to decompose $I_i$ as the disjoint union $I_k=I_k^+ \cup I_k^-$, where
\[I_k^{\pm}=\{l\,|\, w_{k,l}=\pm 1\}.\]
We prove this step separately for $i=0$ and $i>0$.

First assume that $i=0$. Consider
\[\spincs=(-S,\epsilon_1, \dots, \epsilon_{j-1}, a_{j},\epsilon_{j+1}, \dots, \epsilon_l)\in \mathcal{M}.\]
By Lemma~\ref{Ch1:lem:evalD} and Step~\ref{Ch1:Step:Vivanishing}, we have $D^{p/q}_K([\spincs])=0$. So by Lemma~\ref{Ch1:lem:spincZ}$(i)$,  there is $c\in \{\pm 1\}^B$ such that $c\cdot w_0=-S$ and $c\cdot w_j=a_j$. If $k\in I_0$, then $c_k=-1$ and if $k\in I_j^{\pm}$, then $c_k=\pm 1$. Thus $I_0$ and $I_j^+$ are necessarily disjoint.
Thus we can compute
\[w_0 \cdot w_j = \sum_{k\in I_j^+\cap I_0}w_{0,k} - \sum_{k\in I_j^-\cap I_0}w_{0,k}= - \sum_{k\in I_j^-\cap I_0}w_{0,k}= - \sum_{k\in I_j\cap I_0}w_{0,k}.\]
Since we are assuming that $w_{0,k}\geq 1$ for all $k\in I_0$, it follows that $w_0\cdot w_j=0$ if and only if $|I_0\cap I_k|=0$, and that $w_0\cdot w_j=-1$ only if $|I_0\cap I_j|=1$.  Since $w_0\cdot w_j \in \{0,-1\}$, it follows that $w_0\cdot w_j=-|I_0 \cap I_j|$, as required.

Now assume that $i\geq 1$. Consider
\[\spincs=(-S,\epsilon_0, \dots, \epsilon_{i-1}, a_{i},\epsilon_{i+1}, \dots,\epsilon_{j-1}, -a_{j},\epsilon_{j+1},\dots, \epsilon_l)\in \mathcal{M}.\]
By Lemma~\ref{Ch1:lem:evalD} and Step~\ref{Ch1:Step:Vivanishing}, we have $D^{p/q}_K([\spincs])=0$.
So by Lemma~\ref{Ch1:lem:spincZ}$(i)$, there is $c\in \{\pm 1\}^B$ such that $c\cdot w_i=a_i$ and $c\cdot w_j=-a_j$.
Observe that for such a $c$, $k\in I_i^{\pm}$ implies $c_k=\pm 1$ and $k\in I_j^{\pm}$ implies $c_k=\mp 1$. Thus we have $|I_i^+ \cap I_j^+|= |I_i^- \cap I_j^-|=0$. Since
\[w_i\cdot w_j = |I_i^+ \cap I_j^+|+ |I_i^- \cap I_j^-|-|I_i^+ \cap I_j^-|- |I_i^- \cap I_j^+|,\]
this shows that
\[w_i\cdot w_j =-|I_i^+ \cap I_j^-|- |I_i^- \cap I_j^+|=-|I_i \cap I_j|,\]
as required.
\end{proof}
After possibly reordering the $e_i$, we may assume that $w_0$ takes the form
\begin{equation}\label{Ch1:eq:w0form}
w_0=
\begin{cases}
\sigma_1 e_1 + \dots + \sigma_t e_t &\text{if $l=0$}\\
e_{t+1} + \sigma_1 e_1 + \dots + \sigma_t e_t &\text{if $l\geq 1$,}
\end{cases}
\end{equation}
where $t+1$ is the unique element of $I_1 \cap I_0$ given by Step~\ref{Ch1:Step:overlap} when $l\geq 1$.

\begin{step}\label{Ch1:Step:sigmaCMconditions}
The tuple $(\sigma_1, \dots, \sigma_t)$ of elements appearing in \eqref{Ch1:eq:w0form} satisfies the changemaker condition. That is, Condition~I of Definition~\ref{Intro:def:CMlattice} is fulfilled.
\end{step}
\begin{proof}
We deal with the cases $l=0$ and $l\geq 1$ separately.

First we assume that $l=0$, i.e that $p/q=a_0\in \Z$. In this case we have $S=\sum_{i=1}^t \sigma_i$. It follows from Step~\ref{Ch1:Step:Vivanishing} and Lemma~\ref{Ch1:lem:evalD}, that for all $c_0\equiv a_0 \bmod{2}$ in the range $-S\leq c_0\leq S$, $D^{p/q}_K([c_0])=0$. Thus, there is $c\in \{\pm 1\}^B$ such that $c\cdot w_0=c_0$. If we write $c_0=-S+2k$, where $0\leq k \leq S=\sum_{i=1}^t \sigma_i$, and $c$ in the form
\[c=(2d_1-1,\dots, 2d_B-1),\]
where $d_j \in \{0,1\}$ for all $j$, then we have
\[c\cdot w_0 = \sum_{i=1}^t (2\sigma_i d_i - \sigma_i)= 2k-S.\]
and hence that
\[k=\sum_{i=1}^t d_i\sigma_i \quad \text{for some $d_i\in \{0,1\}$.}\]
Since $k$ was an arbitrary value in the range $0\leq k \leq \sum_{i=1}^t \sigma_i$ it follows that $(\sigma_1, \dots, \sigma_t)$ satisfies the changemaker condition.

Now assume that $l\geq 1$. In this case, we have $S=1+\sum_{i=1}^t  \sigma_i$. For $c_0\equiv a_0 \bmod{2}$ in the range $2-S \leq c_0 \leq S$, consider the \spinc-structure
\[\spincs=(c_0,-a_1, \epsilon_2,\dots, \epsilon_l)\in \mathcal{M}.\]
Noting that $-\spincs$ is left-full, Step~\ref{Ch1:Step:Vivanishing} and Lemma~\ref{Ch1:lem:evalD} shows that $D^{p/q}_K([\spincs])=0$. Therefore there is $c\in \{\pm 1\}^B$ such that $c \cdot w_0 = c_0$ and $c\cdot w_1=-a_1$. By Step~\ref{Ch1:Step:overlap} and the assumption that $w_0\cdot e_{t+1}=1$, we must have $w_1\cdot e_{t+1}=-1$. Thus by Step~\ref{Ch1:Step:wijvalues}, we must have $c\cdot e_{t+1}=1$. This means that if we write $c_0=2-S + 2k$, where $0\leq k \leq S-1 = \sum_{i=1}^t \sigma_i$, and $c$ in the form
\[c=(2d_1-1,\dots, 2d_B-1),\]
where $d_j \in \{0,1\}$ for all $j$, then we have
\[c\cdot w_0 = 1+\sum_{i=1}^t (2\sigma_i d_i -\sigma_i)= 2-S+2k=2k+1-\sum_{i=1}^t \sigma_i.\]
This shows that
\[k=\sum_{i=1}^t d_i \sigma_i, \quad \text{for some $d_i\in \{0,1\}$.}\]
Since $k$ was an arbitrary value in the range $0\leq k \leq \sum_{i=1}^t \sigma_i$ it follows that $(\sigma_1, \dots, \sigma_t)$ satisfies the changemaker condition.
\end{proof}
This completes the proof of the lemma.
\end{proof}

\begin{lem}\label{Ch1:lem:QXisom}
We have an isomorphism
\[Q_{-X}\cong \langle w_0, \dots, w_l \rangle^\bot \subseteq \Z^B.\]
\end{lem}
\begin{proof}
By considering the Mayer-Vietoris sequence for $Z=W\cup_Y -X$, then we obtain the exact sequence
\begin{equation*}
H_2(Y)=0 \rightarrow H_2(W)\oplus H_2(-X)\rightarrow H_2(Z).
\end{equation*}
This shows that there is an embedding
\begin{equation}\label{Ch1:eq:inducedembedding}
Q_{-X}\hookrightarrow L'=\langle w_0, \dots, w_l \rangle^\bot \subseteq \Z^B.
\end{equation}
The following two claims show that $Q_{-X}$ and $L'$ have the same discriminant. Since $L'$ and $Q_{-X}$ have the same rank, this shows that map in \eqref{Ch1:eq:inducedembedding} is an isomorphism, as required.
\begin{claim}
$\disc(Q_X)=|H^2(Y)|=p$.
\end{claim}
\begin{proof}[Proof of Claim]
Since the restriction map $\spinc(X)\rightarrow \spinc(Y)$ is surjective, the map $H^2(X)\rightarrow H^2(Y)$ induced by inclusion is surjective. By considering the long exact sequence for cohomology of the pair $(X,Y)$, we obtain exact sequences
\begin{equation}\label{Ch1:eq:discriminantses}
0\rightarrow H^2(X,Y)\rightarrow H^2(X) \rightarrow H^2(Y) \rightarrow 0
\end{equation}
and
\[0\rightarrow H^3(X,Y)\rightarrow H^3(X).\]
By combining these two sequences with the universal coefficients theorem and Poincar\'{e} duality, we see that the map in \eqref{Ch1:eq:discriminantses} must give an isomorphism on the torsion subgroups $\tors(H^2(X,Y))$ and $\tors(H^2(X))$. Therefore \eqref{Ch1:eq:discriminantses} gives the the exact sequence
\[0\rightarrow \frac{H^2(X,Y)}{{\rm Tors}}\rightarrow \frac{H^2(X)}{{\rm Tors}} \rightarrow H^2(Y) \rightarrow 0.\]
If we identify $H^2(X,Y)$ with $H_2(X)$ via Poincar\'{e} duality, then this first map is precisely the inclusion of $\frac{H_2(X)}{{\rm Tors}}$ into its dual lattice, giving the desired value of the discriminant.
\end{proof}

\begin{claim}
We have $\disc(L')=p$
\end{claim}
\begin{proof}[Proof of Claim]
This argument is adapted from \cite[Lemma~3.10]{GreeneLRP}. Define the homomorphism $\phi\colon \Z^{B} \rightarrow \Z^{l+1} / M \Z^{l+1}$ given by
\[\phi(x) = (x\cdot w_0, \dots, x\cdot w_l) \bmod M,\]
where $M$ is the matrix
\[M =
  \begin{pmatrix}
   a_0  & -1   &        &       \\
   -1   & a_1  & -1     &       \\
        & -1   & \ddots & -1    \\
        &      & -1     & a_l
  \end{pmatrix}.\]
The group $\Z^{l+1} / M \Z^{l+1} \cong H_1(Y)$ is cyclic of order $p$.
Observe that $\phi$ is surjective. If $p/q=1$, then this is trivially true. If $q>1$, then Lemma~\ref{Ch1:lem:wigiveCMlattice} shows that there is unit vector $f\in \Z^B$ such that $w_l\cdot f=1$ and $w_i\cdot f=0$ for $i<l$. Such an $f$ satisfies $\phi(f)=(0,\dots,0,1)$, which is a generator. Similarly, if $a_0>1$, then the set of changemaker coefficients is non-trivial so there is $f$ with such that $\phi(f)=(1,0,\dots, 0)$. This shows that $\disc(\ker(\phi))=|\Z^{l+1} / M \Z^{l+1}|^2=p^2$. However, the kernel of $\phi$ splits as $\ker(\phi)=L' \oplus \langle w_0, \dots, w_l \rangle$. Since $\disc (\langle w_0, \dots, w_l \rangle)=\det(M)=p$, this shows that $\disc(L')= \disc(\ker(\phi))/p=p$, as required.
\end{proof}
\end{proof}

\begin{rem}
In fact, this argument shows that any $p/q$-changemaker lattice has discriminant $p$.
\end{rem}
We now piece together the results from this section to prove Theorem~\ref{intro:thm:CM}.
\begin{proof}[Proof of Theorem~\ref{intro:thm:CM}]
Suppose that $S_{p/q}^3(K)$ bounds a sharp manifold $X$. Together Lemma~\ref{Ch1:lem:wigiveCMlattice} and Lemma~\ref{Ch1:lem:QXisom} show that there are vectors $w_0, \dots, w_l\in \Z^B$, such that
\[-Q_{X}\cong Q_{-X} \cong \langle w_0, \dots, w_l\rangle^\bot \subseteq \Z^B,\]
$w_0$ satisfies \eqref{Ch1:eq:w0formula} and the $w_i$ satisfy conditions I, II and III in Definition~\ref{Intro:def:CMlattice}. If we take $e_1,\dots, e_B$ as an orthonormal basis for $\Z^B$, then without loss of generality, we can assume that
\[\{e_i \,|\, w_j \cdot e_i \ne 0 \text{ for some $j$}\}=\{e_1, \dots, e_{N}\},\]
for some $N\leq B$. Thus, we see that $w_0, \dots, w_l$ satisfy condition IV of Definition~\ref{Intro:def:CMlattice} when considered as vectors in $\langle e_1, \dots, e_N \rangle\cong \Z^N$. This means that
 \[L=\langle w_0, \dots, w_l \rangle^\bot \subseteq \langle e_{1}, \dots, e_{N} \rangle\]
is a changemaker lattice. Therefore, we have
\[-Q_{X}\cong L \oplus \langle e_{N+1}, \dots, e_B \rangle\cong L\oplus \Z^{S},\]
where $S=B-N$, as required.

Finally, we verify \eqref{Ch1:eq:w0formula2}. Observe that if $c=(c_1, \dots, c_B) \in \Char(\Z^B)$ satisfies
\[c\cdot w_0 = n-2i \quad \text{and} \quad V_i=\norm{c}-B,\]
for some $0\leq i\leq n/2$, then $c_j\in \{\pm 1\}$ for all $j>t$. Thus if we take $c'=(c_1, \dots, c_N)\in \Char(\Z^N)$, then
\[c'\cdot w_0 = n-2i \quad \text{and} \quad V_i=\norm{c'}-t.\]
Thus \eqref{Ch1:eq:w0formula} is sufficient to imply \eqref{Ch1:eq:w0formula2}. This completes the proof.
\end{proof}

\section{Sharp cobordisms}
We have previously discussed sharp manifolds. It will be useful to make the related definition of a sharp cobordism. Let $W\colon Y' \rightarrow Y$ be a negative-definite cobordism between oriented rational homology 3-spheres. That is $W$ is a manifold with boundary $-Y' \cup Y$. By \eqref{Ch1:eq:dinvinequality}, we have
\[c_1(\spincs)^2+b_2(W)\geq 4d(Y,\spincs|_Y)-4d(Y',\spincs|_{Y'})\]
for all $\spincs \in \spinc(W)$.
\begin{defn}
We say that the cobordism $W$ is {\em sharp} if for all $\spinct \in \spinc(Y)$, there is $\spincs\in \spinc(W)$ such that $\spincs|_Y = \spinct$ and
\[c_1(\spincs)^2+b_2(W)= 4d(Y,\spinct)-4d(Y',\spincs|_{Y'}).\]
\end{defn}
\begin{rem}
Since $d(S^3,\spinct)=0$, where $\spinct$ is the unique \spinc-structure on $S^3$, we see that a manifold $X$ is sharp if and only if the manifold obtained by removing an open ball from $X$ is a sharp cobordism from $S^3$ to $\partial X$.
\end{rem}
The composition of two sharp cobordism is again a sharp cobordism.
\begin{lem}\label{Ch1:lem:sharpcomposition}
Let $W\colon Y' \rightarrow Y$ and $W'\colon Y'' \rightarrow Y'$ be sharp cobordisms. Then the composition $W'\cup_{Y'} W\colon Y'' \rightarrow Y$ is also sharp. In particular, if $Y'$ bounds a sharp manifold $X$ and there is a sharp cobordism $W\colon Y' \rightarrow Y$, then $X\cup_{Y'} W$ is a sharp manifold with boundary $Y$.
\end{lem}
\begin{proof}
Observe that the composition $W'\cup_{Y'} W$ is also negative-definite with $b_2(W'\cup_{Y'} W)=b_2(W)+b_2(W')$. Since $W$ is sharp, for every $\spinct\in\spinc(Y)$ there is $\spincs \in \spinc(W)$ such that $\spincs|_Y=\spinct$ and $c_1(\spincs)^2+b_2(W)= 4d(Y,\spinct)-4d(Y',\spincs|_{Y'})$. Since $W'$ is sharp, we can find $\spincs' \in \spinc(W')$ such that $\spincs'|_{Y'}=\spincs|_{Y'}$ and
$c_1(\spincs')^2+b_2(W')= 4d(Y',\spincs|_{Y'})-4d(Y'',\spincs'|_{Y''})$. Now consider the \spinc-structure $\mathfrak{r}\in \spinc(W\cup_{Y'} W')$ obtained by gluing $\spincs$ and $\spincs'$. By construction $\mathfrak{r}|_Y=\spinct$ and
\begin{align*}
c_1(\mathfrak{r})^2 + b_2(W\cup_{Y'} W') &=c_1(\spincs)^2 + c_1(\spincs')^2+b_2(W)+b_2(W')\\
&=4d(Y,\spinct)-4d(Y',\spincs|_{Y'}) + 4d(Y',\spincs|_{Y'})-4d(Y'',\spincs'|_{Y''})\\
&=4d(Y,\spinct)-4d(Y'',\spincs'|_{Y''}).
\end{align*}
This shows that $W'\cup_{Y'} W$ is sharp.
\end{proof}

\subsection{Sharp cobordisms between surgeries}
Let $p/q=[a_0, \dots , a_l]^-$ and $p'/q'=[a_0, \dots, a_l, b_1, \dots, b_k]^-$, where $a_0,b_k\geq 1$ and $a_i,b_j \geq 2$ for all $1\leq i \leq l$ and $1\leq j <k$. Now let $W$ and $W'$ be the 4-manifolds with boundaries $Y\cong S_{p/q}^3(K)$ and $Y'\cong S_{p'/q'}^3(K)$ obtained by attaching 2-handles according to the two given continued fractions as in Figure~\ref{Ch1:fig:surgerytrace}. The manifold $W$ is naturally included as a submanifold in $W'$ and $Z=W' \setminus {\rm int \,}W$ is a positive-definite manifold with boundary $-S_{p/q}^3(K)\cup S_{p'/q'}^3(K)$. As before, we may take a basis for the homology groups $H_2(W)$ and $H_2(W')$ given by the 2-handles and in the same way we may identify $\spinc(W)$ and $\spinc(W')$ with $\Char(H_2(W))$ and $\Char(H_2(W'))$ respectively. We can also define subsets $\mathcal{C}\subseteq \mathcal{M} \subseteq \Char(H_2(W))$ and $\mathcal{C'} \subseteq \mathcal{M'}\subseteq \Char(H_2(W'))$, as in Section~\ref{sec:representatives}.

\begin{lem}\label{Ch1:lem:sharpsurgcobord}
If $p'/q'$ is not an odd integer or $p'/q'> 2\nu^{+}(K)-1$, then
\[-Z\colon S_{p'/q'}^3(K)\rightarrow S_{p/q}^3(K)\]
is a sharp cobordism. 
\end{lem}
\begin{proof}
For any $\spinct \in \spinc(Y)$, we may choose $\spincs =(c_0,\dots, c_l)\in \mathcal{C}$ such that $[\spincs]=\spinct$. Let $\epsilon_j =
\begin{cases}
1 &\text{if $b_j$ is odd}\\
0  &\text{if $b_j$ is even,}
\end{cases}$
and consider
\[\spincs'_{\pm}=(c_0,\dots,c_l,\epsilon_1,\dots,\epsilon_{k-1},\pm \epsilon_k).\]
We will make use of the following claim.
\begin{claim}
If $p'/q'$ is not an odd integer or $p'/q'> 2\nu^{+}(K)-1$, then for at least one $\square \in \{+,-\}$, we have $\spincs'_{\square}\in \mathcal{M}'$ and
$D^{p/q}_K(\spinct)=D^{p'/q'}_K([\spincs'_{\square}])$.
\end{claim}

\begin{proof}[Proof of Claim]
Since $\spincs \in \mathcal{C}$, we always have $\spincs'_{-} \in \mathcal{M}'$ and $\spincs'_{-}$ is left-full if and only if $\spincs$ is left-full. Thus Lemma~\ref{Ch1:lem:evalD} shows $D^{p/q}_K(\spinct)= D^{p'/q'}_K([\spincs'_{-}])$ whenever $c_0>0$ or $-\spincs'_{-}$ is not left-full. So the claim is true unless $c_0\leq 0$ and $-\spincs'_{-}$ is left-full.

As $-\spincs'_{-}$ is left full only if $b_k=1$, $b_j=2$ for all $1\leq j <k$ and $c_i=2-a_i$ for all $1\leq i \leq l$, we now need to prove the claim under the assumptions that $(b_1, \dots, b_k)=(2,\dots, 2,1)$ and $\spincs=(c_0, 2-a_1, \dots, 2-a_l)$, where $c_0\leq 0$. In this case, we have $\spincs_{+}'\in \mathcal{M}'$ and $\spincs_{+}'$ is left full if and only if $a_i=2$ for all $1\leq i\leq l$. Thus Lemma~\ref{Ch1:lem:evalD} shows $D^{p/q}_K(\spinct)= D^{p'/q'}_K([\spincs'_{+}])$ except when $\spincs=(0, \dots,0)$.

It remains only to prove the claim when $\spincs=(0, \dots,0)$ and $p'/q'=[a_0, 2, \dots, 2,1]^-=a_0-1$. In this case, $a_0$ must be even and we have
\[D^{p/q}_K(\spinct)=2V_{\frac{a_0}{2}} \text{ and } D^{p'/q'}_K([\spincs'_{+}])= 2V_{\frac{a_0-2}{2}}.\]
By assumption, we have $p'/q'=a_0-1> 2\nu^{+}(K)-1$. As $a_0$ is even, this implies $a_0\geq 2\nu_+ +2$, and hence that $V_{\frac{a_0-2}{2}}=0$. This shows that $D^{p/q}_K(\spinct)=D^{p'/q'}_K([\spincs'_{+}])=0$, concluding the proof of the claim.
\end{proof}

It follows from the above claim, that there is $\spincs' \in\{\spincs'_+,\spincs'_-\} \cap \mathcal{M}' \subseteq \spinc(W')$ such that $D^{p/q}_K(\spinct)=D^{p'/q'}_K([\spincs'])$. Let $\mathfrak{r}$ denote the restriction of $\spincs'$ to $Z$. Since $c_1(\spincs')^2 = c_1(\spincs)^2 + c_1(\mathfrak{r})^2$,
it follows from \eqref{Ch1:eq:lensspaced}, that
\begin{align*}
\frac{c_1(\mathfrak{r})^2-b_2(Z)}{4}&=d(S_{p'/q'}^3(U),[\spincs'])-d(S_{p/q}^3(U),\spinct)\\
    &=(d(Y',[\spincs'])+D^{p'/q'}_K([\spincs']))-(d(Y,\spinct)+D^{p/q}_K(\spinct))\\
    &=d(Y',[\spincs'])-d(Y,\spinct).
\end{align*}
This implies that $-Z$ is a sharp cobordism.
\end{proof}
\begin{rem}
In fact, the proof of Lemma~\ref{Ch1:lem:sharpsurgcobord} shows that $-Z$ is sharp except when  the continued fraction for $p'/q'$ is $p'/q'=[a_0,2, \dots, 2, 1]^-=a_0-1$ where $a_0$ is an even integer and $V_{\frac{a_0-2}{2}}> V_{\frac{a_0}{2}}$.
\end{rem}

\subsection{Proof of Theorem~\ref{Ch1:thm:sharpextension}}
Recall that if $S_{p'/q'}(K)$ bounds a sharp manifold for some $p'/q'>0$, then \eqref{Ch1:eq:sharpnuplusbound} in Lemma~\ref{Ch1:lem:wigiveCMlattice} shows that $p'/q'>2\nu^+(K)$. Thus Theorem~\ref{Ch1:thm:sharpextension} follows from the following lemma.

\begin{lem}\label{Ch1:lem:generalsharpcobor}
If $p/q\geq p'/q'> 2\nu^{+}(K)-1$ are positive rational numbers, then there is a sharp cobordism from   $S_{p'/q'}^3(K)$ to $S_{p/q}^3(K)$. Consequently, if $S_{p'/q'}^3(K)$ bounds a sharp manifold, then so does $S_{p/q}^3(K)$.
\end{lem}
\begin{proof}
Lemma~\ref{Ch1:lem:sharpsurgcobord} shows that if $[a_0, \dots, a_l, b_1, \dots, b_k]^- >2\nu^{+}(K)-1$, then there is a sharp cobordism from $S_{[a_0, \dots, a_l, b_1, \dots, b_k]^-}^3(K)$ to $S^3_{[a_0, \dots , a_l]^-}(K)$ for $a_i\geq 2$ for $1\leq i\leq l$, $a_0\geq 1$, $b_i\geq 2$ for $1\leq i<k$ and $b_k\geq 1$. In particular, the identity
\[[a_0, \dots, a_l]^-=[a_0, \dots, a_l+1,1]^-\]
shows that there is a sharp cobordism from $S_{[a_0, \dots, a_l]^-}^3(K)$ to $S_{[a_0, \dots, a_l+1]^-}^3(K)$. If $p'/q'< p/q$, then for some $m\geq0$, we can write their continued fractions in the forms
\[p/q=[a_1, \dots, a_m, a_{m+1}, \dots, a_{m+k}]^-\]
and
\[p'/q'=[a_1, \dots, a_m, a_{m+1}', \dots, a_{m+k'}']^-,\]
where $a_{m+1}'< a_{m+1}$. From the above, we see that
\begin{align*}
r_0         &=p'/q'=[a_1, \dots, a_m, a_{m+1}', \dots, a_{m+k'}']^-,\\
r_1         &=[a_1, \dots, a_m, a_{m+1}']^-,\\
r_2         &=[a_1, \dots, a_m, a_{m+1}'+1]^-,\\
            &\quad \vdots\\
r_\alpha    &=[a_1, \dots, a_m, a_{m+1}-1]^-=[a_1, \dots, a_m, a_{m+1},1]^-,\\
r_{\alpha+1}&=[a_1, \dots, a_m, a_{m+1},2]^-,\\
            &\quad\vdots\\
r_{M-1}     &=[a_1, \dots, a_m, a_{m+1}, \dots, a_{m+k}-1]^-,\\
r_{M}       &=p/q=[a_1, \dots, a_m, a_{m+1}, \dots, a_{m+k}]^-,
\end{align*}
forms an increasing sequence of rational numbers, with $r_0=p'/q'$ and $r_M=p/q$ such that there is a sharp cobordism from $S^3_{r_{i}}(K)$ to $S^3_{r_{i+1}}(K)$. Since Lemma~\ref{Ch1:lem:sharpcomposition} shows that the composition of sharp cobordisms is again sharp, there is a sharp cobordism from $S^3_{p'/q'}(K)$ to $S^3_{p/q}(K)$, as required.
\end{proof}

\section{Calculating stable coefficients}\label{sec:CMinvariant}
In this section, we deduce Theorem~\ref{Ch1:thm:stabledependence} from Theorem~\ref{intro:thm:CM} by using equation \eqref{Ch1:eq:w0formula2} to show that the stable coefficients of the changemaker lattice $L$ are an invariant of $K$. Thus, this section is devoted to solving the following problem.

\begin{ques}\label{Ch1:ques:solvingforw0}
Given a non-negative sequence, $(V_i)_{i\geq0}$, which satisfies $V_i-1 \leq V_{i+1}\leq V_i$ for all $i$, is there $\rho=(\rho_1, \dots, \rho_t) \in \mathbb{Z}^{t}$ with $\norm{\rho}=n$ such that
\begin{equation}\label{Ch1:eq:Vi1}
8V_{k} = \min_{\substack{|c\cdot \rho| =  n - 2k \\c \in \Char(\mathbb{Z}^{t})}} \norm{c} - t,
\end{equation}
for $0\leq k \leq n/2$?
\end{ques}
We may assume that $\rho_i\geq 0$ for all $i$ and that the $\rho_i$ form a decreasing sequence:
\[\rho_1\geq  \dots \geq \rho_t\geq 0.\]

Observe that \eqref{Ch1:eq:Vi1} has three pieces of input data, the sequence $(V_i)_{i\geq 0}$ and the two integers $n$ and $t$. Given some choice of $(V_i)_{i\geq 0}$, $n$ and $t$, there is no guarantee that there is $\rho$ satisfying \eqref{Ch1:eq:Vi1}. However, we will show that if such a $\rho$ exists, then it is unique and characterize the dependence of its coefficients on $(V_i)_{i\geq 0}$.
\begin{thm}\label{Ch1:thm:calculaterho}
If $\rho$ satisfies \eqref{Ch1:eq:Vi1}, then the coefficients satisfying $\rho_i\geq 2$ are determined by $(V_i)_{i\geq 0}$. In particular, they are independent of $n$ and $t$. 
\end{thm}
This is exactly what we need to deduce Theorem~\ref{Ch1:thm:stabledependence} from Theorem~\ref{intro:thm:CM}.
\begin{rem}\label{Ch1:rem:removezeroes}
If $\rho_t=0$, then any $c$ minimizing the right hand side of \eqref{Ch1:eq:Vi1} must have $c_t=\pm 1$. So we see that $\rho'=(\rho_1, \dots, \rho_{t-1})$ satisfies
\begin{equation*}
8V_{k} = \min_{\substack{|c\cdot \rho'| \equiv n - 2k\\c \in \Char(\mathbb{Z}^{t-1})}}\norm{c} - (t-1),
\end{equation*}
for all $0\leq k\leq n/2$. Thus we may assume that $\rho_i\geq 1$ for all $i$.
\end{rem}

For $m \geq 1$, it will be convenient to consider the quantities
\[T_m=|\{0\leq i \,|\, 1 \leq V_i\leq m\}|.\]
These are illustrated in Figure~\ref{Ch1:fig:tvdiagram}. Throughout this section, we will let $g$ denote
\[g= T_{V_0}=\min \{i\, | \, V_i=0\}.\]
We will show how to calculate the $T_m$ in terms of $\rho$. First we need to define the following collection of tuples for each $m\geq 0$:
\[S_m=\{\alpha \in \mathbb{Z}^{t}\,|\, \alpha_i \geq 0, 2m=\sum_{i=1}^t\alpha_i(\alpha_i + 1) \}.\]

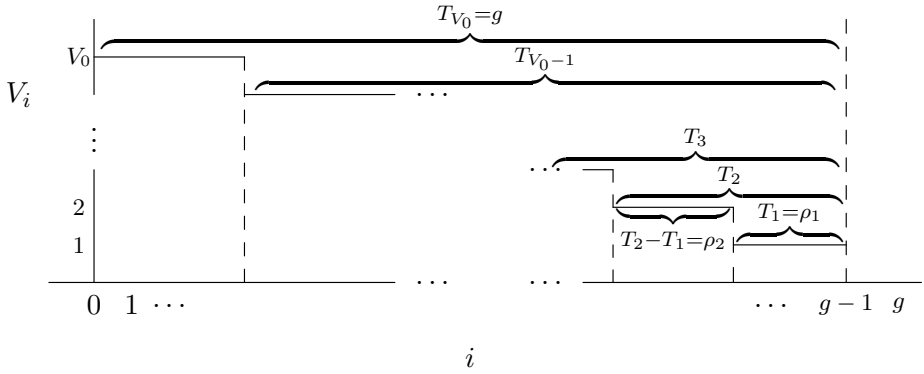
\begin{figure}[h]
\begin{xy}
<1cm,1cm>*{};<1cm,2.5cm>*{}**@{-},  
<1cm,3.5cm>*{};<1cm,4.5cm>*{}**@{-},  
<0.4cm,1cm>*{};<5cm,1cm>*{}**@{-},
<5.5cm,1cm>*{\dots},             
<7cm,1cm>*{\dots},             
<7.5cm,1cm>*{};<12cm,1cm>*{}**@{-},  
<1cm,3cm>*{\vdots},             
<0cm,3.5cm>*{V_i},                
<6cm,0cm>*{i},                    
<1cm,4cm>*{};<3cm,4cm>*{}**@{-},    
<7.5cm,2.5cm>*{};<7.9cm,2.5cm>*{}**@{-},    
<3cm,1cm>*{};<3cm,4cm>*{}**@{--},   
<3cm,3.5cm>*{};<5cm,3.5cm>*{}**@{-},
<5.5cm,3.5cm>*{\dots},             
<7cm,2.5cm>*{\dots},             
<7.9cm,2cm>*{};<9.5cm,2cm>*{}**@{-},  
<7.9cm,2.5cm>*{};<7.9cm,1cm>*{}**@{--},   
<9.5cm,2cm>*{};<9.5cm,1cm>*{}**@{--},   
<9.5cm,1.5cm>*{};<11cm,1.5cm>*{}**@{-},  
<11cm,4.5cm>*{};<11cm,1cm>*{}**@{--},   
<1cm,0.7cm>*{0},        
<2cm,0.7cm>*{\dots},   
<10cm,0.7cm>*{\dots}, 
<1.5cm,0.7cm>*{1},
<11cm,0.7cm>*{\text{{\footnotesize $g-1$}}}, 
<11.7cm,0.7cm>*{\text{{\footnotesize $g$}}}, 
<0.8cm,4cm>*{\text{{\footnotesize $V_0$}}},        
<0.8cm,2cm>*{\text{{\footnotesize $2$}}},        
<0.8cm,1.5cm>*{\text{{\footnotesize $1$}}},        
<10.25cm,1.8cm>*{\overbrace{\text{\makebox[1.4cm]{}}}^{T_1=\rho_1}}, 
<9.45cm,2.3cm>*{\overbrace{\text{\makebox[3cm]{}}}^{T_2}}, 
<9cm,2.8cm>*{\overbrace{\text{\makebox[3.8cm]{}}}^{T_3}}, 
<7cm,3.85cm>*{\overbrace{\text{\makebox[7.7cm]{}}}^{T_{V_0-1}}}, 
<6cm,4.4cm>*{\overbrace{\text{\makebox[9.8cm]{}}}^{T_{V_0}=g}}, 
<8.7cm,1.7cm>*{\underbrace{\text{\makebox[1.5cm]{}}}_{T_2-T_1=\rho_2}},
\end{xy}
\caption[The relationship between the $V_i$ and the $T_i$]{A graph to show the relationship between the $V_i$ and the $T_i$. We have also shown how $\rho_1$ and $\rho_2$ occur as the number of $V_i$ equal to one and two, respectively.}
\label{Ch1:fig:tvdiagram}
 \end{figure}

\begin{lem}\label{Ch1:lem:calcTi}
We can calculate $T_m$ by
\[T_m = \max_{\alpha \in S_m} \rho \cdot \alpha,\]
for $0\leq m<V_0$, and $T_{V_0}=g$ satisfies
\[T_{V_0}=\frac{1}{2}\sum_{i=1}^t \rho_i^2-\rho_i
\quad \text{and}\quad
T_{V_0}\leq \max_{\alpha \in S_{V_0}} \rho \cdot \alpha.\]
\end{lem}
\begin{proof}
Since the $V_k$ form a decreasing sequence with $V_k=0$ if and only if $k\geq g$, we necessarily have $T_{V_0}=g$. We see that $V_k=0$ if and only if there is $c\in \{\pm 1\}^{t}$ with $c\cdot \rho=n-2k$. The smallest of value $k$ for which this is true is $k=\frac{1}{2}(n-\sum_{i=1}^{t} \rho_i)$, which is obtained by taking $c=(1,\dots, 1)$. Thus we get $2g=\sum_{i=1}^{t} \rho_i^2-\rho_i$, as required.

Now observe that for $0\leq m<V_0$, we have
\[T_m= g-\min\{k \,|\, V_k=m\}.\]
If $V_k = m$ and $0\leq k<n/2$, then there is $c\in \Char(\mathbb{Z}^{t})$ such that $\norm{c}-t=8m$ and $c\cdot \rho = n-2k$. If we write the coefficients of $c$ in the form $c_i=(2\alpha_i +1)$, then these conditions become
\begin{equation}\label{Ch1:eq:calcTiformula}
\sum_{i=1}^{t}\alpha_i(\alpha_i+1)=2m
\quad\text{and}\quad
2k= n - \sum_{i=1}^t \rho_j - 2\alpha \cdot \rho=2g -2\alpha \cdot \rho.
\end{equation}
Therefore, we have
\[\min\{k \,|\, V_k=m\}=\min_{\{\alpha\in \Z^t \,|\, \sum \alpha_i(\alpha_i+1)=2m\}} g-\alpha\cdot \rho.\]
As $\rho_i\geq0$ for all $i$, any vector $\alpha$ minimizing the right-hand side of this expression  can be assumed to satisfy $\alpha_i\geq 0$ for all $i$. Thus we we see that
\[T_m = \max_{\alpha \in S_m} \rho \cdot \alpha,\]
for $0\leq m<V_0$.

The equation \eqref{Ch1:eq:calcTiformula} also shows that there must exist $\alpha$ satisfying $\sum_{i=1}^{t}\alpha_i(\alpha_i+1)=2V_0$ and $\alpha \cdot \rho = g$. This implies the inequality
\[T_{V_0}\leq \max_{\alpha \in S_{V_0}} \rho \cdot \alpha,\]
which completes the proof.
\end{proof}

\begin{rem}\label{Ch1:rem:rho0rho1}
It follows from this lemma that $T_1=\rho_1$ and $T_2=\rho_1+\rho_2$. In particular this implies that $\rho_2=T_2-T_1$. This is illustrated in Figure~\ref{Ch1:fig:tvdiagram}.
\end{rem}
We now begin the process of showing how the remaining $\rho_i$ can be recovered from the sequence $(V_i)_{i\geq0}$. We begin with the simplest case, which is when $V_0\leq 1$.
\begin{lem}\label{Ch1:lem:V0=1}
If $V_0\leq 1$, then $g\leq 3$ and $\rho$ takes the form
\[
\rho=
\begin{cases}
(1,1 , \dots, 1)   &\text{if }g=0\\
(2,1 , \dots, 1)   &\text{if }g=1\\
(2,2 ,1, \dots, 1) &\text{if }g=2\\
(3,1 , \dots, 1)   &\text{if }g=3.
\end{cases}\]
\end{lem}
\begin{proof}
If $V_0=0$, then $g=0$ and Lemma~\ref{Ch1:lem:calcTi} implies that $\sum_{i=1}^t \rho_i^2-\rho_i=0$. This shows that we have $\rho_i=1$ for all $1\leq i \leq t$.

Suppose now that $V_0=1$. By Lemma~\ref{Ch1:lem:calcTi}, we have
\[0<T_{1}=g=\frac{1}{2}\sum_{i=1}^t \rho_i^2-\rho_i\leq \max_{\alpha \in S_{1}} \rho \cdot \alpha=\rho_1.\]
Thus we must have $\rho_1^2-\rho_1\leq 2\rho_1,$ and hence $\rho_1\leq 3$. If $\rho_1=3$, then we have
\[g=3+\frac{1}{2}\sum_{i=2}^t\rho_i(\rho_i-1)\leq \rho_1 =3,\]
which implies the $\rho_i=1$ for $2\leq i \leq t$ and $g=3$. If $\rho_1=2$, then $g\leq 2$ and $\rho_2\in \{1,2\}$, giving the other two possibilities in the statement of the lemma.
\end{proof}

From now on we will suppose that $V_0>1$. This allows us to define the quantity \[\mu = \min_{1\leq i<V_0} \{ T_i-T_{i-1}\}.\]
Since $T_1=\rho_1$ and $T_0=0$, we must have $\mu\leq \rho_1$.
\begin{lem}\label{Ch1:lem:mubound}
If $\rho_1\geq 5$ or $\sum_{\rho_i \text{even}}\rho_i\geq 6$, then $\mu \leq 2$.
\end{lem}
\begin{proof}
For $m<V_0$, Lemma~\ref{Ch1:lem:calcTi} shows that there is $\alpha \in S_m$ such that $\rho \cdot \alpha=T_m$. If $\alpha_l>0$, then consider $\alpha'$ defined by
\[\alpha'_i=
\begin{cases}
\alpha_i & i\ne l\\
\alpha_i-1 &i=l.
\end{cases}\]
By construction, we have $\alpha'\in S_{m-\alpha_l}$ and $\alpha'\cdot \rho = \alpha \cdot \rho - \rho_l=T_m-\rho_l.$
As $\alpha' \cdot \rho \leq T_{m-\alpha_l}$, we get
\begin{equation}\label{Ch1:eq:mubound}
\rho_l\geq T_m-T_{m-\alpha_l}\geq \alpha_l \mu.
\end{equation}
If we have a maximiser $\alpha \in S_m$ such that $\rho \cdot \alpha=T_m$ and $\alpha$ does not satisfy
\begin{equation}\label{Ch1:eq:rhobounds}
\alpha_i \leq
\begin{cases}
\frac{\rho_i-2}{2}    &\text{if $\rho_i$ even,}\\
\frac{\rho_i-3}{2}    &\text{if $\rho_i\geq 5$ odd,}\\
\frac{\rho_i-1}{2}    &\text{if $\rho_i\in\{1,3\}$,}
\end{cases}
\end{equation}
for all $i$, then there is $l$ such that $\frac{\rho_l}{\alpha_l}<3$. So by \eqref{Ch1:eq:mubound}, we see that $\mu\leq 2$. We will show that if $\rho$ satisfies the hypotheses of the lemma, then there must be at least one maximiser which does not satisfy \eqref{Ch1:eq:rhobounds} for all $i$.

Let $c\in \Char(\mathbb{Z}^{t})$ be such that $c\cdot\rho=n$ and $8V_0\geq \norm{c} -t$. The Cauchy-Schwarz inequality implies that
\[|c\cdot\rho|^2=n^2\leq \norm{\rho}\norm{c}=n\norm{c},\]
showing that $\norm{c}\geq n$ with equality if and only if $c=\rho$. Therefore,
\[V_0\geq \frac{\norm{\rho}-t}{8}=\frac{1}{8}\sum_{i=1}^t (\rho_i^2-1),\]
with equality only if $\rho\in \Char(\mathbb{Z}^{t})$.
We will let $N$ denote the quantity
\[N=\lfloor \frac{1}{8}\sum_{i=1}^t (\rho_i^2-1) \rfloor\leq V_0.\]

Now take $\alpha \in S_m$, which satisfies the conditions given by \eqref{Ch1:eq:rhobounds}.
It follows that
\begin{align}
\begin{split}\label{Ch1:eq:maximiserbound}
m&=\frac{1}{2}\sum_{i=1}^t \alpha_i (\alpha_i +1)\\
 &\leq \sum_{\rho_i \geq 5 \text{ odd}}\frac{(\rho_i-3)(\rho_i-1)}{8} + \sum_{\rho_i \text{ even}}\frac{\rho_i(\rho_i-2)}{8} + \sum_{\rho_i \in \{1,3\}} \frac{\rho_i^2-1}{8}\\
 &=\sum_{i=1}^t \frac{\rho_i^2-1}{8} + \sum_{\rho_i\geq 5 \text{ odd}}\frac{(1-\rho_i)}{2} +
\sum_{\rho_i \text{ even}}\frac{1-2\rho_i}{8}.
\end{split}
\end{align}

If $\rho_1$ is odd and $\rho_1\geq 5$, then \eqref{Ch1:eq:maximiserbound} shows that
\[m\leq \sum_{i=1}^t \frac{\rho_i^2-1}{8} -2 < N-1\]
In particular, there is no $\beta \in S_{N-1}$ satisfying \eqref{Ch1:eq:rhobounds}. Since $N-1< V_0$, there is $\beta \in S_{N-1}$ with $\beta \cdot \rho=T_{N-1}$ and so \eqref{Ch1:eq:mubound} implies that $\mu\leq 2$.

If $\sum_{\rho_i \text{ even}}\rho_i\geq 6$, then we must have $\sum_{\rho_i \text{ even}}(2\rho_i-1)\geq \frac{3}{2}\sum_{\rho_i \text{ even}}\rho_i \geq 9$. Therefore, \eqref{Ch1:eq:maximiserbound} shows that
\[m< \sum_{i=1}^t \frac{\rho_i^2-1}{8} -1 < N.\]
In particular, there is no $\beta \in S_{N}$ satisfying \eqref{Ch1:eq:mubound}. Since we are assuming there is an even $\rho_i$, we have $N<V_0$ and so there exists $\beta\in S_{N}$ such that $\beta \cdot \rho=T_{N}$ and so \eqref{Ch1:eq:mubound} implies that $\mu\leq 2$.
\end{proof}
If $\mu>2$, then $\rho$ must fall into one of a small number of cases.
\begin{lem}\label{Ch1:lem:mu=3}
If $\mu>2$, then $\mu=T_1=3$ and $\rho$ takes the form
\[
\rho=
\begin{cases}
(\underbrace{3, \dots , 3}_{d}, 1, \dots, 1)      &\text{if }g=3d\\
(\underbrace{3, \dots , 3}_{d},2, 1, \dots, 1)    &\text{if }g=3d+1\\
(\underbrace{3, \dots , 3}_{d},2,2, 1, \dots, 1)  &\text{if }g=3d+2.
\end{cases}
\]
\end{lem}
\begin{proof}
If $\mu>2$, then Lemma~\ref{Ch1:lem:mubound} and the observation that $\mu \leq T_1=\rho_1$ show that $\rho_1 \in \{3,4\}$. We consider these two possibilities separately.

First, suppose that $\rho_1=4$. Lemma~\ref{Ch1:lem:mubound} shows that if $\rho_1=4$ and $\mu>2$, then $\rho_i$ is odd for all $i\geq 1$. Thus, $\rho$ must take the form
\[\rho=(4,\underbrace{3, \dots , 3}_{d}, 1, \dots , 1)   \quad\text{where }g=3d+6.\]
However, it is straightforward to verify that this $\rho$ gives $T_{d+1}=4+3d$ and $T_{d+2}=5+3d$. This shows that we must have $\mu=1$, in this case.

If $\rho_1=3$ and $\mu>2$, then Lemma~\ref{Ch1:lem:mubound} implies that we have $\rho_i=2$ for at most two values $i$. Thus if $\rho_1=3$, then $\rho$ must take the required form given in the lemma. Furthermore, it can be easily verified that $T_{k}=3k$ for $0\leq k \leq d$ and $g\leq \max_{\alpha \in S_{d+1}}\alpha\cdot \rho$. This shows that $\mu=3$ in these cases.
\end{proof}
Let $\rho^{(k)}$ denote $(\rho_1,\dots, \rho_k,0,\dots,0)$. The following lemma allows us to calculate the $\rho_i$ iteratively.
\begin{lem}\label{Ch1:lem:computerho}
If $1\leq t<V_0$ is minimal such that $\max_{\alpha \in S_t} \rho^{(l)}\cdot \alpha < T_t$, then
\[\rho_{l+1}=T_t-T_{t-1}.\]
\end{lem}
\begin{proof}
Let $\alpha\in S_{t-1}$ be such that $\rho^{(l)}\cdot \alpha=T_{t-1}$. Such an $\alpha$ must also satisfy $\rho \cdot \alpha=T_{t-1}$. In particular $\alpha_i=0$ for $i>l$.

Now we consider $\alpha'\in S_t$ defined by
\[\alpha'_i=
\begin{cases}
\alpha_i & i\ne l+1\\
1 & i=l+1.
\end{cases}\]
We have $\alpha'\cdot \rho = T_{t-1}+\rho_{l+1}\leq T_t$. This implies that
\begin{equation}\label{Ch1:eq:algproof1}
\rho_{l+1}\leq T_t-T_{t-1}.
\end{equation}

Let $\beta \in S_{t}$, be such that $\rho \cdot \beta = T_t$. Since $\max_{\alpha \in S_t} \rho^{(l)}<T_t$, we may assume $\beta_{l+1}>0$. Thus we can define $\beta'$ by
\[\beta'_i=
\begin{cases}
\beta_i & i\ne l\\
\beta_i-1 &i=l+1.
\end{cases}\]
We have $\beta'\in S_{t-\beta_l}$. Therefore we obtain
\begin{equation}\label{Ch1:eq:algproof2}
T_{t-1}\geq T_{t-\beta_l} \geq \rho \cdot \beta' =  T_{t}-\rho_{l+1}.
\end{equation}
Combining \eqref{Ch1:eq:algproof1} and \eqref{Ch1:eq:algproof2} shows that $\rho_{l+1}=T_{t}-T_{t-1}$, as claimed.
\end{proof}
The following lemma shows that this iterative process will calculate most of the $\rho_i$.
\begin{lem}\label{Ch1:lem:finalcoeffs}
If $\rho^{(N)}$ satisfies $\max_{\alpha \in S_t} \rho^{(N)}\cdot \alpha = T_t$ for all $t<V_0$, then $\rho_i\leq \mu$ for all $i>N$.
\end{lem}
\begin{proof}
Let $\tau<V_0-1$ be such that $T_{\tau+1}-T_{\tau}=\mu$. There is $\alpha \in S_{\tau}$ such that $\alpha \cdot \rho = \alpha \cdot \rho^{(N)} = T_{\tau}$. Such an $\alpha$ must satisfy $\alpha_l=0$ for $l>N$. Let $\alpha'\in S_{\tau +1}$ be defined by
\[\alpha'_i=
\begin{cases}
\alpha_i & i\ne l\\
1 &i=l.
\end{cases}\]
We have
\[\rho_l=\alpha'\cdot \rho-T_{\tau}\leq T_{\tau+1}-T_{\tau}=\mu,\]
as required.
\end{proof}

Now we prove Theorem~\ref{Ch1:thm:calculaterho}. Note that this proof can easily be converted into an algorithm for calculating $\rho$.

\begin{proof}[Proof of Theorem~\ref{Ch1:thm:calculaterho}]
If $V_0\leq 1$, then Lemma~\ref{Ch1:lem:V0=1} shows that the stable coefficients of $\rho$ are determined by $g$. Thus we may assume $V_0>1$ and hence we may define $\mu = \min_{1\leq i<V_0} \{ T_i-T_{i-1}\}$. If $\mu>2$, then Lemma~\ref{Ch1:lem:mu=3} shows that the stable coefficients of $\rho$ are determined by $g$. Thus we may assume that $\mu\leq 2$. Now Lemma~\ref{Ch1:lem:computerho} shows that if $\rho^{(l)}$ is determined by the $V_i$, then either $\rho^{(l+1)}$ is determined by the $V_i$ or $\rho^{(l)}$ satisfies $\max_{\alpha \in S_t} \rho^{(l)}\cdot \alpha = T_t$ for all $t<V_0$. However, in this latter case Lemma~\ref{Ch1:lem:finalcoeffs} shows that $\rho_{i}\leq \mu =2$ for all $i>l$. Thus the $V_i$ determine all values $\rho_i>2$. Since
\[g=\sum_{\rho_i >2}\frac{\rho_i(\rho_i-1)}{2} + |\{i \,|\, \rho_i =2\}|,\]
this is sufficient show that all the stable coefficients of $\rho$ are determined by the $V_i$.
\end{proof}
\begin{proof}[Proof of Theorem~\ref{Ch1:thm:stabledependence} and Corollary~\ref{Ch1:cor:interfromunique}]
By Theorem~\ref{intro:thm:CM}, if $S_{p/q}^3(K)$ bounds a sharp manifold $X$ for some $p/q>0$, then the intersection form $Q_X$ satisfies
\[-Q_X \cong L \oplus \Z^S,\]
for some $p/q$-changemaker lattice $L=\langle w_0, \dots, w_l \rangle^\bot\subseteq \Z^N$ where $w_0$ satisfies \eqref{Ch1:eq:w0formula2}. Theorem~\ref{Ch1:thm:calculaterho} shows that the stable coefficients of $L$ are determined by the values of the $V_i$. As the $V_i$ are invariants of $K$, this shows that the stable coefficients are an invariant of the knot $K$.

Since a $p/q$-changemaker lattice is determined up to isomorphism by its stable coefficients and the value $p/q$, this shows that $L$ is determined by $K$ and $p/q$. As we have $b_2(X)=\rk(Q_X) =S+ \rk(L)$, this shows that $Q_X$ is determined by the $K$ the slope $p/q$ and $b_2(X)$, proving Theorem~\ref{Ch1:thm:sharpextension}. Moreover, if $X'$ is a sharp manifold with boundary $S_{p/q}^3(K)$ and $b_2(X')=b_2(X)+k$, we have
\[-Q_{X'}\cong L\oplus \Z^{S+k}\cong -Q_X \oplus \Z^k.\]
Since $Q_{X\#_{k}\overline{\mathbb{C}P}^2}\cong Q_X  \oplus (-\Z)^k$, this proves Corollary~\ref{Ch1:cor:interfromunique}.
\end{proof} 
\chapter{Characterizing slopes for torus knots}\label{chap:charslopes}
In this chapter, we take a brief diversion to address the problem of determining characterizing slopes of torus knots.
\begin{defn}
We say that $p/q$ is a characterizing slope for $K\subseteq S^3$ if the existence of an orientation-preserving homeomorphism from $S^3_{p/q}(K)$ to $S^3_{p/q}(K')$ implies that $K=K'$.
\end{defn}
Using a combination of Heegaard Floer homology and geometric techniques, Ni and Zhang were able to prove the following theorem.
\begin{thm}[Ni and Zhang, \cite{Ni14characterizing}]\label{Ch2:thm:NiZhang}
For the torus knot $T_{r,s}$ with $r>s>1$ any non-trivial slope $p/q$ satisfying
\[\frac{p}{q} \geq \frac{30(r^2-1)(s^2-1)}{67}\]
is a characterizing slope.
\end{thm}
Their argument requires an upper bound on the genus of any knot $K$ satisfying $S^3_{p/q}(K)\cong S^3_{p/q}(T_{r,s})$. Since $S^3_{p/q}(T_{r,s})$ is an $L$-space bounding a sharp manifold for $p/q\geq rs-1$, we can apply the following theorem to obtain $g(K)=g(T_{r,s})$, whenever $p/q\geq 4g(K)+6$ \cite{mccoy2014sharp}.
\begin{thm}\label{Ch2:thm:Alexuniqueness}
Suppose that for some $p/q>0$, there are knots $K,K'\subseteq S^3$ such that $S^3_{p/q}(K)\cong S^3_{p/q}(K')$ is an $L$-space bounding a sharp manifold. If $p/q\geq 4g(K)+6$, then
\[\Delta_K(t)=\Delta_{K'}(t) \text{ and } g(K)=g(K').\]
\end{thm}

This allows us to lower Ni and Zhang's quadratic bound to one which is linear in $rs$.
\begin{thm}\label{Ch2:thm:charslopes}
For the torus knot $T_{r,s}$ with $r>s>1$ any non-trivial slope $p/q$ satisfying
\[\frac{p}{q} \geq \frac{43}{4}(2g(T_{r,s})-1)=\frac{43}{4}(rs-r-s).\]
is a characterizing slope.
\end{thm}
\begin{rem}
Theorem~\ref{Ch2:thm:Alexuniqueness} is not the best possible result. For example, the genus bound can be improved to $p/q\geq 4g(K)+4$ \cite{mccoy2014sharp}. We have opted for the weaker statement here because it allows the proof to be simplified but leaves the conclusions of the main application, Theorem~\ref{Ch2:thm:charslopes}, unaffected.
\end{rem}

\section{Recovering the Alexander polynomial}
The proof of Theorem~\ref{Ch2:thm:Alexuniqueness} occupies this section. The key ingredient is the following lemma which provides conditions under which the changemaker structure of a lattice is unique.
\begin{lem}\label{Ch2:lem:CMuniqueness}
Let $L=\langle w_0, \dots , w_l \rangle^{\bot}\subseteq \mathbb{Z}^{N}$ be a $p/q$-changemaker lattice with stable coefficients $(\rho_1, \dots, \rho_t)$, where $2\leq \rho_1\leq \dots \leq \rho_t$.  Let $\phi:L\rightarrow \mathbb{Z}^N$ be an embedding such that
\[\phi(L)=\langle w_0', \dots , w_l' \rangle^{\bot} \subseteq \mathbb{Z}^N\]
is a $p/q$-changemaker lattice. If $p/q \geq 2 + 2\rho_t+ \sum_{i=1}^t \rho_i^2$, then $w_0$ and $w_0'$ have the same changemaker coefficients.
\end{lem}
\begin{proof}
First note that if the continued fraction expansion for $p/q$ is $p/q=[a_0,\dots, a_l]^-$, then $N\leq -l+ \sum_{i=0}^l a_i$, with equality if and only if the set of stable coefficients for $L$ is empty. Thus we see that the lemma is true when $L$ has no stable coefficients simply by looking at the value of $N$.

From now on we will assume $L$ has non-trivial stable coefficients. There is a choice of orthonormal basis for $\Z^N$ such that $w_0$ takes the form
\[w_0=
\begin{cases}
\rho_t e_{m+t} + \dots + \rho_1 e_{m+1}+ e_m + \dotsb + e_1      &\text{if $q=1$} \\
\rho_t e_{m+t} + \dots + \rho_1 e_{m+1}+ e_m + \dotsb + e_1 +e_0 &\text{if } q>1,
\end{cases}
\]
where $m\geq 2+2\rho_t\geq 6$ and $\norm{w_0}=n=\lceil p/q \rceil$. It follows that $L$ contains vectors $v_2,\dots, v_{m+t}$ defined by
\[v_k=
\begin{cases}
-e_k + e_{k-1}      &\text{for } 2\leq k\leq m, \\
-e_k + e_1+ \dots + e_{\rho_{k}} &\text{for } m+1 \leq k \leq m+t.
\end{cases}
\]

We will consider the image of these vectors under $\phi$. For $k$ in the range $2\leq k \leq m+t$, let $u_k$ denote the vector $u_k= \phi(v_k)$. For $j$ and $k$ satisfying $2\leq k<j \leq m$, we have $\norm{v_k}=\norm{v_j}=2$ and
\[
v_k\cdot v_j=
\begin{cases}
-1 &\text{if } j=k+1\\
0  &\text{otherwise.}
\end{cases}
\]

It is clear that we may choose orthogonal unit vectors $f_1,f_2,f_3$ such that $u_2=-f_2+f_1$ and $u_3=-f_3+f_2$.

There are two possibilities for $u_4$. We can either have $u_4=-f_2-f_1$ or there is a unit vector $f_4\notin \{f_1,f_2,f_3\}$ such that $u_4=-f_4+f_3$. However, if we have $u_4=-f_2-f_1$, then there is no vector $x \in \mathbb{Z}^N$ with $x\cdot f_1 \ne 0$ and $x\cdot u_4=x \cdot u_2=0$, contradicting the existence of $w_k'$ with $w_k'\cdot f_1 \ne 0$. Thus $u_4$ must take the form $u_4=-f_4+f_3$. Continuing in this way, it follows that there is a choice of distinct orthogonal unit vectors $f_1, \dotsc, f_m$ in $\Z^N$, such that $u_k=-f_k+f_{k-1}$ for each $k$ in the range  $2\leq k \leq m$.

Now we determine the form that $u_{m+k}$ must take. Consider the quantity $\lambda_k=u_{m+k}\cdot f_1$. For $2\leq i \leq m$ we have $v_i\cdot v_{m+k}=0$ for $i\ne  \rho_{k}+1$ and $v_k\cdot v_{\rho_{k+1}}=1$. This shows that we have
\[u_{m+k}\cdot f_i =
\begin{cases}
\lambda_k   &\text{for } 1\leq i\leq \rho_{1}\\
\lambda_k-1 &\text{for } i> \rho_{1}.
\end{cases}
\]
Thus by computing the norm of $u_{m+k}$, we obtain
\[\norm{u_{m+k}}=\norm{v_{m+k}}= \rho_k+1 \geq (\lambda_k-1)^2 (m-\rho_k) + \lambda_k^2 \rho_k.\]
Since $\rho_k+1<m-\rho_k$, it follows that $\lambda_k=1$.

Thus we see that $u_{m+k}$ takes the form
\[
u_{m+k}=-f_{m+k}+ f_1 + \dots + f_{\rho_{k}}
\]
for some choice of unit vector $f_{m+k} \notin \{\pm f_1, \dots, \pm f_m\}$. Since $u_{m+k}\cdot u_{m+l}= \min\{\rho_k,\rho_l\}$, it follows that $f_{m+k}\cdot f_{m+l}=0$ for $k\ne l$.

Let $x$ be a vector in the orthogonal complement of $\phi(L)$. Since $x$ must satisfy $u_k\cdot x=0$ for $2\leq k\leq m+t$, we have
\[
x\cdot f_k =
\begin{cases}
x\cdot f_1               &\text{for } 1\leq k\leq m \text{ and} \\
\rho_{k-m}(x\cdot f_1)   &\text{for } m+1 \leq k \leq m+t.
\end{cases}
\]
In particular, if $x\cdot f_1\ne 0$, then $|x\cdot f_{m+t}|>1$. This shows that if
\[\phi(L)=\langle w_0', \dots , w_l' \rangle^{\bot} \subseteq \mathbb{Z}^N\]
is a $p/q$-changemaker lattice, then $w_i'\cdot f_1= 0$ for $i\geq 1$ and $w_0'\cdot f_1\ne 0$.
As
\[m+\sum_{i=1}^t\rho_i^2+1 \geq \norm{w_0'}=n\geq (w_0'\cdot f_1)^2(m+\sum_{i=1}^t\rho_i^2),\]
it follows that $w_0'$ takes the form,
\[\pm w'_0=
\begin{cases}
\rho_t f_{m+t} + \dots + \rho_1 f_{m+1}+ f_m + \dots + f_1      &\text{if } q=1, \\
\rho_t f_{m+t} + \dots + \rho_1 f_{m+1}+ f_m + \dots + f_1 +f_0 &\text{if } q>1.
\end{cases}
\]
This shows that $w_0$ and $w_0'$ have the same changemaker coefficients, as required.
\end{proof}

There are examples of lattices admitting embeddings into $\mathbb{Z}^N$ as changemaker lattices in more than one way, we cannot prove unconditionally that the changemaker structure of a lattice is unique. \begin{example}\label{Ch2:exam:differentcoef}
There is an isomorphism of changemaker lattices:
\[\langle 4e_0+e_1+e_2+e_3+e_4+e_5\rangle^\bot\cong \langle 2e_0 + 2e_1+ 2e_2+2e_3+2e_4+e_5 \rangle^\bot \subseteq \mathbb{Z}^6.\]
This isomorphism can be seen by observing that both lattices admit a basis for which the bilinear form is given by the matrix
\[
  \begin{pmatrix}
   5   & -1  &      &    &     \\
   -1  & 2   & -1   &    &     \\
       & -1  & 2    & -1 &     \\
       &     & -1   & 2  &-1   \\
       &     &      & -1 & 2   \\
  \end{pmatrix}.
\]
This example is a consequence of the fact that $S^3_{21}(T_{5,4})\cong S^3_{21}(T_{11,2})\cong L(21,4)$.
\end{example}

\begin{proof}[Proof of Theorem~\ref{Ch2:thm:Alexuniqueness}]
Suppose that $Y \cong S^3_{p/q}(K)$ is an $L$-space bounding a sharp 4-manifold $X$ with intersection form $Q_X$. Let $n=\lceil p/q \rceil$. Theorem~\ref{intro:thm:CM} shows that there is an isomorphism
\[-Q_X \cong L \oplus \Z^S,\]
where
\[L=\langle w_0, \dots , w_l \rangle^\bot \subseteq \mathbb{Z}^N\]
is a $p/q$-changemaker lattice and the torsion coefficients of $\Delta_K(t)$ can be computed by the formula
\begin{equation}\label{Ch2:eq:tiK}
8t_i(K) = \min_{ \substack{ |c\cdot w_0| \equiv n - 2i \\ c \in \Char(\mathbb{Z}^{N})}} \norm{c} - N.
\end{equation}
If we write $w_0$ in the form $w_0=\rho_t e_{t+m} + \dots + \rho_1 e_{m+1}+e_m + \dots + e_1$, then Lemma~\ref{Ch1:lem:calcTi} shows that $g(K)$ can be computed by
\[2g(K)= \sum_{i=1}^t \rho_i(\rho_i-1).\]
Since $\rho_i \geq 2$ for all $i$, we have $\rho_i^2 \leq 2\rho_i(\rho_i -1)$. Thus we have
\begin{align}\begin{split}\label{Ch2:eq:genusbound}
\sum_{i=1}^t \rho_i^2 + 2\rho_t &\leq 2\sum_{i=1}^t \rho_i(\rho_i-1) - \rho_t^2 +4\rho_t\\
&=4g(K) -(\rho_t-2)^2 + 4 \leq 4g(K)+4.
\end{split}\end{align}
If $K'\subseteq S^3$ is another knot such that $Y\cong S^3_{p/q}(K')$, then there is also an isomorphism
\[-Q_X \cong L' \oplus \Z^{S'},\]
where
\[L'=\langle w'_0, \dots , w'_l \rangle^\bot\subseteq \mathbb{Z}^{N'},\]
is a $p/q$-changemaker lattice and the torsion coefficients of $\Delta_{K'}(t)$ satisfy the formula
\begin{equation}\label{Ch2:eq:tiK'}
8t_i(K') = \min_{ \substack{ c\cdot w'_0 = n- 2i \\ c \in \Char(\mathbb{Z}^{N'})}} \norm{c} - N'.
\end{equation}
As a changemaker lattice does not contain any vectors of norm one, we must have $S=S'$. It follows that we also have $N=N'$. This shows that $L$ and $L'$ are isomorphic changemaker lattices in $\Z^N$. Combining the inequality \eqref{Ch2:eq:genusbound} with the assumption $p/q\geq 4g(K)+6$ allows us to apply Lemma~\ref{Ch2:lem:CMuniqueness}. This shows $w_0$ and $w_0'$ have the same changemaker coefficients. Therefore, \eqref{Ch2:eq:tiK} and \eqref{Ch2:eq:tiK'} show that the torsion coefficients satisfy $t_i(K)=t_i(K')$ for all $i\geq 0$. As shown in Remark~\ref{Ch1:rem:torsiondeterminespoly}, the torsion coefficients of $K$ and $K'$ determine their Alexander polynomials, so we have $\Delta_{K'}(t)=\Delta_{K}(t)$ and $g(K)=g(K')$, as required.
\end{proof}

\section{Characterizing slopes}
In this section, we prove Theorem~\ref{Ch2:thm:charslopes}. Our proof follows the one given by Ni and Zhang. We obtain our improvement through the following lemma.
\begin{lem}\label{Ch2:lem:torusgenus}
For the torus knot $T_{r,s}$ with $r>s>1$, any knot $K\subseteq S^3$ satisfying
\[S^3_{p/q}(K)\cong S^3_{p/q}(T_{r,s}),\]
for some $p/q\geq 4g(T_{r,s})+6$, has genus $g(K)=g(T_{r,s})$ and Alexander polynomial $\Delta_{K}(t)=\Delta_{T_{r,s}}(t)$.
\end{lem}
\begin{proof}
Since $S^3_{rs-1}(T_{r,s})$ is a lens space \cite{Moser71elementary}, the work of Ozsv{\'a}th and Szab{\'o} have shows that it is an $L$-space bounding a sharp manifold \cite{Ozsvath03plumbed, ozsvath2005heegaard}. Using Theorem~\ref{Ch1:thm:sharpextension}, this shows that $S^3_{p/q}(T_{r,s})$ is an $L$-space bounding a sharp manifold for all $p/q\geq rs-1$. As $4g(T_{r,s})+6>rs-1$, we may apply Theorem~\ref{Ch2:thm:Alexuniqueness} to get the desired conclusion.
\end{proof}
\begin{rem}
It is possible to exhibit a sharp manifold bounding $S^3_{p/q}(T_{r,s})$ explicitly. Since the manifold $S^3_{p/q}(T_{r,s})$ is a Seifert fibred space with base orbifold $S^2$ with at most 3 exceptional fibres \cite{Moser71elementary}, it bounds a plumbed 4-manifold. For $p/q\geq rs-1$, one can find such a plumbing which is negative-definite and hence sharp \cite{Ozsvath03plumbed}.
\end{rem}

Using results of Agol, Lackenby, Cao and Meyerhoff, Ni and Zhang obtain a restriction on exceptional slopes of a hyperbolic knot \cite{Agol00BoundsI, Lackenby03Exceptional, Cao01cusped}.
\begin{prop}[Lemma~2.2, \cite{Ni14characterizing}]\label{Ch2:prop:hyperbolicbound}
Let $K\subseteq S^3$ be a hyperbolic knot. If
\[|p|\geq 10.75(2g(K)-1),\]
then $S^3_{p/q}(K)$ is hyperbolic.
\end{prop}
Combining this with work of Gabai, they show that it is not possible for surgery of sufficiently large slope on a satellite knot and a torus knot to yield the same manifold.
\begin{lem}\label{Ch2:lem:satellitebound}
If $K$ is a knot such that $S^3_{p/q}(K)\cong S^3_{p/q}(T_{r,s})$ for $r>s>1$ and
\[p/q\geq 10.75(2g(T_{r,s})-1),\] then $K$ is not a satellite.
\end{lem}
\begin{proof}
If $K$ is a satellite knot, then let $R\subseteq S^3 \setminus K$ be an incompressible torus. This bounds a solid torus $V\subseteq S^3$ which contains $K$. Let $K'$ be the core of the solid torus $V$. By choosing $R$ to be ``innermost'', we may assume that $K'$ is not a satellite. This means that $K'$ is either a torus knot or it is hyperbolic \cite{Thurston823DKleinanGroups}. Since $S^3_{p/q}(T_{r,s})$ contains no incompressible tori and is irreducible, it follows from the work of Gabai that $V_{p/q}(K)$ is again a solid torus and $K$ is either a 1-bridge knot or a torus knot in $V$ \cite{gabai89solidtori}. In either case, this is a braid in $V$ and we have
$S_{p/q}^3(K)\cong S_{p/q'}^3(K')$ where $q'=qw^2$ and $w>1$ is the winding number of $K$ in $V$.

Since $p\geq 10.75(2g(K)-1)$, Proposition~\ref{Ch2:prop:hyperbolicbound} shows that $K'$ cannot be hyperbolic. Thus we may assume that $K'$ is a torus knot, say $K'=T_{m,n}$. Since $S^3_{p/q}(T_{r,s})$ is an $L$-space and $p/q'>0$ we have $m,n>1$. The manifold $S_{p/q}^3(T_{r,s})\cong S_{p/q'}^3(T_{m,n})$ is Seifert fibred over $S^2$ with exceptional fibres of order $\{r,s,p-qrs\}= \{m,n, |p-q'mn|\}$. Hence we can assume $m=r$. By Lemma~\ref{Ch2:lem:torusgenus}, we have $\Delta_{T_{r,s}}(t)=\Delta_K(t)$. However, since $K$ is a satellite, its Alexander polynomial takes the form $\Delta_K(t)=\Delta_C(t)\Delta_{K'}(t^w)$, where $C$ is the companion knot of $K$. In particular, we have $g(K')<g(T_{r,s})$ and, consequently, $n<s$. Comparing the orders of the exceptional fibres again, this implies that $n=p-qrs$. However, we have
\begin{align*}
n=p-rsq&\geq 9.75q(rs-r-s)-q(r+s)\\
     &\geq 9.75(\max\{r,s\}-2) -(2\max\{r,s\}-1)\\
     &= 7.75\max\{r,s\} -18.5\\
     &\geq \max\{r,s\},
\end{align*}
where the last inequality holds because we have $\max\{r,s\}\geq 3$. This is a contradiction and shows that $K'$ cannot be a torus knot. Thus we see that $K$ cannot be a satellite knot.
\end{proof}

\begin{proof}[Proof of Theorem~\ref{Ch2:thm:charslopes}]
Suppose that $K$ is a knot in $S^3$ with $Y\cong S^3_{p/q}(K) \cong S^3_{p/q}(T_{r,s})$ for $p/q\geq 10.75(rs-r-s)$. Lemma~\ref{Ch2:lem:torusgenus} shows that $g(K)=g(T_{r,s})$ and $\Delta_K(t)=\Delta_{T_{r,s}}(t)$. Since $Y$ is not hyperbolic, Proposition~\ref{Ch2:prop:hyperbolicbound} shows that $K$ is not a hyperbolic knot. Lemma~\ref{Ch2:lem:satellitebound} shows that $K$ is not a satellite knot. Therefore, it follows that $K$ is a torus knot. Since two distinct torus knots have the same Alexander polynomial only if they are mirrors of one another, $K$ is either $T_{r,s}$ or $T_{-r,s}$. As $K$ admits positive $L$-space surgeries, it follows that $K=T_{r,s}$, as required.
\end{proof} 
\chapter{Lattices}\label{chap:lattices}
We develop the lattice theoretic results on which the proofs of the main results rely. The two types of lattice we will require are changemaker lattices, as defined in Definition~\ref{Intro:def:CMlattice}, and graph lattices, which are discussed in Sections~\ref{Ch3:sec:CMlattices} and \ref{Ch3:sec:graphlattices} respectively.

\paragraph{}For the purposes of this thesis, a {\em positive-definite integer lattice} $L$ is a finitely-generated free abelian group with a positive-definite, symmetric form $L\times L \rightarrow \Z$, given by ${(x,y)\mapsto x\cdot y}$. All lattices occurring in this paper will be of this form. We say that $L$ is {\em indecomposable} if it cannot be written as an orthogonal direct sum $L= L_1 \oplus L_2$ with $L_1,L_2$ non-zero sublattices.

For $v \in L$, its {\em norm} is the quantity $\norm{v}=v\cdot v$. We say that $v$ is {\em irreducible} if we cannot find $x,y \in L \setminus \{0\}$ such that $v=x+y$ and $x\cdot y \geq 0$.

\section{Changemaker lattices}\label{Ch3:sec:CMlattices}
In this section, we will explore the properties of changemaker lattices. The definition of a $p/q$-changemaker lattice was given in Definition~\ref{Intro:def:CMlattice}. For $p/q=[a_0, \dots, a_l]^-$, a $p/q$-changemaker lattice $L$ is an orthogonal complement
\[L=\langle w_0, \dots, w_l \rangle^\bot \subseteq \Z^N,\]
where the $w_i$ satisfy, among other properties,
\[
w_i\cdot w_j =
  \begin{cases}
   a_i      & \text{if } i=j \\
   -1       & \text{if } |i-j|=1\\
   0        & \text{if } |i-j|>1.
  \end{cases}
\]
In order to work effectively with changemaker lattices we will choose an orthonormal basis for $\Z^N$ as follows.

If $p/q \in \Z$, then we will choose an orthonormal basis $e_1, \dots, e_t$ such that
\[w_0= \sigma_1 e_1 + \dots + \sigma_t e_t,\]
where $1=\sigma_1 \leq \dots \leq \sigma_t$ are the changemaker coefficients of $L$ and $t=N$.

If $p/q \not\in \Z$, then we will choose an orthonormal basis $e_1, \dots, e_t, f_0, \dots, f_s$ such that
\[w_0=f_0 + \sigma_1 e_1 + \dots + \sigma_t e_t,\]
where $1=\sigma_1 \leq \dots \leq \sigma_t$ are the changemaker coefficients and for $k\geq 1$,
\[w_k=-e_{\alpha_{k-1}}+e_{\alpha_{k-1}+1}+ \dots + e_{\alpha_{k}},\]
where $\alpha_0=0$ and $\alpha_k=\sum_{i=1}^k(a_i-1)$ for $1\leq k\leq l$. In this case $N=s+t+1$ and $s=\alpha_l$.

\begin{defn}
We say that the changemaker coefficient $\sigma_k$ is {\em tight} if
\[\sigma_k=1+\sigma_1+ \dots + \sigma_{k-1}.\]
We say that $L$ is {\em tight} if $\sigma_k$ is tight for some $k>1$, otherwise we say that $L$ is {\em slack}.
\end{defn}
\begin{rem} Although the stable coefficients determine a $p/q$-changemaker lattice up to isomorphism, they are not invariants of the lattice. In Example~\ref{Ch2:exam:differentcoef}, we saw that the two 21-changemaker lattices with stable coefficients $(4)$ and $(2,2,2,2,2)$ are isomorphic.
\end{rem}

\subsection{Fractional and integral parts}\label{Ch3:sec:fracparts}
In this section, we study the properties of non-integer changemaker lattices. Throughout this section $L$ will be the $p/q$-changemaker lattice,
\[L=\langle w_0=e_0+\sigma_1 f_1 + \dots + \sigma_t f_t, w_1, \dots , w_l\rangle^\bot \subseteq \mathbb{Z}^{t+s+1}= \langle e_1, \dots, e_t, f_0, \dots, f_s \rangle,\]
where $p/q=n-r/q= [a_0,\dots, a_l]^-$ and $q/r= [a_1,\dots, a_l]^-$. We will also assume from now on that $p>q>1$. The case $p<q$ is a degenerate case, and, as the following remark justifies, ignoring it does not result in any serious loss of generality in what follows.

\begin{rem}\label{Ch3:rem:p<q}
If $p/q<1$, then $\norm{w_0}=1$ and $w_0=e_0$. Thus any $x \in L$ satisfies $x\cdot e_0=0$. In particular, we see that
\[L=\langle w_1+e_0, \dots, w_l\rangle^\bot \subseteq \Z^{s}=\langle e_1, \dots, e_s \rangle.\]
Since $[a_1-1,\dots, a_l]^-=(q-r)/r=p/(q-p)$, this shows that such an $L$ is isomorphic to the $\frac{p}{q-p}$-changemaker with trivial stable coefficients. By applying this observation several times if necessary, we see that $L$ is isomorphic to the $p/q'$-changemaker lattice with trivial stable coefficients, where $q'$ satisfies $1\leq q' <p$ and $q'\equiv q \bmod p$.
\end{rem}

It will be convenient to decompose $L$ into its fractional and integral parts.
\begin{defn}
The {\em fractional part} $L_F$ is defined to be the lattice
\[L_F=\langle w_1,\dots, w_l \rangle^\bot \cap \langle e_0, \dots , e_s \rangle.\]
The {\em integral part} $L_I$ is defined to be
\[L_I=\langle w_0 \rangle^\bot \cap \langle e_0, f_1, \dots , f_t \rangle.\]
\end{defn}

\begin{rem}Note that the following are true.
\begin{enumerate}
\item In general, neither $L_F$ nor $L_I$ is a sublattice of $L$.
\item The fractional part $L_F$ is isomorphic to the $q/r$-changemaker with trivial stable coefficients.
\item The integral part $L_I$ is isomorphic to the $n$-changemaker lattice with the same stable coefficients as $L$.
\item For any $x \in L$ there are unique vectors $x_I\in L_I$ and $x_F \in L_F$, such that $x=x_I + x_F - (x\cdot e_0)e_0$. It is a straightforward calculation that for any $x,y \in L$, we can compute their pairing via the formula:
\begin{equation}\label{Ch3:eq:splittingproduct}
x\cdot y = x_I \cdot y_I + x_F \cdot y_F- (x\cdot e_0) (y\cdot e_0).
\end{equation}
\end{enumerate}
\end{rem}

Next we will construct a basis for the fractional part of $L$. Recall the integers $\alpha_k=\sum_{i=1}^k(a_i-1)$ occurring in our choice of orthonormal basis. Consider the set $M=\{0, \dots, s \} \setminus \{\alpha_0, \dots , \alpha_{l}\}$. We can write $M=\{\beta_1,\dots, \beta_{m}\}$, where
\[\min M=\beta_1 < \dots < \beta_{m}=\max M\]
and $|M|=m$. Define also $\beta_{m+1}=s$. For $0\leq i \leq  m$ set
\[\mu_i=\begin{cases}
e_0+\dots + e_{\beta_1} &\text{if $i=0$}\\
-e_{\beta_k}+e_{\beta_k+1}+ \dots + e_{\beta_{k+1}} &\text{if $1\leq i \leq m$.}
\end{cases}
\]

First observe that the $\alpha_k$ and $\beta_k$ in the definition of the $w_i$ and $\mu_j$ are such that for all $i\geq 1$ and all $j\geq 0$,
\[
\{k\,|\, w_i\cdot e_k\ne 0\}  \cap \{k\,|\, \mu_j\cdot e_k\ne 0\}=
\begin{cases}
\emptyset &\text{if $\beta_{j+1}<\alpha_{i-1}$ or $\beta_{j-1}>\alpha_i$,}\\
\{\alpha_{i-1},\alpha_{i-1}+1\} &\text{if $\beta_{j}<\alpha_{i-1}<\beta_{j+1}$,}\\
\{\beta_{j},\beta_{j}+1\} &\text{if $\alpha_{i-1}<\beta_{j}<\alpha_{i}$}.\\
\end{cases}
\]
It follows that $w_i\cdot \mu_j=0$ for all $i\geq 1$ and all $j$, and hence that $\mu_i \in L_F$ for all $i$. Observe further that for all $j\geq 1$, we have $\mu_j \cdot w_0=0$. Thus it follows that $\mu_i\in L$ for all $i\geq 1$.

By construction, the $\mu_i$ satisfy
\[
\mu_i\cdot \mu_j=
\begin{cases}
\norm{\mu_j} &\text{if $i=j$}\\
-1 &\text{if $|i-j|=1$}\\
0 &\text{if $|i-j|>1$.}
\end{cases}
\]
\begin{example}\label{Ch3:example:3over5}
Since $107/5$ has continued fraction $107/5=[22,2,3]^-$ and $(1,2,4)$ satisfies the changemaker condition, one example of a $107/5$-changemaker lattice is
\[L=\langle 4f_3 + 2f_2+ f_1 + e_0, -e_0+e_1,-e_1+e_2+e_3 \rangle^\bot \subseteq \mathbb{Z}^{7}=\langle e_0, e_1, e_2, e_3, f_1, f_2,f_3 \rangle.\]
In this case, the fractional part is
\[L_F= \langle -e_0+e_1,-e_1+e_2+e_3\rangle^\bot \subseteq \langle e_0, e_1, e_2, e_3\rangle,\]
and the $\mu_i$ arising by the above construction are $\mu_0=e_0+e_1+e_2$ and $\mu_1=-e_2+e_3$.
The integral part of $L$ is the lattice
\[ L_I = \langle 4f_3 + 2f_2+ f_1 + e_0 \rangle^\bot \subseteq \langle e_0, f_1, f_2, f_3 \rangle.\]
\end{example}

Now we check that the $\mu_i$ form a basis for $L_F$.
\begin{lem}\label{Ch3:lem:LFbasis}
The set $\{\mu_0,\dots, \mu_m\}$ forms a basis for $L_F$.
\end{lem}
\begin{proof}
For any integer combination $x=\sum a_i \mu_i$, with at least one $a_i$ non-zero, we can consider $m'=\max\{i\,|\, a_i\ne 0\}$. Since this satisfies $x\cdot e_{\beta_{m'+1}}=a_{m'}\ne 0$ showing that $x\ne 0$. Therefore, the $\mu_i$ are linearly independent.

It remains to check that the $\mu_i$ span $L_F$. Suppose that we have non-zero $x \in L_F$. We will prove by induction on $k=\max\{i \,|\, x\cdot e_i \ne 0\}$ that $x$ is in the span of the $\mu_i$. Since $x \cdot w_i=0$ for all $i\geq 1$, it follows that $k\ne \alpha_j$ for any $0\leq j<l$, otherwise we would have $x\cdot w_{j+1}=-x\cdot e_k\ne 0$. Thus $k=\beta_{j+1}$ for some $0\leq j\leq m$. If we consider $x'=x-(x\cdot e_k)\mu_j$, then this satisfies $x'=0$ or $\max\{i \,|\, x'\cdot e_i \ne 0\}<k$. In either case, we can assume that $x'$ is in the span of the $\mu_i$. It follows that $x$ is also in the span of the $\mu_i$.
\end{proof}

The following continued fraction calculation will be useful later.
\begin{lem}\label{Ch3:lem:contfractransform}
The $\mu_i$ satisfy the continued fraction
\[[\norm{\mu_0}, \dots , \norm{\mu_{m}}]^-=\frac{q}{q-r}.\]
\end{lem}

\begin{proof}
We can prove this by induction on the size of $q$. If $q=2$, then $\mu_1=e_0+e_1$, so the continued fraction in question is simply $[2]^-=2$.

If $q/r=[a_1, \dots, a_l]^-$ with $a_1=2$, then consider
$q'/r'=[a_2, \dots, a_l]^-$. In this case, we have $\norm{\mu_0}>2$ and if we consider the fractional part of a $(n-\frac{r'}{q'})$-changemaker lattice, then it has a basis $\{\mu_0', \dots, \mu_{m'}'\}$, where $\norm{\mu_0'}=\norm{\mu_0}-1$ and $\norm{\mu_i'}=\norm{\mu_i}$ for $i>0$.
That is
\[\frac{q}{r}=[2, a_2, \dots, a_l]^-=\frac{2q'-r'}{q'}\]
and, using the inductive hypothesis,
\begin{align*}
[\norm{\mu_0}, \dots , \norm{\mu_{m}}]^-&=[1+\norm{\mu'_0}, \dots , \norm{\mu'_{m}}]^-\\
&=1+\frac{q'}{q'-r'}=\frac{2q'-r'}{q'-r'}\\
&=\frac{q}{q-r},
\end{align*}
as required.

If $q/r=[a_1, \dots, a_l]^-$ with $a_1>2$, then consider $q'/r'=[a_1-1, \dots, a_l]^-$. In this case, we have $\norm{\mu_0}=2$ and if we consider the fractional part of a $(n-\frac{r'}{q'})$-changemaker lattice, then it has a basis $\{\mu_0', \dots, \mu_{m-1}'\}$, where $\norm{\mu_i'}=\norm{\mu_{i+1}}$ for all $i$.
That is
\[\frac{q}{r}=1+\frac{q'}{r'}=\frac{r'+q'}{r'},\]
and, using the inductive hypothesis,
\begin{align*}
[\norm{\mu_0}, \dots , \norm{\mu_{m}}]^-&=[2,\norm{\mu'_0}, \dots , \norm{\mu'_{m-1}}]^-\\
&=2-\frac{q'-r'}{q'}=\frac{q'+r'}{q'}\\
&=\frac{q}{q-r},
\end{align*}
as required.
\end{proof}

\subsection{Standard bases}\label{Ch3:sec:standardbases}
In this section, we construct a basis for a $p/q$-changemaker lattice $L$. The changemaker conditions show that if $k\geq 2$ and $\sigma_k$ is not tight then there is a subset $A_k \subseteq\{1,\dots, k-2\}$ such that
$\sum_{i\in A_k}\sigma_i = \sigma_{k}-\sigma_{k-1}$. Choosing such an $A_k$ whenever $\sigma_k$ is not tight we can define $\nu_k\in L$ by
\[\nu_k=
\begin{cases}
-f_k + f_{k-1} + \sum_{i \in A_k}f_i &\text{if $\sigma_k$ is not tight}\\
-f_k + f_{k-1}+\dots + f_1 +\mu_0 &\text{if $\sigma_k$ is tight,}
\end{cases}
\]
when $p/q$ is not an integer, and
\[\nu_k=
\begin{cases}
-f_k + f_{k-1} + \sum_{i \in A_k}f_i &\text{if $\sigma_k$ is not tight}\\
-f_k + f_{k-1}+ \dots + 2f_1 &\text{if $\sigma_k$ is tight,}
\end{cases}
\]
when $p/q$ is an integer. Given such a collection of $\nu_k$, we will call the set
\[
S=\begin{cases}
\{\nu_1, \dots, \nu_t,\mu_1, \dots, \mu_m\} &\text{if $p/q\not\in \Z$}\\
\{\nu_2, \dots, \nu_t\} &\text{if $p/q\in \Z$}\\
\end{cases}
\]
a {\em standard basis} for $L$.

\begin{prop}\label{Ch3:prop:stdbasis}
A standard basis $S$ is a basis for $L$.
\end{prop}
\begin{proof}
We give a proof when $p/q$ is not an integer. This can easily be modified to the integer case.

First observe that the standard basis elements are linearly independent. Suppose that we we have $x=\sum c_i \nu_i + \sum b_j \mu_j$. If there is some $c_i\ne 0$, then we can consider $M=\max \{i \,|\, c_i\ne 0\}$. This satisfies $x\cdot f_M=-c_M\ne 0$. Similarly if there is some $b_j\ne 0$, then we can consider $M=\max \{j \,|\, b_j\ne 0\}$. This satisfies $x\cdot e_{\beta_{M+1}}=b_M\ne 0$. In either case if there is a non-zero $c_i$ or $b_j$, then $x$ is non-zero, showing the linear independence of the standard basis elements.

It remains to check that the standard basis spans $L$. As $L$ is a non-integer changemaker lattice, we may write any $x\in L$ as $x=x_I +x_F - (x\cdot e_0)e_0$, where $x_I \in L_I$ and $x_F\in L_F$. We have $x\in L\cap L_F$ if and only if $x_I=0$. Hence if $x_I=0$, then $x$ lies in the span of the $\mu_1, \dots, \mu_m$, by Lemma~\ref{Ch3:lem:LFbasis}. Consider $x\in L$, with $x_I$ non-zero. In this case, we will proceed by induction on the quantity $k=\max \{i\,|\, x\cdot f_i \ne 0\}$. Consider the vector $x'=x+(x\cdot f_k)\nu_k$. Either $x_I'=0$ or $\max \{i\,|\, x'\cdot f_i \ne 0\}<k$. In either case, we can assume that $x'$ lies in the span of the standard basis and hence so does $x$. Thus the standard basis spans $L$.
\end{proof}

The following lemma shows that if a lattice has a basis which ``looks like'' a standard basis for a half-integer changemaker lattice, then the lattice is a half-integer changemaker lattice.

\begin{lem}\label{Ch3:lem:halfintidentification}
Suppose we have a collection of vectors
\[\{v_1, \dots , v_t\} \subseteq \mathbb{Z}^{t+2}=\langle f_1, \dots, f_t, e_0,e_1 \rangle\]
such that $v_1=-f_1+e_0+e_1$ and for $k>1$,
$v_k=-f_k+f_{k-1}+\sum_{i\in A_k} f_i +\varepsilon_k(e_0+e_1)$ with $A_k \subseteq \{1,\dots , k-2\}$ and $\varepsilon_k \in \{0,1\}$, then the lattice spanned by $\{v_1, \dots , v_t\}$ is a half-integer changemaker lattice.
\end{lem}
\begin{proof}
Define $\sigma_i$ inductively by $\sigma_1=1$ and $\sigma_k=\sigma_{k-1}+\varepsilon_k +\sum_{i\in A_k} \sigma_i$ for $k\geq 2$. For each $i\geq 2$, we have
\[\sigma_{i-1} \leq \sigma_i \leq 1+\sigma_1 + \dots + \sigma_{i-1}.\]
By Proposition~\ref{Ch3:prop:CMcondition}, the $\sigma_i$ satisfy the changemaker conditions. If we take $w_1=e_1-e_0$ and $w_0=e_0 + \sigma_1 f_1 + \dots + \sigma_t f_t$, then
\[L=\langle w_0,w_1 \rangle^\bot\subseteq \Z^{t+2}\]
is a half-integer changemaker lattice. By construction, $v_k$ satisfies $w_0\cdot v_k=w_1\cdot v_k=0$ for all $k$. Therefore $v_k\in L$ for all $k$. In fact, we see that the $v_k$ provides an example of a standard basis for $L$. Thus by Proposition~\ref{Ch3:prop:stdbasis}, $L$ is the lattice spanned by the $v_k$.
\end{proof}

\subsection{Irreducibility}
Now we wish to study the irreducibility of certain vectors in $L$ and $L_F$. Recall that $z \in L$ is irreducible if it cannot be written in the form $z=x+y$ for non-zero $x$ and $y$ in $L$ with $x\cdot y\geq 0$.

The following observation on irreducible vectors of sublattice of diagonal lattices will be useful in what follows.
\begin{prop}\label{Ch3:prop:genirredcondition}
Consider $\Z^n$ with orthonormal basis $e_1,\dots, e_n$. Let $\Lambda \subseteq \Z^n$ be a sublattice. Consider a vector $z\in \Lambda$ which can be written in the form $z=\sum_{i\in A} e_i - \sum_{j\in B}e_j$ for disjoint $A,B \subseteq \{1,\dots, n\}$. If $z=x+y$ for some $x,y\in \Lambda$ with $x\cdot y\geq 0$, then there are subsets $A'\subseteq A$ and $B'\subseteq B$, such that
\[
x=\sum_{i\in A'} e_i - \sum_{j\in B'}e_j
\quad\text{and}\quad
y=\sum_{i\in A\setminus A'} e_i - \sum_{j\in B\setminus B'}e_j.
\]
\end{prop}
\begin{proof}
If such an $x$ and $y$ exist, then
\begin{equation}\label{Ch3:eq:genirred}
x\cdot y = \sum_{i\in A}^n (x\cdot e_i)(1-x \cdot e_i)+ \sum_{i\in B}^n -(x\cdot e_i)(1+x \cdot e_i) +\sum_{i\not\in A \cup B}^n -(x\cdot e_i)^2\geq 0.
\end{equation}
However every summand is \eqref{Ch3:eq:genirred} is at most zero, so it must be exactly zero. In particular $x\cdot e_i \in \{0,1\}$ for all $i\in A$, $x\cdot e_i \in \{0,-1\}$ for all $i\in B$ and $x\cdot e_i=0$ for all $i\not\in A\cup B$. Thus we see that $x=\sum_{i\in A'} e_i - \sum_{j\in B'}e_j$ for $A'=\{i \,|\, x\cdot e_i =1\}\subseteq A$ and $B'=\{i \,|\, x\cdot e_i =-1\}\subseteq B$. Since $y=z-x$, this also gives the statement about $y$.
\end{proof}
\begin{lem}\label{Ch3:lem:suppfacts}
Let $L$ be a $p/q$-changemaker lattice. If $x\in L$ satisfies $x\cdot f_i \geq 0$ and $x\cdot e_j \geq 0$ for all $i$ and $j$, then $x=0$.
\end{lem}
\begin{proof}
Suppose $x\in L$ satisfies $x\cdot f_i \geq 0$ and $x\cdot e_j \geq 0$ for all $i$ and $j$.

If $L$ is an integer changemaker lattice, then
\[0=x\cdot w_0 =\sum_{i=1}^t \sigma_i (x\cdot f_i)\geq 0.\]
This shows that $x\cdot f_i=0$ for all $i$ and hence that $x=0$.

If $L$ is a non-integer changemaker lattice, then it follows from $0= x\cdot w_0$ that
\[0\leq x\cdot e_0 =- \sum_{i=1}^t \sigma_i (x\cdot f_i)\leq 0.\]
This shows that $x\cdot e_0=0$ and $x\cdot f_i=0$ for all $i$. For any $1\leq j\leq s$, there is a unique $k$ such that $e_j\cdot w_k=1$. Observe that $W=w_0+\dots + w_k$ satisfies $W\cdot e_i\in\{0,1\}$ for all $i$ and $W\cdot e_j=1$. Thus we see that
 \[0=W\cdot x\geq x\cdot e_j \geq 0.\]
 This shows that $x\cdot e_j=0$ for all $j$. Thus $x=0$, as required.
\end{proof}

\begin{lem}\label{Ch3:lem:generalirred}
Let $L$ be a $p/q$-changemaker lattice. If $z\in L$ takes the form
\begin{enumerate}[(i)]
\item $z=-f_k + \sum_{i \in A} f_i + \sum_{j\in B} e_j$ or
\item $z=-e_k +\sum_{j\in C} e_j$,
\end{enumerate}
for some subset $A\subseteq \{1, \dots, t\}$, $B \subseteq \{0, \dots, s\}$
or $C\subseteq \{k+1,\dots, s\}$, then $z$ is irreducible.
\end{lem}
\begin{proof}
Suppose we can write such a $z$ as $z=x+y$ for some $x,y\in L$ with $x\cdot y\geq 0$. There is a single $f\in \{e_0, \dots, e_s, f_1, \dots, f_t\}$ such that $z\cdot f=-1$. Without loss of generality we may assume that $x\cdot f<0$. However, it then follows from Proposition~\ref{Ch3:prop:genirredcondition} that $y\cdot f'\geq 0$ for all $f' \in \{e_0, \dots, e_s, f_1, \dots, f_t\}$. Thus by Lemma~\ref{Ch3:lem:suppfacts}, we have $y=0$, showing that $z$ is irreducible.
\end{proof}
The other irreducible vectors we will be interested in are the following.
\begin{lem}\label{Ch3:lem:tightstdvecirred}
If $L$ is a $p/q$-changemaker lattice with $q\leq 2$ and some changemaker coefficient $\sigma_k$ is tight, then
\[z=-f_{k}+f_{k-1}+\dots+ f_2+2f_1\in L\]
is irreducible.
\end{lem}
\begin{proof}
We give a proof when $q=2$. The same argument works when $q=1$.

If $L$ is a half-integer changemaker lattice, then it takes the form
\[L=\langle w_0, e_1-e_0 \rangle^\bot \subseteq \Z^{t+2}\]
Suppose we can write $z$ as $z=x+y$ with $x,y\in L$ and $x\cdot y \geq 0$. This gives
\begin{align}\begin{split}\label{Ch3:eq:irredsum2}
x \cdot y &= \sum_{i=1}^t (x\cdot f_i)(y\cdot f_i) + 2(x\cdot e_0)(y\cdot e_0)\\
&= -\sum_{i=k+1}^t (x\cdot f_i)^2 +\sum_{i=2}^{k-1} (x\cdot f_i)(1-x\cdot f_i) \\
&\quad +(x\cdot f_1)(2-x\cdot f_1) -(x\cdot f_k)(1+x\cdot f_k)  -2(x\cdot e_0)^2\geq 0.
\end{split}\end{align}
Now the only term in \eqref{Ch3:eq:irredsum2} which can be strictly positive is then term $(x\cdot f_1)(2-x\cdot f_1)$ which is at most one. It follows that we must have $x\cdot e_0=y\cdot e_0=0$ and $x\cdot f_i\in \{0,1\}$ for $2\leq i <k$ and $x\cdot f_k\in \{0,-1\}$. We may assume that
$x\cdot f_k=-1$.

If $y$ is non-zero then we must have $y\cdot f_g<0$ for some $g$. Such a $g$ is necessarily unique with $y\cdot f_g=-1$ and $g>k$. If such a $g$ exists, then $(y\cdot f_g)(x\cdot f_g)=-1$ shows that we have $y\cdot f_1=x\cdot f_1=1$. Altogether this shows that there is $A\subseteq\{1,\dots, k-1\}$
such that
\[y=-f_g+ \sum_{i\in A}f_i.\]
However as $y\in L$, this shows
\[\sigma_g =  \sum_{i\in A} \sigma_i <\sigma_k,\]
which contradicts the requirement that $g>k$. Therefore $y=0$ and $z$ is irreducible, as required.
\end{proof}
\begin{rem}\label{Ch3:rem:stdbasisirred}
Lemma~\ref{Ch3:lem:generalirred} and Lemma~\ref{Ch3:lem:tightstdvecirred} shows that every element of a standard basis is irreducible.
\end{rem}

Next we classify the irreducible vectors in the fractional part of a changemaker lattice.
\begin{lem}\label{Ch3:lem:irredfracpart}
Let $L$ be a non-integer changemaker lattice. The vector $x_F \in L_F$ is irreducible if and only if $x_F$ is in the form
\[x_F=\pm (\mu_a + \dots + \mu_b),\]
for some $0\leq a \leq b \leq m$.
\end{lem}
\begin{proof}
Suppose that $x_F=\sum_{i=1}^{m}c_iv_i$ is irreducible. For convenience, we set $c_{m+1}=0$. Let $a$ be minimal such that $c_a\ne 0$. Since $x_F$ is irreducible if and only if $-x_F$ is irreducible, we may assume that $c_a>0$. Let $b\geq a$ be minimal such that $c_{b+1}\leq 0$. Let $z=\mu_a + \dots + \mu_b$. By direct computation, we obtain
\[(x_F - z) \cdot z = (c_a-1)\norm{\mu_a} + \sum_{i=a}^b (\norm{\mu_i}-2)(c_i-1) + (c_b-1)\norm{\mu_b} - c_{b+1}.\]
By the choice of $a$ and $b$, this shows that
\[(x_F - z) \cdot z \geq -c_{b+1}\geq 0.\]
As $x_F$ is irreducible, this implies $x_F=z$. Thus we see that any irreducible vector is of the stated form.

To prove the converse, consider $z=\mu_a + \dots + \mu_b$. If $a>0$, then $z\cdot e_{\beta_a}=-1$, $z\cdot e_i\in \{0,1\}$ for all $i\ne \beta_a$, and $z\cdot e_0=0$. If $a=0$, then $z\cdot e_i\in \{0,1\}$ for all $i$. In either case, if $z=x+y$ with $x,y\in L_F$ and $x\cdot y\geq0$ , then Proposition~\ref{Ch3:prop:genirredcondition} shows that either $x\cdot e_i\geq 0$ for all $i>0$ and $x\cdot e_0=0$, or $y\cdot e_i\geq 0$ for all $i>0$ and $y\cdot e_0=0$. However for any $v\in L_F$, $v\cdot e_0=0$ implies that $v\in L$. Thus Lemma~\ref{Ch3:lem:suppfacts} implies either $x=0$ or $y=0$. Therefore $z$ is irreducible, as required.
\end{proof}

\begin{lem}\label{Ch3:lem:fracpartirred}
Let $L$ be a non-integer changemaker lattice. If $x \in L$ is irreducible, then the fractional part $x_F\in L_F$ is irreducible.
\end{lem}
\begin{proof}
Let $x_0=x\cdot e_0$. We may write $x=x_F+x_I - x_0e_0$, with $x_I\in L_I$ and $x_F \in L_F$. If $x_0=0$, then $x_I,x_F \in L$ and $x_I \cdot x_F=0$. Since $x=x_I + x_F$, irreducibility of $x$ implies that $x_F=x$ or $x_F=0$.

Now we suppose $x_0\ne 0$. We may assume that $x_0>0$ and that $x_F=\sum_{i=1}^{m}c_iv_i$, where $c_0=x_0$. Let $m\geq g_F\geq 0$ be minimal such that $c_{g_F+1}\leq 0$. For convenience, we are using the convention that $c_{m+1}=0$. Now consider
\[z_F=\mu_0+ \dots + \mu_{g_F} \in L_F.\]
By Lemma~\ref{Ch3:lem:irredfracpart}, this is irreducible. We shall prove the lemma by showing $x_F=z_F$. First we need to bound the quantity $(x_F - z_F)\cdot z_F$. We have
\begin{align*}
(x_F - z_F)\cdot z_F &= \sum_{i=0}^{g_F}\mu_i \cdot (x_F-z_F) \\
    &= (\norm{\mu_0}-1)(c_0-1) + \sum_{i=1}^{g_F-1}(\norm{\mu_i}-2)(c_i -1) + (\norm{\mu_{g_F}}-1)(c_{g_F}-1) - c_{g_F+1}\\
    &\geq (\norm{\mu_0}-1)(c_0-1)\\
    &=(\norm{\mu_0}-1)(x_0-1).
\end{align*}
In particular, this yields the inequality
\begin{equation}\label{Ch3:eq:bound1}
(x_F - z_F)\cdot z_F \geq x_0-1.
\end{equation}
Now let $g_I$ be minimal such that $x\cdot f_{g_I}\leq 0$. There is $A \subseteq \{1, \dots, g_I - 1\}$ such that $\sigma_{g_I} - 1 =\sum_{i \in A} \sigma_i$. Hence, if we define
\[z_I=-f_{g_I} + e_0 + \sum_{i \in A}f_i,\]
then $z_I \in L_I$. Since
\begin{equation*}
(x_I-z_I)\cdot z_I = -x\cdot f_{g_I} - 1 + \sum_{i\in A}(x\cdot f_i -1) + (x_0-1),
\end{equation*}
and $x\cdot f_i\geq 1$ for all $i\in A$, we have the bound
\begin{equation}\label{Ch3:eq:bound2}
(x_I-z_I)\cdot z_I\geq x_0 -2.
\end{equation}
Now consider
\[z=z_I+z_F-e_0\in L.\]
Using \eqref{Ch3:eq:splittingproduct} along with the inequalities \eqref{Ch3:eq:bound1} and \eqref{Ch3:eq:bound2}, we have
\begin{align*}
(x-z)\cdot z &= (x_I-z_I)\cdot z_I - (x_0-1) + (x_F - z_F)\cdot z_F \\
    &\geq x_0-2 \geq -1.
\end{align*}
If $(x-z)\cdot z\geq 0$, then the irreducibility of $x$ implies that $x=z$. Otherwise the above inequality shows that $(x-z)\cdot z=-1$ and in particular that $x_0=1$. Thus we may consider $z'=x_I+z_F-e_0 \in L$. Since
\[(x-z')\cdot z'=(x_F-z_f)\cdot z_F \geq x_0 -1 =0,\]
it follows that $x=z'$. In either case, we have $x_F=z_F$, which implies irreducibility, as required.
\end{proof}

\subsection{Indecomposability}
Now we study the indecomposability of $p/q$-changemaker lattices. Recall that the lattice $L$ is indecomposable if it cannot be written as an orthogonal direct sum $L=L_1 \oplus L_2$ with $L_1,L_2 \ne 0$.
\begin{lem}\label{Ch3:lem:fracindecomp}
If $L$ is a non-integer changemaker lattice, then $L$ is indecomposable.
\end{lem}
\begin{proof}
Let $S=\{\nu_1, \dots, \nu_t,\mu_1, \dots, \mu_m\}$ be a standard basis for $L$. Suppose that $L=L_1 \oplus L_2$. Since every $x\in S$ is irreducible, we must have $x\in L_1$ or $x\in L_2$. We may assume that $\nu_1\in L_1$. Since $\nu_1 \cdot \mu_1 =-1$ and $\mu_{k}\cdot \mu_{k+1}=-1$ for all $1\leq k<m$, it follows that we must have $\{\nu_1,\mu_1, \dots, \mu_m\}\subseteq L_1$. Now consider $\nu_k$ for $k>1$. If $\sigma_k$ is tight, then $\nu_k=-f_k+f_{k-1}+\dots +f_1 + \mu_0$. In this case, $\nu_k\cdot \nu_1=\norm{\mu_0}-1\geq 1$ so we have $\nu_k \in L_1$. If $\sigma_k$ is not tight then $\nu_k=-f_k + \sum_{i\in A}f_i$ for some $A\subseteq \{1, \dots , k-1\}$. If $d=\min A$, then $\nu_k \cdot \nu_d =-1$. Since $d<k$, this allows us to prove by induction that $\nu_k\in L_1$. Thus we have $S\subseteq L_1$ and we can conclude that $L_1=L$. This shows that $L$ is indecomposable.
\end{proof}

\begin{example}\label{Ch3:exa:integerdecomp}
The requirement that $L$ is a non-integer changemaker lattice is necessary for Lemma~\ref{Ch3:lem:fracindecomp} to hold. One can easily construct examples of decomposable integer changemaker lattices. For example, we have
\[\langle f_1 + f_2 + 2f_3 \rangle^\bot = \langle f_1-f_2 \rangle \oplus \langle f_1+f_2-f_3 \rangle\subseteq \Z^3.\]
\end{example}

However in the case of integer lattices, we have the following.
\begin{lem}\label{Ch3:lem:reduciblebound}
If $L$ is a decomposable integer changemaker lattice with changemaker coefficients $1=\sigma_1 \leq \dots \leq \sigma_t$ and stable coefficients $2\leq \sigma_m \leq \dots \leq \sigma_t$, then $\sigma_m= m-1$.
\end{lem}
\begin{proof}
We will show that $L$ is indecomposable if $\sigma_m \ne m-1$. Let $S=\{\nu_2, \dots, \nu_t\}$ be a standard basis for $L$. By Remark~\ref{Ch3:rem:stdbasisirred}, every element of $S$ is irreducible. Suppose that $L=L_1 \oplus L_2$. As $x\in S$ is irreducible, we must have $x\in L_1$ or $x\in L_2$. We may assume that $\nu_2\in L_1$. Now we will show that if $\sigma_m\ne m-1$, then for every $k>2$, there is $l<k$ with $\nu_k \cdot \nu_l\ne 0$. Since $\nu_l\in L_1$ and $\nu_k \cdot \nu_l\ne 0$ implies $\nu_k\in L_1$. This allows us to deduce $S \subseteq L_1$ and hence that $L=L_1$, proving indecomposability.

Consider any $k>2$. We consider 3 possibilities separately. If $\sigma_k$ is tight, then $\nu_k=\nu_k=-f_k+f_{k-1}+\dots + f_2 +2f_1$ and we have $\nu_k \cdot \nu_2 \geq 1$. If $\sigma_k=\sum_{i=1}^{k-1} \sigma_i$, then $\nu_k=\nu_k=-f_k+f_{k-1}+\dots + f_2 +f_1$. Note that in this case we must have $k>m$, as we are assuming $\sigma_m\ne m-1$. Thus we have $\nu_k \cdot \nu_m\geq \norm{\nu_m}-2\geq 1$. It remains only to consider the case $\sigma_k<\sum_{i=1}^{k-1} \sigma_i$. In this case $\nu_k= -f_k + \sum_{i\in A} f_i$ for some $A \subsetneq \{1, \dots , k-1\}$. In particular, we must have either $\min A>1$ or $\min \{1, \dots , k-1\} \setminus A >1$. If $c=\min A>1$, then $\nu_c \cdot \nu_k=-1$. If $d=\min\{1, \dots , k-1\} \setminus A>1$, then since $\nu_k \cdot f_i=1$ for all $i<d$, we have $\nu_d \cdot \nu_k = \norm{\nu_d}-1>0$. This completes the proof.
\end{proof}

\section{Graph lattices}\label{Ch3:sec:graphlattices}
In this section, we collect the necessary results on graph lattices. The material in this section is largely based on that of Greene \cite[Section 3.2]{GreeneLRP}. There has, however, been some reworking since we find it convenient to avoid using the concept of a root vertex.

Let $G=(V,E)$ be a finite, connected, undirected graph with no self-loops. For a pair of disjoint subsets $R,S \subseteq V$, let $E(R,S)$ be the set of edges between $R$ and $S$. Define $e(R,S)=|E(R,S)|$. The {\em degree} of a vertex $v\in V$ is $d(v)=e(v,V\setminus \{v\})$.

\paragraph{} Let $\overline{\Lambda}(G)$ be the free abelian group generated by $v\in V$. Define a symmetric bilinear form on $\overline{\Lambda}(G)$ by
\[
v\cdot w =
  \begin{cases}
   d(v)          & \text{if } v=w \\
   -e(v,w)       & \text{if } v\ne w.
  \end{cases}
\]
We will use the notation $[R]=\sum_{v\in R}v$, for $R\subseteq V$. It follows from the above definition that
\[
v\cdot [R] =
  \begin{cases}
   -e(v,R)            & \text{if } v\notin R \\
   e(v,V\setminus R)  & \text{if } v\in R,
  \end{cases}
\]
for any $v\in V$ and any $R\subseteq V$. It follows that $[V]\cdot x= 0$ for all $x \in \overline{\Lambda}(G)$. If we define the {\em graph lattice} of $G$ to be
\[\Lambda(G):= \frac{\overline{\Lambda}(G)}{\mathbb{Z}[V]},\]
then the bilinear form on $\overline{\Lambda}(G)$ descends to $\Lambda(G)$. It follows from the assumption that $G$ is connected that the pairing on $\Lambda(G)$ is positive-definite\footnote{Taking $y=x=\sum_{v\in V}b_v v$ in \eqref{Ch3:eq:prod1} shows that $x\cdot x \leq 0$ if and only if $b_v$ is constant on the connected components of $G$. As $G$ is connected, this shows $x\cdot x\leq 0$ only if $x=0\in \Lambda(G)$.}. This makes $\Lambda(G)$ into an integer lattice. Henceforth, we will abuse notation by using $v$ to denote both the vertex $v \in V$ and its image in $\Lambda(G)$.

\paragraph{}We compute the product of arbitrary $x,y \in \Lambda(G)$. Let $x=\sum_{v\in V} b_v v$ and $y=\sum_{v\in V} c_v v$ be elements of $\Lambda(G)$. For any $v\in V$,
\[v\cdot y=c_v d(v) - \sum_{u\ne v} e(v,u)c_u = \sum_{u \in V} (c_v-c_u)e(v,u).\]
Therefore,
\[x \cdot y = \sum_{v \in V}b_v \sum_{u\in V}(c_v-c_u)e(v,u)=\sum_{u,v \in V}b_v(c_v-c_u)e(v,u).\]
Since we also have
\[x \cdot y =\sum_{u,v \in V}c_u(b_u-b_v)e(v,u),\]
we can express the pairing $x\cdot y$ as
\begin{equation}\label{Ch3:eq:prod1}
x \cdot y = \frac{1}{2}\sum_{u,v \in V}(c_v-c_u)(b_v-b_u)e(v,u).
\end{equation}
If $y=[R]-x$ for some $R\subseteq V$, then
\[
c_v =
  \begin{cases}
   1-b_v            & \text{if } v\in R \\
   -b_v       & \text{if } v\notin R.
  \end{cases}
\]
Therefore \eqref{Ch3:eq:prod1} shows that

\begin{equation}\label{Ch3:eq:usefulformula}
([R]-x)\cdot x =
\sum_{u\in R, v\in V\setminus R}b_{v,u} (1-b_{v,u})e(v,u)
-\frac{1}{2}\sum_{u,v\in R}b_{v,u}^2 e(v,u)
-\frac{1}{2}\sum_{u,v \in V\setminus R}b_{v,u}^2 e(v,u),
\end{equation}
where $b_{v,u}=b_v-b_u$. Examining each term in \eqref{Ch3:eq:usefulformula} individually, we see that the right hand side is at most zero. This inequality will be used so often that we will record it as the following lemma.
\begin{lem}\label{Ch3:lem:usefulbound}
Let $x=[R]$ be a sum of vertices, then for any $z\in \Lambda(G)$, we have
\[(x-z)\cdot z\leq 0.\]
\end{lem}
The irreducible vectors in $\Lambda(G)$ will be of particular interest.
\begin{lem}\label{Ch3:lem:irreducible}
The vector $x \in \Lambda(G)\setminus \{0\}$ is irreducible if and only if $x=[R]$ for some $R\subseteq V$ such that $R$ and $V\setminus R$ induce connected subgraphs of $G$.
\end{lem}
\begin{proof}
Write $x=\sum_{v\in V} a_v v$. Since $[V]=0$, we can assume the $a_v$ are chosen so that $\min_{v\in V} a_v = 0$.

Suppose $x\ne 0$ is irreducible. Let $a=\max_{v \in V}a_v\geq 1$ and $R=\{v|a_v=a\}$.
By direct computation, we get
\begin{align*}
(x-[R])\cdot [R] &= \sum_{v\in R} \bigg((a-1)d(v) -\sum_{u\in R\setminus \{v\}}(a-1) e(u,v)-\sum_{u\in V\setminus R}a_u e(u,v)\bigg)\\
&\geq (a-1)\sum_{v\in R} \bigg(d(v) -\sum_{u\in V\setminus \{v\}}e(u,v)\bigg)\\
&= 0.
\end{align*}
Combining this with the irreducibility of $x$, it follows that $x=[R]$. If the subgraph induced by $R$ is not connected there would be $S,T$ disjoint with $R=S\cup T$ and $e(S,T)=0$, but this would give $x=[S]+[T]$ and $[S]\cdot [T]=0$. This also shows the subgraph induced by $V\setminus R$ must be connected, since $x$ irreducible implies that $-x=[V]-x$ is irreducible.
\paragraph{} To show the converse, we apply \eqref{Ch3:eq:usefulformula} in the case where $R$ and $V\setminus R$ induce connected subgraphs. Let $y= \sum_{v\in V} c_v v$ be such that $([R]-y)\cdot y\geq 0$. By Lemma~\ref{Ch3:lem:usefulbound}, this means $([R]-y)\cdot y = 0$. In particular, using \eqref{Ch3:eq:usefulformula}, we have
\begin{equation*}
([R]-y)\cdot y=\sum_{u\in R, v \in V\setminus R} c_{u,v}(1-c_{u,v})e(u,v)-\frac{1}{2}\sum_{u,v \in V\setminus R}c_{u,v}^2e(u,v)-\sum_{u,v \in R}c_{u,v}^2 e(u,v)=0,
\end{equation*}
where $c_{u,v}=c_u-c_v$. This means every summand in the above equation must be 0. Thus if $u$ and $v$ are both in $R$ or $V\setminus R$ and $e(u,v)>0$, then $c_u=c_v$. Thus since $R$ and $V\setminus R$ induce connected subgraphs, $c_v$ is constant on $R$ and $V\setminus R$. If $u\in R$ and $v\in V \setminus R$, then $c_u=c_v$ or $c_u=1+c_v$. Thus if $c=c_v$ for some $v \in V\setminus R$, then $y=c[V]=0$ or $y=c[V\setminus R] + (c+1)[R]=[R]$. Thus $[R]$ is irreducible.
\end{proof}
Recall that a connected graph is {\em 2-connected} if it can not be disconnected by deleting a vertex. This property is characterised nicely in the corresponding graph lattice.
\begin{lem}\label{Ch3:lem:2connectgraphlat}
The following are equivalent:
\begin{enumerate}[(i)]
\item The graph $G$ is 2-connected;
\item Every vertex $v\in V$ is irreducible;
\item The lattice $\Lambda(G)$ is indecomposable.
\end{enumerate}
\end{lem}
\begin{proof} The equivalence $(i)\Leftrightarrow (ii)$ follows from Lemma \ref{Ch3:lem:irreducible}.

To show $(ii) \Rightarrow (iii)$, suppose that we may write $\Lambda(G)$ as the orthogonal direct sum $\Lambda(G)= L_1 \oplus L_2$. For any irreducible $x\in \Lambda(G)$, we must have $x\in L_1$ or $x\in L_2$. Therefore if every $v\in V$ is irreducible, then the sets $R_i=\{v \in V | v\in L_i\}$ partition $V$. Therefore $[R_1]+[R_2]=[V]=0$ and $[R_1]\cdot [R_2]=0$. This implies that $[R_1]=[R_2]=0$, and hence that the partition is trivial, i.e. we have either $R_1=V$ or $R_2=V$. Therefore $L_1=\Lambda(G)$ or $L_2=\Lambda(G)$ and $\Lambda(G)$ is indecomposable.

Now we prove the implication $(iii)\Rightarrow (i)$. Suppose $G$ is not 2-connected, so there is a vertex $v$ such that $G\setminus\{v\}$ is disconnected. Suppose that $G\setminus\{v\}$ has non-empty components $G_1$ and $G_2$. Then we have $\Lambda(G)=L_1\oplus L_2$, where $L_1$ and $L_2$ are the sublattices spanned by the vertices of $G_1$ and $G_2$ respectively. This shows that if $G$ is not 2-connected, then $\Lambda(G)$ is decomposable.
\end{proof}

Recall that an edge, $e$, is a {\em cut-edge} if $G\setminus \{e\}$ is disconnected.

\begin{lem}\label{Ch3:lem:norm1vectors}
The lattice $\Lambda(G)$ contains a vector of norm one if and only if $G$ contains a cut edge.
\end{lem}
\begin{proof}
Suppose $\Lambda(G)$ contains a vector $z$ with $\norm{z}=1$. Since any vector of norm one is irreducible, Lemma~\ref{Ch3:lem:irreducible} shows that there is a subgraph $G_1$ such that $z=[G_1]$. Since $\norm{z}=e(G_1,G\setminus G_1)=1$, we see there is a single edge between $G_1$ and $G\setminus G_1$. This edge is necessarily a cut-edge. Conversely, suppose there is a cut-edge in $G$. This means there is a subgraph $G_1$ such that $e(G_1,G\setminus G_1)=1$. If $z=[G_1]$, then
\[\norm{z}=e(G_1,G\setminus G_1)=1,\]
giving a vector of norm 1 in $\Lambda(G)$.
\end{proof}

\begin{lem}\label{Ch3:lem:cutedge}
Suppose that $G$ is contains no cut-edges and there is an irreducible vertex $v$ such that we can find $x,y\in \Lambda(G)$, with $v=x+y$ and $x\cdot y=-1$. Then there is a cut edge $e$ in $G\setminus \{v\}$ and if $R,S$ are the vertices of the two components of $(G\setminus \{v\})\setminus\{e\}$ then $\{x,y\}=\{[R]+v,[S]+v\}$. Furthermore, there are unique vertices $u_1,u_2\ne v$, with $x\cdot u_1=y\cdot u_2=1$, and any vertex $w \notin \{v,u_1,u_2\}$ satisfies $w\cdot x,w\cdot y\leq 0$.
\end{lem}
\begin{proof}
We will use the irreducibility of $v$ to show that $x$ and $y$ are irreducible. Suppose that we can write $x=z+w$ with $z\cdot w\geq 0$. Since $x\cdot y =-1$, it follows that $z\cdot y\geq 0$ or $w\cdot y\geq 0$. Without loss of generality, we may assume that $w\cdot y\geq 0$ and hence that $w\cdot (y+z)\geq 0$. Using the irreducibility of $v$, this implies either $w=0$ or $y+z=0$. If we assume $w\ne 0$, then it follows that $w=v$ and $y=-z$. Thus from $x\cdot y=-1$ and $w.z\geq 0$, we obtain $z\cdot (v+z)=1$ and  $v.z\geq 0$. Combining these shows that $\norm{z}\leq 1$. As $G$ contains no cut-edges, Lemma~\ref{Ch3:lem:norm1vectors} shows that $\Lambda(G)$ does not contain any vectors of norm 1, so we can conclude that $z=0$. Therefore we have shown that $z=0$ or $w=0$, so we have shown that $x$ is irreducible. An identical argument shows that $y$ is irreducible.

\paragraph{} As $x$ and $y$ are irreducible, Lemma \ref{Ch3:lem:irreducible} shows that they take the form $x=[R']$ and $y=[S']$, for some $R',S' \subseteq V$ satisfying
\[v=[R']+[S']=[R' \cup S'] + [R' \cap S'].\]
Since $x,y \ne 0$, we must have $R' \cup S'=V$ and $R' \cap S'=\{v\}$. Thus, $x=v+[R]$ and $y=v+[S]$, where $R$ and $S$ are disjoint and $R\cup S=V\setminus\{v\}$. Moreover, using irreducibility and Lemma \ref{Ch3:lem:irreducible} again, it follows that  subgraphs induced by $R,S, R\cup\{v\}$ and $S\cup\{v\}$ must all be connected. Since $[R]+[S]=[V]-v=-v$, it follows that
\begin{equation*}
x\cdot y=\norm{v}+v\cdot ([R]+[S]) + [R]\cdot[S] =[R]\cdot[S]=-e(R,S)=-1.
\end{equation*}
This gives a unique edge from $R$ to $S$. The rest of the lemma follows by taking $u_1\in R$ and $u_2\in S$ to be the end points of this edge.
\end{proof} 
\chapter{Alternating diagrams and tangles}\label{chap:altdiags}
This chapter is spent developing the background material on alternating diagrams and tangles that we will require in subsequent chapters.
\section{The Goeritz form}\label{sec:altdiagrams}
Given a diagram $D$ of a knot or link $L$, we get a division of the plane into connected regions. These regions can be coloured black and white in a chessboard manner. There are two possible choices of colouring and both of these give an incidence number, $\mu(c)\in \{\pm 1\}$, for each crossing $c$ of $D$ according to the convention in Figure~\ref{Ch4:fig:incidencenumber}.
 \begin{figure}[ht]
  \centering
  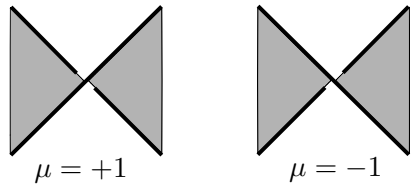
 \caption{The incidence number of a crossing.}
 \label{Ch4:fig:incidencenumber}
\end{figure}
We construct a planar graph, $\Gamma_D$, by drawing a vertex in each white region and an edge $e$ for every crossing $c$ between the two white regions it joins and we define $\mu(e):=\mu(c)$. We call this the {\em white graph} corresponding to $D$. This gives rise to a {\em Goeritz matrix} $G_D=(G_{ij})$ defined by labeling the vertices of $\Gamma_D$ by $v_1,\dots , v_{r+1}$ and, for $1\leq i,j \leq r$, setting \cite[Chapter~9]{lickorish1997introduction}
\[g_{ij}=
\begin{cases}
\sum_{e \in E(v_i,v_j)}\mu(e) & i\ne j\\
- \sum_{e \in E(v_i, \Gamma_D\setminus v_i)}\mu(e) & i=j.
\end{cases}
\]
Although the Goeritz matrix depends on the choice of diagram, its determinant does not. Hence we can define the {\em determinant} of $L$ by
 \[\det L= |\det G_D|.\]
 The other invariant we wish to compute from the Goeritz matrix is the signature.
 \begin{thm}[Gordon and Litherland \cite{gordon1978signature}]\label{Ch4:thm:sigformula}
 Let $D$ be a diagram for a link $L$ that is shaded and oriented so that there are $n_+$ positive crossings of incidence -1 and $n_-$ negative crossings of incidence +1. Let $\operatorname{sig}(G_D)$ be the signature of the Goeritz matrix $G_D$. Then the signature can be computed by the formula,
 \begin{equation}\label{Ch4:eq:fullsigformula}
\sigma(L)=\operatorname{sig}(G_D) + n_- -n_+.
\end{equation}
\end{thm}

\paragraph{}Now suppose that $L$ is an alternating, non-split link, with a reduced alternating diagram $D$. We may fix the colouring so that $\mu(c)=-1$ for all crossings. For such a diagram, the white graph $\Gamma_D$ is connected and contains no self-loops or cut-edges. In this case, the pairing given by $G_D$ defines a lattice $\Lambda_D$ which is isomorphic to the graph lattice $\Lambda(\Gamma_D)$. Under these assumptions \eqref{Ch4:eq:fullsigformula} becomes
\begin{equation}\label{Ch4:eq:altsigformula}
\sigma(L)= \rk (\Lambda_D) - n_+,
\end{equation}
where $n_+$ is equal to the number of positive crossings.

\paragraph{} In proving Theorem~\ref{Ch4:thm:sigformula}, Gordon and Litherland showed that for any diagram $D$ of a knot $K$, there is a 4-manifold with boundary $\Sigma(K)$ and intersection form given by the Goeritz matrix $G_D$ \cite{gordon1978signature}. Ozsv{\'a}th and Szab{\'o} showed that if $D$ is a non-split alternating link, then this manifold is sharp.

\begin{thm}[Ozsv{\'a}th-Szab{\'o}, \cite{ozsvath2005heegaard}]\label{Ch4:thm:altboundssharp}
Let $L$ be a non-split alternating link with a reduced alternating diagram $D$. The double branched cover $\Sigma(\overline{L})\cong-\Sigma(L)$ is an $L$-space and it bounds a sharp 4-manifold $X_D$ with intersection form isomorphic to $-\Lambda_D$.
\end{thm}
Combining this with the results of the first chapter, we obtain the following corollary.
\begin{cor}\label{Ch4:cor:CMcondit}
Let $\kappa\subseteq S^3$ be a knot such that $S_{-p/q}^3(\kappa)\cong \Sigma(L)$ for some $p/q>0$ and some alternating link $L$. If $D$ is a reduced alternating diagram for $L$ and $n=\lceil p/q \rceil$, then $\Lambda_D$ is isomorphic to the $p/q$-changemaker lattice
\[L=\langle w_0, \dots, w_l \rangle^\bot \subseteq \Z^N,\]
such that the torsion coefficients of $\Delta_\kappa(t)$ satisfy
\begin{equation}\label{Ch4:eq:torsionformula}
8t_{i}(\kappa) = \min_{ \substack{ |c'\cdot w_0|= n-2i \\ c' \in \Char(\Z^{N})}} \norm{c'} - N
\end{equation}
for $0\leq i \leq n/2$ and $t_{i}(\kappa)=0$ for $i\geq n/2$.
Furthermore, if $2\leq \rho_1 \leq \dots \leq \rho_m$ are the stable coefficients of $L$, then these depend only on $\kappa$ and the genus can be computed by
\[2g(\kappa)=\sum_{i=1}^m \rho_i(\rho_i-1).\]
 \end{cor}
\begin{proof}
As $|H_1(\Sigma(L);\Z)|=\det(L)=p>0$, $L$ is not a split link. So, given a reduced alternating diagram $D$, we may apply Theorem~\ref{intro:thm:CM} to the sharp manifold $X_D$ from Theorem~\ref{Ch4:thm:altboundssharp} and the surgery $S_{p/q}^3(\overline{\kappa})\cong -\Sigma(L)$. This shows that for some $S\geq 0$, $\Lambda_D\cong L \oplus \Z^S$, where $L$ is the $p/q$-changemaker lattice
\[L=\langle w_0, \dots, w_l \rangle^\bot \subseteq \Z^N\]
such that $w_0$ satisfies
\begin{equation*}
8V_i = \min_{ \substack{ |c'\cdot w_0|= n-2i \\ c' \in \Char(\Z^{N})}} \norm{c'} - N,
\end{equation*}
for $0\leq i \leq n/2$. By Remark~\ref{Ch1:rem:CMgivesallVi}, $V_i=0$ for $i\geq n/2$. As $\Sigma(L)$ is an $L$-space, the $V_i$ are given by the torsion coefficients of $\Delta_{\overline{\kappa}}(t)=\Delta_\kappa(t)$, proving \eqref{Ch4:eq:torsionformula}. As $D$ is reduced, $\Gamma_D$ contains no cut-edges. By Lemma~\ref{Ch3:lem:norm1vectors}, this shows that $\Lambda_D$ contains no vectors of norm one. This shows that $S=0$, as required. Theorem~\ref{Ch1:thm:stabledependence} shows that the stable coefficients do not depend on the surgery slope. The formula for the genus is a consequence of Lemma~\ref{Ch1:lem:calcTi}.
\end{proof}

\subsection{Flypes}
If $D$ and $D'$ are any two reduced diagrams for $L$, then one can be obtained from another by a sequence of flypes \cite{Menasco93classification}. The following lemma, which is an application of Lemma~\ref{Ch3:lem:cutedge}, allows us to detect certain flypes algebraically and provides an explicit isomorphism by relating the vertices of $\Gamma_D$ and $\Gamma_{D'}$. The flypes in question are depicted in Figure~\ref{Ch4:fig:flype1}.
\begin{lem}\label{Ch4:lem:flype1}
Let $D$ be a reduced alternating diagram. Suppose $\Gamma_D$ is 2-connected and has a vertex $v$, which can be written as $v=x+y$, for some $x,y \in \Lambda_D$ with $x\cdot y=-1$. Then there are unique vertices, $u_1,u_2\ne v$, satisfying $u_1\cdot x, u_2\cdot y>0$, and there is a flype to a diagram $D'$ and an isomorphism $\Lambda_{D'}\cong \Lambda_{D}$, such that the vertices of $\Gamma_{D'}$ are obtained from those of $\Gamma_D$ by replacing $v,u_1$ and $u_2$, with $x,y$ and $u_1+u_2$.
\end{lem}
\begin{figure}[p!]
  \centering
  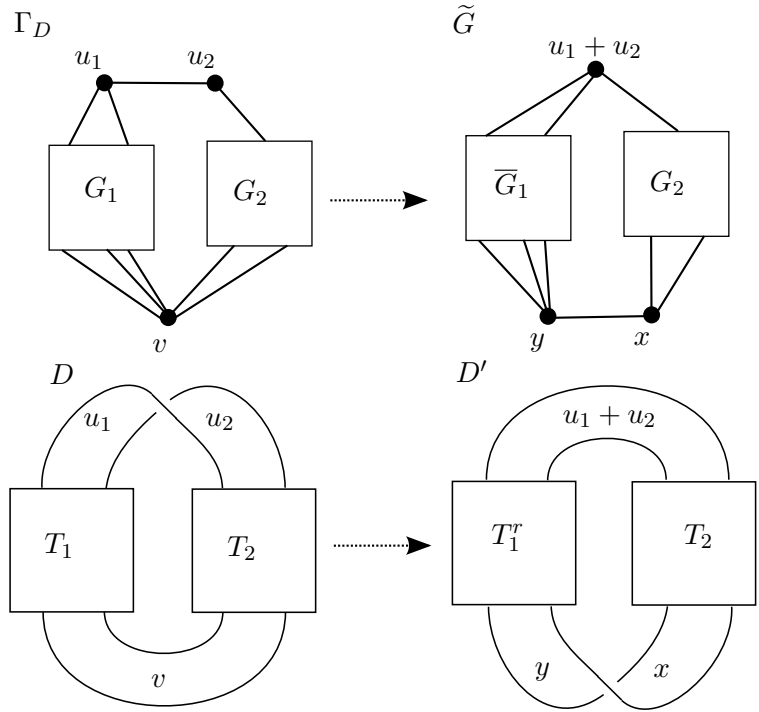
 \caption[The flype given by Lemma \ref{Ch4:lem:flype1}]{The flype and new embeddings given by Lemma \ref{Ch4:lem:flype1}, where $\overline{G}_1$, is the reflection of $G_1$ about a vertical line and $T_1^r$ is the corresponding tangle (which can also be obtained by rotating $T_1$ by an angle of $\pi$ about a vertical line).}
 \label{Ch4:fig:flype1}
\end{figure}
\begin{proof}
 By Lemma \ref{Ch3:lem:cutedge}, there is a cut edge, $e$, in $\Gamma_D\setminus\{v\}$ between $u_1$ and $u_2$ satisfying $u_1\cdot x=u_2\cdot y=1$. Thus there are subgraphs $G_1$ and $G_2$ such that $x=v+u_1+\sum_{z\in G_1}z$ and $y=v+u_2+\sum_{z\in G_2}z$. For $z_1\in G_1$, we have $z_1\cdot x=0$ and $z_1\cdot y= z_1\cdot v$. For $z_2\in G_2$, we have $z_2\cdot y=0$ and $z_2\cdot x=z_2\cdot v$. Thus, it follows that replacing $v,u_1$ and $u_2$ by $x,y$ and $u_1+u_2$ gives the set of vertices for the graph $\widetilde{G}$, obtained by replacing $v$ by two vertices with an edge $e'$ between them, in such a way that $\{e',e\}$ is a cut set, and then contracting $e$. This graph is planar, and can be drawn in the plane so that $\widetilde{G}=\Gamma_{D'}$ for $D'$, obtained by a flype about the crossing corresponding to the edge $e$ and the region corresponding to $v$. This is illustrated in Figure~\ref{Ch4:fig:flype1}.
\end{proof}
The other type of flype we will wish to consider is of the type shown in Figure~\ref{Ch4:fig:flype2}.
Suppose that we have an alternating diagram $D$ with white regions $v$ and $w$ which form a cut set in $\Gamma_D$ and have a crossing between them. If $G_1$ is a component of $\Gamma_D \setminus \{v,w\}$, which is adjacent to an edge between $v$ and $w$ in the plane, then Figure~\ref{Ch4:fig:flype2} shows that there is a flype in $D$ which rotates the tangle corresponding to $G_1$ by $\pi$. Let $D'$ be the diagram resulting from this flype. Since
\[(v+w)\cdot z = - z\cdot[G_1],\]
for all $z \in G_1$, we get an explicit isomorphism from $\Lambda_D$ to $\Lambda_{D'}$ by replacing each vertex $z \in G_1$ by $-z$ and by replacing $v$ and $w$ by $v+[G_1]$ and $w+[G_1]$ respectively.

\begin{figure}[p!]
  \centering
  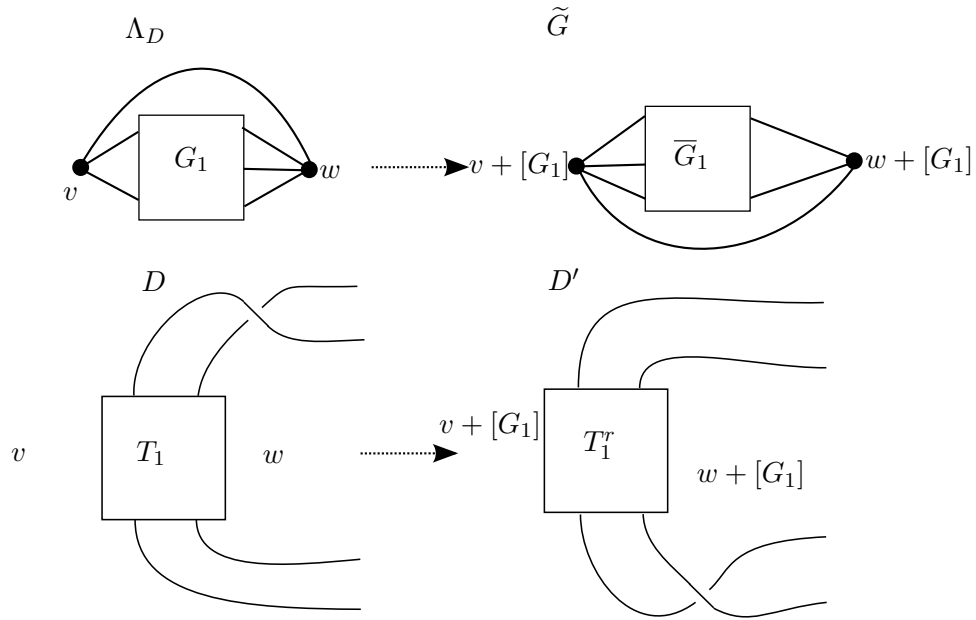
 \caption[Another type of flype]{Flypes arising when $v$ and $w$ form a cutset with an edge between them. Here $\overline{G}_1$ is the reflection of $G_1$ about a vertical line and $T_1^r$ is the corresponding tangle (which can also be obtained by rotating $T_1$ by an angle of $\pi$ about a vertical line).}
 \label{Ch4:fig:flype2}
\end{figure}

\section{Rational tangles}\label{sec:tangles}
We will give a summary of the necessary background material on rational tangles and their slopes. A more comprehensive account can be found in \cite{BurdeZieschang} or \cite{Gordon09dehnsurgery}.
A {\em tangle} in $B^3$ is a pair $(B^3,A)$ where $A$ is a properly embedded 1-manifold. We say $(B^3,A)$ is a {\em marked tangle}, if $\partial B^3 \cap A$ consists of 4 points and we have fixed an identification of the pairs $(\partial B^3, \partial B^3 \cap A)$ and $(S^2, \{NE,NW,SE,SW\})$, where the marking on $S^2$ is as illustrated in Figure~\ref{Ch4:fig:handv}. Two marked tangles $(B^3,A)$ and $(B^3,A')$ are considered equivalent if there is an isotopy of $B$, fixing the boundary $\partial B$, which takes $A$ to $A'$.
 \begin{figure}[ht]
  \centering
  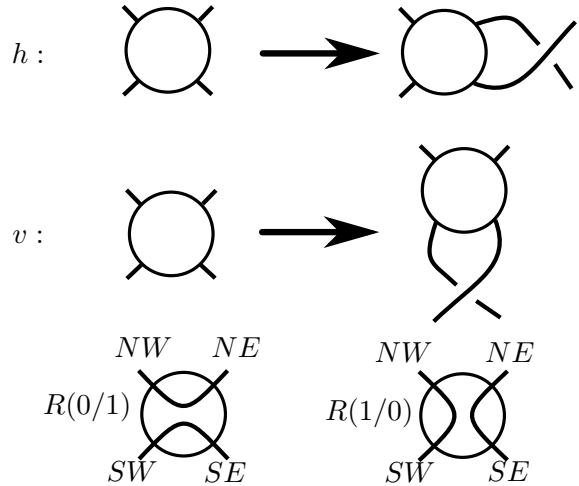
 \caption[The tangle-building operations $h$ and $v$]{The rational tangles $R(0/1)$, $R(1/0)$ and the tangle-building operations $h$ and $v$.}
 \label{Ch4:fig:handv}
\end{figure}
For integers $a_1, \dots, a_k$, we get a rational number $p/q$ via the continued fraction
\[p/q = [a_1 , \dots, a_k]^+=
       a_1 + \cfrac{1}{a_2
           + \cfrac{1}{\ddots
           + \cfrac{1}{a_k} } }.\]
Consider the marked tangles $R(0/1)$ and $R(0/1)$ and the tangle operations $h$ and $v$, as shown in Figure~\ref{Ch4:fig:handv}. Using these operations and the above continued fraction, we can construct a marked tangle $R(p/q)$ defined by
\begin{equation}\label{Ch4:eq:standardtangle}
R(p/q) =
    \begin{cases}
        h^{a_1}v^{a_2} \dots h^{a_{k-1}}v^{a_k}R(1/0)   & \text{if $k$ even}\\
        h^{a_1}v^{a_2} \dots v^{a_{k-1}}h^{a_k}R(0/1)   & \text{if $k$ odd.}
    \end{cases}
\end{equation}
It turns out that this marked tangle $R(p/q)$ depends only on the fraction $p/q$ and not on the choice of continued fraction.
\begin{defn}
We say that a marked tangle $T$ is a {\em rational tangle of slope $p/q$} if it is equivalent as a marked tangle to $R(p/q)$.
\end{defn}
\begin{rem}
By definition, applying $v$ or $h$ to the rational tangle $R(p/q)$ gives a new rational tangle. Using \eqref{Ch4:eq:standardtangle}, we can calculate the slopes of the tangles obtained this way:
\begin{equation}\label{Ch4:eq:vhslopemodification}
vR(\frac{p}{q})=R((1+\frac{q}{p})^{-1})=R(\frac{p}{p+q})\quad \text{and} \quad hR(\frac{p}{q})=R(1+\frac{p}{q})=R(\frac{p+q}{q}).
\end{equation}
\end{rem}
Observe that as an unmarked tangle any rational tangle is homeomorphic to the trivial tangle, $R(1/0)$. Conversely, it is a theorem of Conway that a marked tangle which is homeomorphic to $R(1/0)$ is a rational tangle and that its isotopy class as a marked tangle is determined by its slope \cite{conway69algebraic}.

\begin{figure}[ht]
  \centering
  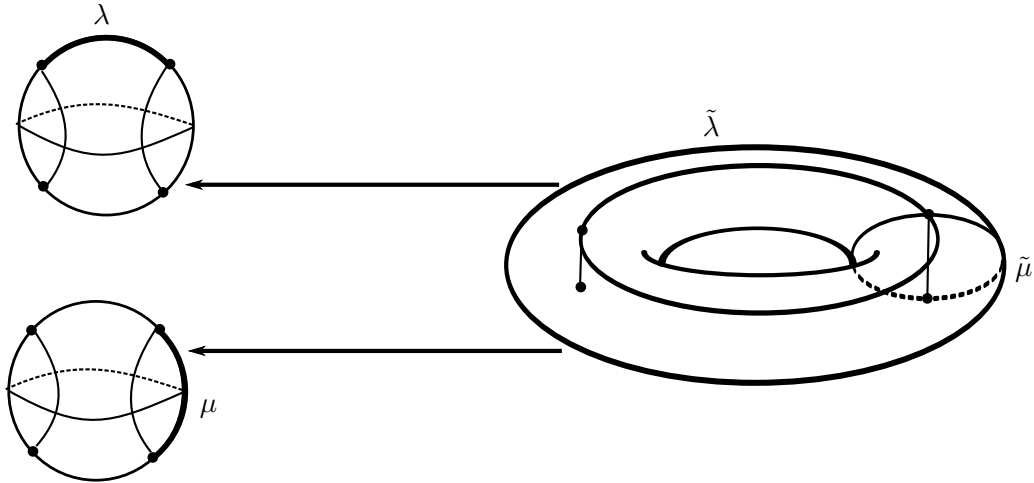
 \caption[Branched cover of $R(1/0)$]{The branched double cover of $R(1/0)$. The curves $\tilde{\mu}$ and $\tilde{\lambda}$ are the lifts of $\mu$ and $\lambda$ respectively.}
 \label{Ch4:fig:coveringmap}
\end{figure}

Since any rational tangle is homeomorphic to the tangle $R(1/0)$, its branched double cover is a solid torus, $V$. Consider the branched double cover of $R(1/0)$ given in Figure~\ref{Ch4:fig:coveringmap}. If we orient the curves $\tilde{\mu}$ and $\tilde{\lambda}$ so that $\tilde{\mu}\cdot\tilde{\lambda}=-1$, then classes $[\tilde{\mu}],[\tilde{\lambda}]$ form a basis for $H_1(\partial V)$ and with respect to this basis, the class which represents a meridian in the branched double cover of $R(p/q)$ is $p[\tilde{\mu}]+q[\tilde{\lambda}]$. This could be used to give an alternative definition of the slope of a rational tangle.

\section{Tangles in alternating diagrams}\label{Ch4:sec:tanglesindiagrams}
Observe that the above identification of rational tangles with the rational numbers depends on the marking on the boundary of the tangle. When a rational tangle occurs in an alternating diagram, we will describe a choice of marking which will allow us to define its slope without ambiguity. Suppose the rational tangle $T$ is contained in an alternating diagram $D$. As in Section~\ref{sec:altdiagrams}, we can colour $D$ so that every crossing has incidence number $-1$. We then choose a marking on the boundary so that the arc $\lambda$ lies in a shaded region and $\mu$ does not. This allows four choices of markings on the sphere. However, in each case the lifts of $\mu$ and $\lambda$ give the same basis for $H_1(\partial V)$, so the slope of $T$ is independent of this choice. Throughout this thesis, we will use this convention for the slope whenever we have a rational tangle in an alternating diagram.
 \begin{figure}[ht]
  \centering
  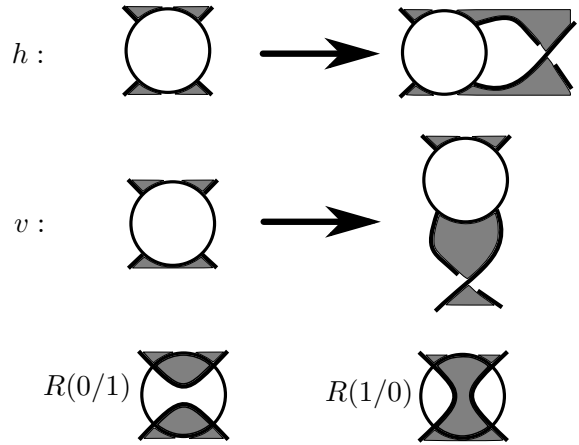
 \caption[Colouring the tangle-building operations]{How the shading of a rational tangle in an alternating diagram behaves on the tangles $R(0/1)$, $R(1/0)$ and how it is altered by the tangle-building operations $h$ and $v$.}
 \label{Ch4:fig:shadedhandv}
\end{figure}

The following proposition allows us to find rational tangles in an alternating diagram and to calculate their slope using the degrees of vertices in the corresponding white graph.
\begin{prop}\label{Ch4:prop:tangledetection}
Let $D$ be a reduced alternating diagram. Suppose that there is a disk in the plane whose boundary intersects two distinct white regions of $D$.  Let $T$ be the tangle contained in the disk and let $\Gamma_T$ be the subgraph of $\Gamma_D$ induced by $T$. Suppose that $\Gamma_T$ consists of vertices $v_0, \dots, v_{l+1}$ for $l\geq 0$, where $v_0$ and $v_{l+1}$ are the white regions which intersect the disk boundary. If there is precisely one edge between $v_i$ and $v_{i+1}$ for $0\leq i< l$ and every remaining edge in $\Gamma_T$ is incident to $v_{l+1}$, then $T$ is a rational tangle. Moreover, if $b_i$ denotes the number of edges incident to $v_i$ in $\Gamma_T$ for $0\leq i \leq l$, then the slope of $T$ is $\frac{p}{q}$, which can be calculated by the continued fraction
\begin{equation}\label{Ch4:eq:tangleslope}
\frac{q}{p} = [b_0, \dots, b_l]^-
= b_0 - \cfrac{1}{b_1
           - \cfrac{1}{\ddots
           - \cfrac{1}{b_l}}}.
\end{equation}
\end{prop}
\begin{proof}
We proceed by induction on the number of crossings. If $T$ does not contain any crossings, then it is equivalent as a marked tangle to $R(1/0)$. It is not $R(0/1)$ as we are assuming $\Gamma_T$ has at least two white regions. In this case, $\Gamma_T$ consists of two vertices and no edges between them. Therefore $b_0=0$ and \eqref{Ch4:eq:tangleslope} is satisfied. The conditions in the proposition are sufficient to guarantee that $T$ is obtained from a smaller tangle $T'$ by applying one of the operations $h$ or $v$. We consider the cases $b_0>1$ and $b_0=1$ separately.

If $b_0>1$, then there is a crossing between $v_0$ and $v_{l+1}$. So we see that $T$ is obtained by applying the operation $v$ to a tangle $T'$ where the white graph $\Gamma_{T'}$ is obtained by deleting an edge between $v_0$ and $v_{l+1}$. By the inductive hypothesis, we can assume that $T'$ is a rational tangle of slope $\frac{p'}{q'}$, where
\[\frac{q'}{p'} = [b_0-1, \dots, b_l]^-.\]
By \eqref{Ch4:eq:vhslopemodification}, we have $vR(\frac{p'}{q'})=R(\frac{p'}{p'+q'})$. Therefore $T$ is a rational tangle of slope $\frac{p}{q}=\frac{p'}{p'+q'}$. Hence, we have
\[ [b_0, \dots, b_l]^-=1+\frac{q'}{p'}=\frac{p'+q'}{p'}=\frac{q}{p},\]
as required.

If $b_0=1$, then there is no crossing between $v_0$ and $v_{l+1}$ and the sole crossing incident to $v_0$ is between $v_0$ and $v_1$. Thus we see that $T$ is obtained by applying the operation $h$ to a tangle $T'$ where the white graph $\Gamma_{T'}$ is obtained by deleting $v_0$. By the inductive hypothesis, we can assume that $T'$ is a rational tangle of slope $p'/q'$, where
\[\frac{q'}{p'} = [b_1-1, \dots, b_l]^-.\]
By \eqref{Ch4:eq:vhslopemodification}, we have $hR(\frac{p'}{q'})=R(\frac{p'+q'}{q'})$. Thus $T$ is a rational tangle of slope $\frac{p}{q}=\frac{p'+q'}{q'}$. Hence, we have
\[ [b_0, \dots, b_l]^-=1-\frac{1}{\frac{q'}{p'}+1}=\frac{q'}{p'+q'}=\frac{q}{p},\]
as required.
\end{proof}

\section{The Montesinos trick and alternating surgeries}
Now we suppose that we have a knot or link $L$ with a diagram $D$ obtained by replacing a $1/0$-tangle in a diagram of the unknot $D'$ by a rational tangle of slope $p/q$. The double cover of $S^3$ branched along $D'$ is again $S^3$ and the $1/0$-tangle lifts to give a solid torus $V \subseteq S^3$. Let $\kappa$ be the knot given by the core of $V$. Let $\lambda_0$ be a null-homologous longitude of $\kappa$ lying in $\partial V$ and $m\in \Z$ be such that
\[ [\tilde\lambda]=m [\tilde\mu] + [\lambda_0]\in H_1(\partial V).\]
If we consider the branched double cover of $D$, we see that it is obtained by cutting out the interior of $V$ and gluing in a solid torus in such a way that a curve representing the homology class $p[\tilde{\mu}]+q[\tilde{\lambda}]$ bounds a disk. Therefore, we see that $\Sigma(L)$ is obtained by surgery on $\kappa$. In particular, we have
\begin{equation}\label{Ch4:eq:surgerycoef}
\Sigma(L)\cong S^3_{-(m +\frac{p}{q})}(\kappa).
\end{equation}
\begin{rem}
There are different conventions for labeling the slope of a rational tangle and the slope of Dehn surgery (c.f the remark after Corollary~4.4 in \cite{Gordon09dehnsurgery}). We have chosen to orient $\tilde{\mu}$ and $\tilde{\lambda}$, so that $\tilde{\mu}\cdot\tilde{\lambda}=-1$. In order to match the usual conventions for Dehn surgery, we would need to reverse the orientation on $\tilde{\lambda}$.
This explains the minus sign appearing in the surgery coefficient of \eqref{Ch4:eq:surgerycoef}.
\end{rem}

In order to determine the sign of the integer $m$, we will quote a special case for which it is known.
\begin{prop}[Proof of Theorem~8.1, \cite{ozsvath2005knots}]\label{Ch4:prop:Montesinos}
Suppose that a knot $K\subseteq S^3$ can be unknotted by changing a negative crossing to a positive one. Then the knot $\kappa \subseteq S^3$ arising through the Montesinos trick satisfies $\Sigma(K)\cong S_{-\delta/2}^3(\kappa)$ where $\delta=(-1)^{\sigma(K)/2}\det(K)$.
\end{prop}

\paragraph{} Recall that we defined $\mathcal{D}$ to be the set of knots arising by applying the Montesinos trick to almost-alternating diagrams of the unknot. We can now calculate the relationship between the surgery slope and tangle replacement in this case.
\begin{prop}\label{Ch4:prop:3implies1}
Let $L$ be a knot or link with an alternating diagram $D$ containing a rational tangle of slope $r/s\geq 0$. Let $D'$ be the alternating diagram obtained by replacing this tangle with a single crossing $c$. If $\sigma(D')=-2$ and $c$ is a positive unknotting crossing or $\sigma(D')=0$ and $c$ is a negative unknotting crossing, then $\Sigma(K)\cong S_{-p/q}^3(\kappa)$, where $q=r+s$; $\det K=p=qm+r>0$ for some integer $m$; and $\kappa \in \mathcal{D}$ is the knot corresponding to the almost-alternating diagram of the unknot obtained by changing $c$.
\end{prop}
\begin{proof}
As $c$ is assumed to be an unknotting crossing, let $\widetilde{D}$ be the almost-alternating diagram of the unknot obtained by changing it. Let $\kappa\in \mathcal{D}$ be the corresponding knot. We see that $D$ is obtained by replacing the non-alternating crossing in $\widetilde{D}$, which is a $-1$-tangle, by a $\frac{r}{s}$-tangle. As shown in Figure~\ref{Ch4:fig:crossingreplace}, this is equivalent to replacing a $1/0$-tangle by a $\frac{r}{r+s}$-tangle.
\begin{figure}
  \centering
  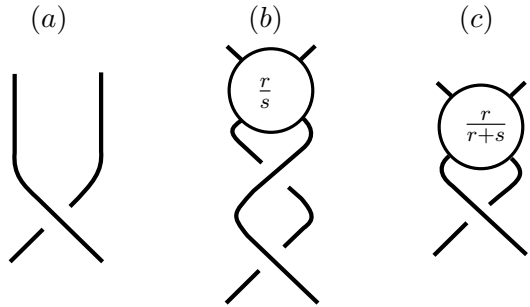
 \caption[Tangle replacement and crossing changes]{Replacing the tangle $R(-1/1)$ with $R(r/s)$ is equivalent to replacing the tangle $(a)$ with $(b)$. Since $(b)$ is isotopic to $(c)$, this shows that the tangle replacement is equivalent to replacing $R(1/0)$ with $R(\frac{r}{r+s})$.}
 \label{Ch4:fig:crossingreplace}
\end{figure}
Applying \eqref{Ch4:eq:surgerycoef} and the discussion preceding it, we see that there is an integer $m$ such that
\[\Sigma(D')\cong S_{-(m + \frac{1}{2})}^3(\kappa)\quad \text{and} \quad \Sigma(K)\cong S_{-(m +\frac{r}{q})}^3(\kappa),\]
where $q=r+s$. The sign conditions on the unknotting crossing $c$ combined with Proposition~\ref{Ch4:prop:Montesinos} imply that $m\geq 0$. Taking  $p=qm +r$, we see that
\[\Sigma(L)=S_{-p/q}^3(\kappa).\]
To complete the proof, observe that
\[|H_1(\Sigma(L))|=\det L = |H_1(S_{-p/q}^3(\kappa))|=p.\]
\end{proof} 
\chapter{Unknotting number one}\label{chap:uk=1}
In this chapter, we use the machinery of changemaker lattices and sharp 4-manifolds to prove Kohn's conjecture for alternating knots.
\begin{thm}\label{Intro:thm:unknotting}
For an alternating knot $K$ the following are equivalent:
\begin{enumerate}[(i)]
\item $K$ has unknotting number one;
\item $\Sigma(K)$ can be obtained by half-integer surgery on a knot in $S^3$;
\item $K$ has an unknotting crossing in every alternating diagram.
\end{enumerate}
\end{thm}
A more precise statement is the following theorem, from which Theorem~\ref{Intro:thm:unknotting} easily follows.
\begin{thm}\label{Ch5:thm:technical}
Let $K$ be a knot with $\det K =d$ and a reduced alternating diagram $D$. The following are equivalent:
\begin{enumerate}[(i)]
\item $\sigma(K)=0$ and $K$ can be unknotted by changing a negative crossing, or $\sigma(K)=-2$ and $K$ can be unknotted by changing a positive crossing;
\item there is a knot $\kappa \subseteq S^3$ such that $\Sigma(K)\cong S^3_{-\frac{d}{2}}(\kappa)$;
\item the Goeritz lattice $\Lambda_D$ is isomorphic to a half-integer changemaker lattice;
\item $\sigma(K)=0$ and $D$ contains a negative unknotting crossing, or $\sigma(K)=-2$ and $D$ contains a positive unknotting crossing;
\end{enumerate}
\end{thm}
The key step accomplished in this chapter is to prove the implication $(iii) \Rightarrow (iv)$. Throughout this chapter, we will consider a half-integer changemaker lattice
\[L=\langle e_0+ \sigma_1 f_1 + \dots + \sigma_r f_r, e_0-e_1 \rangle^\bot \subseteq \Z^{r+2},\]
where the $\sigma_i$ form an increasing sequence with $\sigma_1=1$,
and suppose that we have a reduced alternating diagram $D$ with an isomorphism
\[\iota_D : \Lambda_D \longrightarrow L.\]
This isomorphism gives us a distinguished collection of vectors in $L$ given by the image of the $r+1$ vertices of $\Gamma_D$. We call this collection $V_D$, and, in an abuse of notation, we will sometimes fail to distinguish between a vertex of $\Gamma_D$ and the corresponding element in $V_D$. We will use $\sigma$ to denote the vector
\[\sigma=e_0+ \sigma_1 f_1 + \dots + \sigma_r f_r.\]

\begin{rem} Lemma~\ref{Ch3:lem:fracindecomp} shows that the lattice $L$ is indecomposable. By Lemma~\ref{Ch3:lem:2connectgraphlat}, this means that $\Gamma_D$ is 2-connected and any $v\in V_D$ is irreducible.
\end{rem}

\paragraph{}It will be necessary for us to flype $D$ to obtain a new reduced alternating diagrams. In all cases, this flype will be an application of Lemma~\ref{Ch4:lem:flype1} or a flype as appearing in Figure~\ref{Ch4:fig:flype2}. In either case, if $D'$ is the diagram we obtain from such a flype, then we get a natural choice of $V_{D'}\subseteq L$ and hence an isomorphism
\[\iota_{D'}: \Lambda_{D'} \longrightarrow L.\]
Whenever we flype, we will implicitly use these choices of isomorphism to speak of $V_{D'}$ without ambiguity.

\begin{figure}[ht]
  \centering
  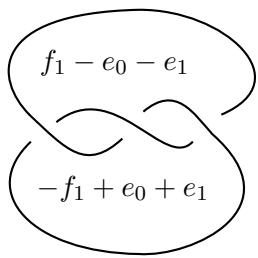
 \caption{The right-handed trefoil}
 \label{Ch5:fig:trefoil}
\end{figure}

\begin{example}\label{Ch5:exam:trefoil}
If $D$ is a reduced alternating diagram for the right-handed trefoil, then $\Gamma_D$ is a graph with two vertices and three edges and we have
\[\Lambda_D\cong\langle f_1+e_0, e_{1}-e_0 \rangle ^\bot=\langle f_1-e_0-e_{1} \rangle\subseteq \mathbb{Z}^3.\]
Thus $\Lambda_D$ is isomorphic to the unique half-integer changemaker lattice of rank one and it has vertices $V_D=\{\pm (f_1-e_0-e_{1})\}$, as shown in Figure~\ref{Ch5:fig:trefoil}. In fact, as $\langle f_1-e_0-e_{1} \rangle$ has only two non-zero irreducible vectors, this is the only reduced alternating diagram with $\Lambda_D$ isomorphic to a changemaker lattice of rank one.
\end{example}

\paragraph{} Recall that $\sigma_s$ is tight if
\[\sigma_s=1+\sigma_1 + \dots + \sigma_{s-1};\]
that $\sigma$ is tight if there is $s>1$ with $\sigma_s$ tight; and that $\sigma$ is {\em slack} otherwise. The following variant on Proposition~\ref{Ch3:prop:CMcondition} will be useful.

\begin{lem}\label{Ch5:lem:slacksum}
If $\sigma_s$ is not tight, then we can write $\sigma_s=\sum_{i \in A} \sigma_i$ for some $A\subseteq \{1, \dots , s-1 \}$. Moreover, if $\sigma$ is slack we may assume that $1\in A$.
\end{lem}
\begin{proof}
If $\sigma_s$ is not tight, then $\sigma_s \leq \sigma_1 + \dots + \sigma_{s-1}$. Thus, Proposition~\ref{Ch3:prop:CMcondition} implies the existence of $A\subseteq \{1,\dots, s-1\}$ such that $\sigma_s=\sum_{i \in A} \sigma_i$.

If $\sigma$ is slack, then we we will prove that there is a choice of $A$ containing 1 by induction. If $s=2$, then $\sigma_s=1$ and $A=\{1\}$. Thus suppose $s>2$ and let $m=\min A$. If $m>1$, then by induction we may assume that $\sigma_m=\sum_{i \in A'}\sigma_i$ for $A'\subseteq \{1, \dots, m-1\}$ with $1\in A'$. The set $A\setminus \{m\} \cup A'$ then has the desired properties.
\end{proof}

\paragraph{}We use the fact that $L$ is a changemaker lattice to deduce information about the elements of $V_D$. The following lemma serves as a useful sample calculation as well as being helpful in its own right.

\begin{lem}\label{Ch5:lem:01splitting}
Let $x=\sum_{v \in R}v$ for some $R\subseteq V_D$ be a vertex or a sum of vertices. If there is $z=-f_g + \sum_{f \in A} f \in L$ with $A\subseteq \{e_0,e_1,f_1, \dots, f_{g-1}\}$ such that $x\cdot f>0$ for all $f \in A$ and $x\cdot f_g \leq 0$, then $(x-z)\cdot z\in \{-1,0 \}$.
\end{lem}
\begin{proof}
We apply Lemma~\ref{Ch3:lem:usefulbound} to $(x-z)\cdot z$. This gives
\begin{equation*}
0\geq (x-z)\cdot z = -x\cdot f_g-1 + \sum_{f\in A} (x\cdot f -1).
\end{equation*}
However by hypothesis, $x\cdot f_g\leq 0$ and $x\cdot f\geq 1$ for all $f \in A$. So we also have \[(x-z)\cdot z\geq -1,\]
as required.
\end{proof}
It is worth noting that further information can be gleaned from the proof of the preceding Lemma. For example, we see that $x\cdot f=1$, for all but possibly one $f\in A$. Such arguments will be used frequently later in this paper. In these arguments, the vector $z$ will usually be constructed by applying Lemma~\ref{Ch5:lem:slacksum}.

\section{Marked crossings}\label{sec:markedcrossings}
In this section, we will define marked crossings and prove their existence. We will also show that if there is more than one marked crossing, then any one of them is an unknotting crossing.

\begin{lem}\label{Ch5:lem:coefbound}
If $x = \sum_{v\in R}v \in L$ for $R\subseteq V_D$ is a sum of vertices, then
\begin{enumerate}[(i)]
\item $|x\cdot e_0|\leq 1$;
\item $|x\cdot f_1|\leq 2$; and
\item if $x$ is irreducible with $x\cdot e_0\ne 0$, then $|x\cdot f_1|\leq 1$.
\end{enumerate}
\end{lem}
\begin{proof} We use similar arguments to establish all three parts of the lemma.

(i) Since $-x =\sum_{v\in V_D \setminus R}v$ is also a sum of vertices, we may assume that $x\cdot e_0\geq 0$. Let $g>0$ be minimal such that $x\cdot f_g\leq 0$. We may write $\sigma_g -1 = \sum_{i \in A} \sigma_i$ for some $A \subseteq \{1, \dots , g-1\}$. Let $z=-f_g + e_0 + e_{1} + \sum_{i \in A} f_i$. By construction, $z \in L$. By Lemma~\ref{Ch3:lem:usefulbound}, $(x-z)\cdot z\leq 0$. This gives
\[0\geq(x-z)\cdot z \geq -1-x\cdot f_g + (x\cdot e_0 - 1) + (x\cdot e_{1} - 1) \geq -1 + 2x\cdot e_0 -2.\]
Thus $x\cdot e_0\leq \frac{3}{2}$.

(ii) We may assume $x\cdot f_1\geq 0$. Let $g>1$ be minimal such that $x\cdot f_g\leq 0$. We may write $\sigma_g -1 = \sum_{i \in A} \sigma_i$ for some $A \subseteq \{1, \dots , g-1\}$. Let $z=-f_g + f_1 + \sum_{i \in A} f_i$. We have $z\cdot f_1=\varepsilon \in \{1,2\}$. By Lemma~\ref{Ch3:lem:usefulbound}, $(x-z)\cdot z\leq 0$. This gives
\[0\geq (x-z)\cdot z\geq -1 + \varepsilon(x\cdot f_1-\varepsilon).\]
Thus $x\cdot f_1\leq \varepsilon + \frac{1}{\varepsilon}\leq \frac{5}{2}$.

(iii) Suppose now that $x$ is irreducible. We may assume that $x\cdot e_0 = 1$. Let $g>0$ be minimal such that $x\cdot f_g\leq 0$. If $g=1$, then let $z=-f_1 + e_0 + e_{1}$. By irreducibility, it follows that either $x=z$ or
\[(x-z)\cdot z= -(x\cdot f_1+1)\leq -1.\]
This shows that $0\geq x\cdot f_1\geq -1$ in this case.
Suppose now that $g>1$. We may write $\sigma_g -1 = \sum_{i \in A} \sigma_i$ for some $A \subseteq \{1, \dots , g-1\}$. If $1\in A$, set $z=-f_g + e_0 + e_{1} + \sum_{i \in A} f_i$. If $1\notin A$, set $z=-f_g + f_{1} + \sum_{i \in A} f_i$. In either case $z\cdot f_1=1$. By irreducibility, it follows that either $x=z$ or
\[ x\cdot f_1 - 1-(x\cdot f_g+1) \leq (x-z)\cdot z \leq -1.\]
This shows that $x \cdot f_1 = 1$ in this case.
\end{proof}

Since the $V_D$ span $L$ and we have $-f_1 + e_0 +e_1\in L$, there is $v\in V_D$ with $v\cdot e_0>0$. Lemma~\ref{Ch5:lem:coefbound}$(i)$ shows that $v$ must be unique and that $v\cdot e_0=1$. Similarly, there is a single vertex $w \in V_D$ with $w\cdot e_0<0$ and it satisfies $w\cdot e_0=-1$.
\begin{defn}\label{Ch5:def:markedcrossing}
We say that $v$ and $w$ are the {\em marker vertices} of $\Gamma_D$ and that any crossing between the corresponding regions of $D$ is a {\em marked crossing}.
\end{defn}

\paragraph{} Returning to Example~\ref{Ch5:exam:trefoil}, we see that every crossing of the trefoil in Figure~\ref{Ch5:fig:trefoil} is a marked crossing.

\paragraph{}Note that replacing $\iota_D$ by $-\iota_D$ gives another isomorphism with the same marked crossings. For the purposes of notation, it will be convenient to fix this choice of sign in the next lemma.

\begin{rem}\label{Ch5:rem:preservesmarked}
Suppose $v$ is a vertex that can be written $v=x+y$, with $x\cdot y=-1$. By Lemma~\ref{Ch4:lem:flype1}, this gives a cut edge $e$ in $\Gamma_D\setminus\{v\}$ and a flype to a diagram with $x$ and $y$ as vertices. Observe that if $x\cdot e_0=0$ or $y\cdot e_0=0$, then the edge $e$ does not correspond to a marked crossing. This means that the flype can be chosen to fix the marked crossings which will again be marked crossings of the new embedding. In particular, such a flype commutes with the act of changing a marked crossing.
\end{rem}

With this in mind, we make our first flypes and prove the existence of marked crossings.

\begin{lem}\label{Ch5:lem:markedcexist}
Let $v$ and $w$ be the marker vertices with $v\cdot e_0=1$ and $w\cdot e_0=-1$. Up to choices of sign for $\iota_D$, we either have $v=-f_1+e_0+e_1$ , or there is a flype, preserving marked crossings, which gives a diagram $D'$ with a marker vertex $v'=-f_1+e_0+e_{1}$. Furthermore, $D'$, and hence $D$, contains between one and three marked crossings.
\end{lem}
\begin{proof}
First, we will show one of the inequalities $v\cdot f_1\leq0$ or $w\cdot f_1\geq 0$ holds. Suppose that we have $v\cdot f_1 >0$. By Lemma~\ref{Ch5:lem:coefbound}$(iii)$, this gives $v\cdot f_1 = 1$. Lemma~\ref{Ch3:lem:generalirred} shows that the vector $v'=-f_1+e_0+e_{1}$ is irreducible. Thus Lemma~\ref{Ch3:lem:irreducible} implies there is $R \subseteq V_D$ such that $v'=\sum_{u\in R}u$. From $v'\cdot e_0=1$, it follows that $v\in R$ and $w\notin R$. Thus $v'-v+w$ is also a sum of vertices. Using Lemma~\ref{Ch5:lem:coefbound}(ii), we get
\[(v'-v+w)\cdot f_1=-2+w\cdot f_1\geq-2,\]
which implies $w\cdot f_1\geq 0$. Therefore at least one of $v\cdot f_1 \leq 0$ or $w\cdot f_1 \geq 0$ must hold.

\paragraph{} Thus up to replacing $\iota_D$ by $-\iota_D$, we may assume that $v\cdot f_1\leq 0$. As before, let $v'=-f_1+e_0+e_{1}$. If $v\cdot f_1=-1$, then $v=v'$ as $v$ is irreducible. Assume that $v\cdot f_1=0$. In this case, $(v-v')\cdot v'=-1$. Applying Lemma~\ref{Ch4:lem:flype1} gives a flype to $D'$ with $(v-v')$ and $v'$ as vertices in $V_{D'}$. Since $(v-v')\cdot e_0=0$, this flype preserves marked crossings, as in Remark~\ref{Ch5:rem:preservesmarked}.

\paragraph{} Since Lemma~\ref{Ch5:lem:coefbound}$(iii)$ shows $|w\cdot f_1|\leq 1$, we have
\[-3\leq v'\cdot w=-2-w\cdot f_1\leq -1.\]
This shows that there are between one and three marked crossings.
\end{proof}

\paragraph{} It follows from Lemma~\ref{Ch5:lem:markedcexist} that we may assume that the diagram $D$ has a marker vertex in the form
\[v=-f_1+e_0+e_{1}.\]
Now we will study some of the effects of changing a marked crossing. In particular, we will show that if there are multiple marked crossings, then they are unknotting crossings.
\begin{lem}\label{Ch5:lem:manymarkedcrossings}
Let $K'$ be the knot obtained by changing a marked crossing $c$ in $D$. Then $K'$ is almost-alternating with $\det(K')=1$ and
\[\sigma(K')=
\begin{cases}
\sigma(K)+2 &\text{if $c$ is positive,}\\
\sigma(K)   &\text{if $c$ is negative.}
\end{cases}
\]
Moreover, if $D$ has more than one marked crossing, then $K'$ is the unknot and the diagram $D'$ obtained by changing $c$ is $\mathcal{C}_m$ or $\overline{\mathcal{C}_m}$ for some $m$, where $\mathcal{C}_m$ is a diagram as in Figure~\ref{fig:claspdiagram}.
\end{lem}
\begin{proof}
As in the statement of the lemma, let $D'$ be the diagram for $K'$ obtained by changing $c$ in $D$. The diagram $D'$ is almost-alternating by definition.

\paragraph{}The embedding of $\Lambda_D$ into $\mathbb{Z}^{r+2}$ gives a factorization of the Goeritz matrix, $G_D=AA^T$, where $A$ is the $r\times (r+2)$-matrix
 \[A =
  \begin{pmatrix}
 v_1\cdot e_1 & v_1\cdot e_0 & v_1\cdot f_1 & \dots  & v_1\cdot f_r \\
   \vdots     & \vdots       & \vdots       & \ddots & \vdots       \\
 v_r\cdot e_1 & v_r\cdot e_0 & v_r\cdot f_1 & \dots  & v_r\cdot f_r \\
  \end{pmatrix}
 \]
 for some choice of $r$ vectors $\{v_1,\dots , v_r\}\subseteq V_D$. Let $C$ be the right-most $r\times r$ submatrix of $A$. Recall from Section~\ref{Ch3:sec:standardbases} that the changemaker lattice $L$ admits a standard basis. Let $\{\nu_1, \dots , \nu_r \}$ be such a standard basis, where $\nu_i$ is the basis element with $\nu_i \cdot f_i = -1$. Let $C'$ be the $r \times r$-matrix $C'=(\nu_i \cdot f_j)_{1 \leq i,j \leq r}$. Since the lattices spanned by the rows of $C$ and $C'$ are isomorphic, we have $|\det(C)|=|\det(C')|$. As $\{\nu_1, \dots , \nu_r \}$ is a standard basis, $C'$ is triangular and all diagonal entries take the value $-1$. Therefore we have $|\det(C)|=|\det(C')|=1$.

\paragraph{}Let $w$ and $v$ be the marker vertices with $w\cdot e_0=-1$ and $v\cdot e_0=1$. Since $(w+e_0+e_{1})\cdot(v-e_0-e_{1})=w\cdot v+2$, we see that $G_{D'}=CC^T$ is a Goeritz matrix for $D'$. Therefore,
\[\det (K')=\det(C)^2 =1,\]
as required.

\paragraph{} Now we compute the change in the knot signature. Observe that $G_{D'}$ is a positive-definite matrix of rank $r$. Since $D$ is alternating, the Gordon-Litherland formula \eqref{Ch4:eq:altsigformula} shows that the signature is calculated by
\[\sigma(K)=r-n_+,\]
where $n_+$ is the number of positive crossings in $D$. If $c$ is positive, then $D'$ has $n_+ -1$ positive crossings of incidence -1 and 1 negative crossing of incidence +1. If $c$ is negative, then $D'$ has $n_+$ positive crossings of incidence -1 and no negative crossings of incidence +1. Thus, Theorem~\ref{Ch4:thm:sigformula} gives
\[\sigma(K')=
\begin{cases}
r-n_+ +2=\sigma(K)+2 &\text{if $c$ is positive,}\\
r-n_+ =\sigma(K)   &\text{if $c$ is negative,}
\end{cases}
\]
as required.

\paragraph{} Suppose that $D$ has more than one marked crossing. We may assume that $D$ is a diagram with marker vertex $v=-f_1+e_0+e_{1}$. Since $\norm{v}=3$, the marked crossings are adjacent in $D$. This allows us to perform a Reidemeister II move on $D'$ to obtain an alternating diagram, $D''$, for $K'$.
The white graph $\Gamma_{D''}$ is obtained by deleting two edges between the marker vertices of $\Gamma_D$. The determinant of an alternating knot is equal to the number of maximal spanning subtrees of the white graph of any alternating diagram \cite{crowell1959nonalternating}. So as $\Gamma_D$ has no self-loops and $\det(K')=1$, it follows that $\Gamma_{D''}$ is a tree. Furthermore, as $\Gamma_D$ has no cut-edges, $\Gamma_{D''}$ must be a path whose endpoints were the marker vertices in $\Gamma_D$. Therefore, $D$ is a diagram of a clasp knot and $D'$ is $\mathcal{C}_m$ or its reflection $\overline{\mathcal{C}_m}$, depending on which marked crossing was changed, for some non-zero $m$.
\end{proof}

\begin{figure}
  \centering
  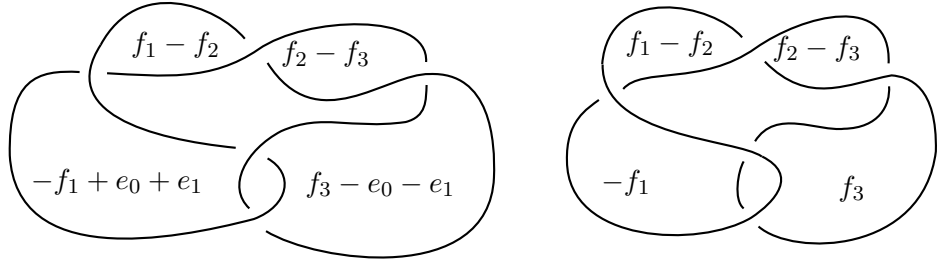
 \caption{Changing a marked crossing in the 5-crossing clasp knot.}
 \label{Ch5:fig:clasp}
\end{figure}

\section{A single marked crossing}\label{sec:singlemarkedc}
Lemma~\ref{Ch5:lem:manymarkedcrossings} shows that if $D$ has multiple marked crossings, then any one of these marked crossings is an unknotting crossing. Now we wish to deal with the case where there is a single marked crossing. We will perform further flypes to find diagrams with extra structure. The aim is to show we can flype so that after changing the marked crossing there is an obvious untongue or untwirl move that can be performed.

\paragraph{}By Lemma~\ref{Ch5:lem:markedcexist}, we may assume that $D$ is such that one of the marker vertices takes the form
\[v= -f_1+e_0+e_{1}.\]
We will write $w$ for the other marker vertex. In this section, we are assuming $w\cdot v=-1$ and hence that $w\cdot f_1=-1$. Since $v$ is a vertex of degree three, there are other vertices pairing non-trivially with $v$.
\begin{defn}\label{Ch5:def:sitAsitB}
There are two cases to consider.
\begin{itemize}
 \item We say we are in {\em Situation~A} if there are two vertices $u_1$ and $u_2$ with
 \[u_1\cdot v=u_2\cdot v=-1.\]
 We necessarily have $u_1\cdot e_0=u_2\cdot e_0=0$ and $u_1\cdot f_1=u_2\cdot f_1=1$.
 \item We say we are in {\em Situation~B} if there is a single vertex $u$ with $u\cdot v=2$. Such a $u$ satisfies $u\cdot e_0=0$ and $u\cdot f_1=2$.
\end{itemize}
In Situations~A and B, we will call $u_1$ and $u_2$, and $u$, the {\em adjacent vertices} respectively.
\end{defn}

\paragraph{} Let $k$ be maximal such that $\sigma_k=1$. If $k>1$, then for $2\leq i\leq k$ let $v_i$ be $v_i=-f_i+f_{i-1}\in L$.

\begin{defn}If $D$ is such that $v_2,\dots , v_k$ and $v$ are vertices, then we say that $D$ is in {\em standard form}.
\end{defn}
The following lemma shows that after further flypes we may assume $D$ is in standard form.
\begin{lem}\label{Ch5:lem:flypetostd}
There is a sequence of flypes to a diagram $D'$ which contains $v_2,\dots , v_k$ and $v$ as vertices. Furthermore, these flypes preserve the marked crossing.
\end{lem}
\begin{proof}
By Lemma \ref{Ch5:lem:markedcexist} we are already assuming that $v$ is a vertex. Suppose there is $1<l\leq k$, for which $v_l$ is not a vertex. Take $m>1$ to be minimal such that $v_m$ is not a vertex. By Lemma~\ref{Ch3:lem:generalirred}, $v_m$ is irreducible, thus Lemma~\ref{Ch3:lem:irreducible} shows that we may write $v_m=\sum_{v\in G_1}v$ for some subgraph $G_1\subseteq \Gamma_D$. As $v_{m-1}$ is a vertex and $v_m\cdot v_{m-1}=-1$, there is a vertex $x\in G_1$ with $x\cdot v_{m-1}=-1$. Since $x\in G_1$, this must satisfy $x\cdot v_m>0$. As $x\ne v_m$ and $x$ is irreducible, it follows that
\[-2<(x-v_m)\cdot v_m=x\cdot v_m -2 <0.\]
In particular, $(x-v_m)\cdot v_m=-1$. By Lemma~\ref{Ch4:lem:flype1}, we can flype to obtain a diagram $D'$ with $v_m$ and $x-v_m$ as vertices. Since $(x-v_m)\cdot v_i\leq 0$ and $v_m\cdot v_i\leq0$  for $i\leq m-1$, this flype leaves $v_2, \dots , v_{m-1}$ and $v$ as vertices and, by Remark~\ref{Ch5:rem:preservesmarked}, it can be chosen to preserve the marked crossing. Thus, we may flype so that all of the $v_2,\dots , v_k$ and $v$ are vertices.
\end{proof}

\paragraph{}It follow that if $k>1$, then we must be in Situation~A and we can take $u_2=v_2$. Since $v_i$ can pair non-trivially with at most two vertices for $1<i\leq k$, it follows that
\[w\cdot f_1= \dots = w\cdot f_{k-1}=-1.\]
Since $u_1\cdot v_i\leq 0$ for all $1<i\leq k$, $u_1$ satisfies
\[u_1\cdot f_1= \dots = u_1\cdot f_k = 1.\]

Now we consider whether $\sigma$ is tight or slack.
\begin{lem}\label{Ch5:lem:slackvtight}
Assume $D$ is in standard form and we are in Situation~A. Then the changemaker vector $\sigma$ is slack if and only if $\Gamma_D\setminus \{v,w\}$ is disconnected. Moreover, if $\sigma$ is slack, then $u_1$ and $u_2$ lie in separate components of $\Gamma_D\setminus \{v,w\}$.
\end{lem}
\begin{proof}
Suppose $\sigma$ is slack. This implies $k>1$. Let $g>1$ be minimal such that $w\cdot f_g\geq 0$. By Lemma~\ref{Ch5:lem:slacksum}, we may find $A\subseteq\{1,\dots , g-1\}$ with $1\in A$, such that $\sigma_g = \sum_{i\in A} \sigma_i$.
Consider $x = f_g-\sum_{i\in A}f_i\in L$. By Lemma~\ref{Ch3:lem:generalirred} this is irreducible. So by Lemma~\ref{Ch3:lem:irreducible},  it can be written as $x=\sum_{v\in G_1}v$ for some connected subgraph $G_1$ of $\Gamma_D$. Using Lemma~\ref{Ch3:lem:usefulbound} combined with $x\cdot v=1$ and $x\cdot w\geq|A|=\norm{x}-1$, it follows that $x\cdot (v+w-x)=0$. Since $x\ne v+w$, this implies $v+w$ is reducible. From Lemma~\ref{Ch3:lem:irreducible}, it follows that $\{v,w\}$ is a cut set. Observe that $u_1$ and $u_2$ must lie in separate components of $\Gamma_D\setminus \{v,w\}$ since $w$ is not a cut-vertex.

\paragraph{} For the converse, suppose there is $s>1$ with $\sigma_s=\sigma_{s-1}+\dots + \sigma_1 +1$.
By Lemma~\ref{Ch3:lem:tightstdvecirred}, the  vector $y=-f_s + f_{s-1} + \dots +f_2 +2f_1$ is irreducible. Thus, Lemma~\ref{Ch3:lem:irreducible} shows that there is a connected subgraph $G_2$, such that $y=\sum_{v\in G_2}v$. Since $y\cdot f_1=2$, it follows that $v,w\notin G_2$ and $u_1,u_2 \in G_2$. Since $G_2$ is connected, it follows that $u_1$ and $u_2$ must lie in the same component of $\Gamma_D \setminus \{v,w\}$. Thus if $\{v,w\}$ were a cut-set, then $\Gamma_D\setminus \{w\}$ would also be disconnected. As $\Gamma_D$ is 2-connected, this shows that $\{v,w\}$ is not a cut-set.
\end{proof}

Armed with this information, we perform another sequence of flypes. As ever, these will come from Lemma~\ref{Ch4:lem:flype1} and, by Remark~\ref{Ch5:rem:preservesmarked}, can be chosen to commute with changing the marked crossing.

\begin{lem}\label{Ch5:lem:tightflyping}
If $\sigma$ is tight, then there is a sequence of flypes to a diagram in standard form such that we have marker vertex $w\ne v$ in the form
\[w=f_g-f_{g-1} - \dots -f_1-e_0- e_{1}.\]
\end{lem}
\begin{proof}
If $\sigma$ is tight and the marker vertex $w\ne v$ in $D$ is not in the form
\[w=f_g-f_{g-1} - \dots -f_1-e_0- e_{1},\]
then we will show that there is a flype to diagram $D'$ in standard form for which the marker vertex either satisfies
\[w'=f_{g'}-f_{g'-1} - \dots -f_1-e_0- e_{1},\]
for some $g'$ or
\[\max \{i \,|\, w'\cdot f_i \ne 0\}<\max \{i \,|\, w\cdot f_i \ne 0\}.\]
This is sufficient to prove the lemma, since a sequence of such flypes cannot continue indefinitely.
\paragraph{}Let $g'$ be minimal such that $w\cdot f_{g'}\geq 0$. There is $A\subseteq \{1,\dots, g'-1\}$ such that $\sigma_{g'}-1=\sum_{i\in A} \sigma_i$.
\begin{claim}
We have $\{1,\dots, k\}\subseteq A$.
\end{claim}
\begin{proof}[Proof of Claim]
Suppose that the claim does not hold. We can take $l=\min \{1,\dots, k\}\setminus A$. Since $\sigma_l=1$, we can take
\[z=f_{g'}- f_l - \sum_{i\in A}f_i \in L.\]
Since $l$ was taken to be minimal, we have either $l=1$ or $1\in A$. Thus we have $z\cdot f_1=-1$. Observe that \[(w+v-z)\cdot z= -1 + w\cdot f_{g'}-\sum_{i\in A \cup \{l\} \setminus \{1\}}(w\cdot f_i + 1) + 1 \geq 0.\]
However, Lemma~\ref{Ch5:lem:slackvtight} and Lemma~\ref{Ch3:lem:irreducible} show that $w+v$ is irreducible when $\sigma$ is tight. This provides a contradiction.
\end{proof}
Consider $w'=f_{g'}- e_0-e_1 - \sum_{i\in A}f_i \in L$. By Lemma~\ref{Ch5:lem:01splitting}, we have $(w-w')\cdot w'\in\{-1,0\}$, with $(w-w')\cdot w'=-1$ only if $w\cdot f_{g'}=0$. If $(w-w')\cdot w'=0$, then $w=w'$ by irreducibility of $w$. This implies that
\[w=f_{g'}-f_{g'-1}-\dots-f_1-e_0 - e_{1},\]
is already in the required form. If $(w-w')\cdot w'=-1$, then, by Lemma~\ref{Ch4:lem:flype1}, we may flype to get diagram a $D'$ containing $w'$ as a vertex. The above claim shows that $w'\cdot v_i=(w-w')\cdot v_i=0$ for $2\leq i \leq k$, $w'\cdot v=-1$ and $(w-w')\cdot v=0$. Hence the diagram $D'$ is in standard form. As $w\cdot f_{g'}=0$, we must have
\[g'=\max \{i \,|\, w'\cdot f_i \ne 0\}<\max \{i \,|\, w\cdot f_i \ne 0\},\]
as required.
\end{proof}

\begin{lem}\label{Ch5:lem:slackflyping}
If $\sigma$ is slack, then there is a sequence of flypes to a diagram in standard form such that we have adjacent vertex $u_1\ne v_2$ in the form
\[u_1= -f_h + f_{h-1} + \dots + f_1.\]
\end{lem}

\begin{proof}
Suppose $\sigma$ is slack, so we are necessarily in Situation~A. Suppose also that the adjacent vertex $u_1\ne v_2$ is not in the form
\[u_1= -f_h + f_{h-1} + \dots + f_1.\]
We will show that there is a flype to diagram a $D'$ in standard form for which the adjacent vertex satisfies
\[u_1'=-f_{h'} + f_{h'-1} + \dots + f_1\]
for some $h'$ or
\[\max \{i \,|\, u_1'\cdot f_i \ne 0\}<\max \{i \,|\, u_1\cdot f_i \ne 0\}.\]
Observe that this is sufficient to prove the lemma, since a sequence of such flypes cannot continue indefinitely.

\paragraph{} Let $h'$ be minimal such that $u_1\cdot f_{h'}\leq 0$. Observe that $h'>k$. By Lemma~\ref{Ch5:lem:slacksum}, there is $A\subseteq \{1, \dots , h'-1\}$ with $1\in A$ and $u_1'=-f_{h'} + \sum_{i\in A} f_i \in L$.

\paragraph{}By Lemma~\ref{Ch5:lem:01splitting}, $(u_1-u_1')\cdot u_1'\in\{0,-1\}$, with $(u_1-u_1')\cdot u_1'=-1$ only if $u_1\cdot f_{h'}=0$. We consider the two cases separately. If $(u_1-u_1')\cdot u_1'=0$, then $u_1=u_1'$ as $u_1$ is irreducible. However, by the minimality of $h'$, this implies that
\[u_1=-f_{h'} + f_{h'-1} + \dots + f_1.\]

\paragraph{} Suppose $(u_1-u_1')\cdot u_1'=-1$. By Lemma~\ref{Ch4:lem:flype1}, we may flype to a diagram $D'$ with $u_1'$ as a vertex. Note that as $u_1\cdot f_{h'}=0$ in this case, we have
\[h'=\max \{i \,|\, u_1'\cdot f_i \ne 0\}<\max \{i \,|\, u_1\cdot f_i \ne 0\}.\]
However, we need to show that $D'$ is in standard form. Observe first that $v\cdot u_1'=-1$ and $(u_1-u_1')\cdot v=0$, so $v$ remains as a vertex in $D'$. Since we are assuming that $\sigma$ is slack, $\Gamma_D\setminus \{v,w\}=G_1 \cup G_2$ is disconnected with $u_1$ and $u_2$ lying in separate components. Say $u_1\in G_1$ and $u_2\in G_2$. Since the vertices $u_2,v_3, \dots, v_k$ form a path, we have $v_3, \dots, v_k \in G_2$. Now let $e$ be a cut-edge in $\Gamma_D \setminus \{u_1\}$. We claim $e$ is not in $G_2$.
\begin{claim}
The edge $e$ is not incident to any vertex in $G_2$.
\end{claim}
\begin{proof}[Proof of Claim.]
Since $e$ is a cut-edge in $\Gamma_D \setminus \{u_1\}$, any cycle containing $e$ must also contain $u_1$. Let $e'$ be an edge incident to some vertex in $G_2$. We will show $e'$ is contained in a cycle not containing $u_1$. Since $\Gamma_D$ is 2-connected, $e'$ is contained in some cycle $C$. If $C$ contains $u_1$ (or any vertex in $G_1$), then it must also contain $v$ and $w$. However as there is an edge between $v$ and $w$, $C\setminus G_1$ form the vertices of a cycle $C'\subseteq G_2\cup \{v,w\}$ containing $e'$. Thus $e'$ is not a cut-edge in $\Gamma_D \setminus \{u_1\}$.
\end{proof}
As $e$ is not incident to any vertex in $G_2$ it is not incident to any of $v_2, \dots, v_k$. It follows from Lemma~\ref{Ch4:lem:flype1} that $u_2=v_2, \dots, v_k$ are vertices in $D'$. In particular, this shows that $D'$ is in standard form.
\end{proof}

\paragraph{}From the preceding lemmata, it follows that there is a sequence of flypes to a diagram $D$ in standard form with
\[w=f_g-f_{g-1} - \dots -f_1-e_0 - e_{1}
\quad \text{or} \quad
u_1= -f_h + f_{h-1} + \dots + f_1,\]
depending on whether $\sigma$ is tight or slack. As the following two lemmata show, we can now deduce the existence of a crossing between the regions corresponding to the marker vertex $w$ and one of the adjacent vertices. Both proofs run along very similar lines, and like much of what has gone before, they make heavy use of Lemma~\ref{Ch3:lem:usefulbound} applied to carefully chosen combinations of vectors.

\begin{lem}\label{Ch5:lem:tightedgedetect}
Suppose $\sigma$ is tight and $D$ is in standard form with
\[w=f_g-f_{g-1} - \dots -f_1 -e_0- e_{1}.\]
Let $U=u_1+u_2$ if we are in Situation~A, and $U=u$ if we are in Situation~B. In either case, $U$ satisfies the inequality $U\cdot w<0$.
\end{lem}
\begin{proof}
Since the vector $U$ is non-zero, we may take $m$ minimal such that $U \cdot f_m<0$. We will prove the lemma for the cases $m\leq g$ and $m>g$ separately.
\paragraph{Case $m\leq g$:} First we suppose that $m\leq g$ holds. There is $A\subseteq\{1,\dots, m-1\}$ such that $\sigma_m-1=\sum_{i\in A} \sigma_i$. If we define $z$ by
\begin{equation}
z=
\begin{cases}
f_m-e_0-e_1-\sum_{i\in A} f_i &\text{if $1\in A$}\\
f_m-f_1-\sum_{i\in A} f_i  &\text{if $1\notin A$,}
\end{cases}
\end{equation}
then we obtain $z\in L$ with $z\cdot f_1=-1$. By direct computation, we obtain the inequalities
\begin{equation}\label{Ch5:eq:tight1}
(w-z)\cdot z= w\cdot f_m -1 \geq -2
\end{equation}
and
\begin{equation}\label{Ch5:eq:tight2}
U\cdot z\leq -U\cdot f_1 + U\cdot f_m = -2 +U\cdot f_m \leq -3.
\end{equation}
By subtracting \eqref{Ch5:eq:tight1} from \eqref{Ch5:eq:tight2}, we see that the inequality
\begin{equation}\label{Ch5:eq:tight3}
(w-z-U)\cdot z>0
\end{equation}
holds. Applying Lemma~\ref{Ch3:lem:usefulbound} to the sum of vertices $w+U$ yields
\begin{equation}\label{Ch5:eq:tight4}
(w-z-U)\cdot z+w\cdot U=(w+U-(U+z))\cdot(z+U)\leq 0.
\end{equation}
Therefore, by combining \eqref{Ch5:eq:tight3} and \eqref{Ch5:eq:tight4} we obtain $w\cdot U\leq -(w-z-U)\cdot z<0$. This completes the proof in the case $m\leq g$.

\paragraph{Case $m>g$:} Now we suppose that $m>g$ holds. In this case, we have
\begin{equation*}
w\cdot U=U\cdot f_g-\sum_{i=1}^{g-1}U\cdot f_i
 \end{equation*}
and $U\cdot f_i\geq 0$ for $1\leq i \leq g$, so we wish to establish an upper bound for $U\cdot f_g$. Let $l>g$ be minimal such that $U\cdot f_l\leq 0$. Observe that such an $l$ exists and is at most $m$. By Lemma~\ref{Ch5:lem:slacksum}, we may pick $A\subseteq\{g+1,\dots , l-1\}$, $B\subseteq \{1, \dots , g-1\}$ with $1\in B$, and $\varepsilon, \delta \in \{0,1\}$, such that the vector
\[x=-f_l + \sum_{i \in B} f_i + \varepsilon f_g + \sum_{j\in A} f_j +\delta(e_0+e_1)\]
is in $L$. Let $C$ be the set $C=\{0, \dots , g-1\}\setminus B$. Note that we have $1\notin C$ and that
\begin{equation}\label{Ch5:eq:tight5}
x+w=-e_l + \sum_{j\in A}e_j + (1+\varepsilon)e_g - \sum_{i\in C} e_i+(\delta-1)(e_0+e_1).
\end{equation}
Since $w+U$ is a sum of vertices, we can apply Lemma~\ref{Ch3:lem:usefulbound} to get
\begin{equation}\label{Ch5:eq:tight6}
0\geq(w+U - (x+w))\cdot(x+w)=(U-x)\cdot(x+w).
\end{equation}
Therefore combining \eqref{Ch5:eq:tight5} and \eqref{Ch5:eq:tight6} gives
\begin{align}
\begin{split}\label{Ch5:eq:tight7}
0 &\geq -(U\cdot f_l+1) + \sum_{j\in A}(U\cdot f_j-1) + (1+\varepsilon)(U\cdot f_g-\varepsilon) - \sum_{i\in C} U\cdot f_i+ 2\delta(\delta-1)
\\ &\geq -1 + (1+\varepsilon)(U\cdot f_g-\varepsilon) - \sum_{i\in C} U\cdot f_i,
\end{split}
\end{align}
where we are also using that $U\cdot f_j\geq1$ for $j\in A$. As $1\notin C$, we also get the inequality,
\begin{equation}\label{Ch5:eq:tight8}
\sum_{i\in C} U\cdot f_i \leq \sum_{i=2}^{g-1}U\cdot f_i =\sum_{i=1}^{g-1}U\cdot f_i -2.
\end{equation}
Using the inequalities in \eqref{Ch5:eq:tight7} and \eqref{Ch5:eq:tight8} we get the upper bound
\begin{equation*}
U\cdot f_g \leq \frac{1}{1+\varepsilon}(\sum_{i=1}^{g-1}U\cdot f_i -1) + \varepsilon
\end{equation*}
for $U\cdot f_g$. This allows us to compute
\begin{align*}
w\cdot U & =U\cdot f_g-\sum_{i=1}^{g-1}U\cdot f_i
\\&\leq \varepsilon - \frac{1}{1+\varepsilon}(\varepsilon \sum_{i=1}^{g-1}U\cdot f_i + 1)
\\&\leq \varepsilon - \frac{\varepsilon U\cdot f_1 + 1}{1+\varepsilon}
\\&= \varepsilon - \frac{2\varepsilon + 1}{1+\varepsilon}.
\end{align*}
Since $\varepsilon \in \{0,1\}$, this implies $w\cdot U<0$ which is the required inequality.
\end{proof}

\begin{lem}\label{Ch5:lem:slackedgedetect}
Suppose $\sigma$ is slack and $D$ is in standard form with
\[u_1=-f_h+f_{h-1} + \dots + f_1.\]
If $w\cdot f_2=0$, then $w\cdot u_2=-1$. If $w\cdot f_2 =-1$, then $u_1\cdot w<0$.
\end{lem}
\begin{proof} As $\sigma$ is slack, we have $u_2=-f_2+f_1$. If $w\cdot f_2=0$, then $w\cdot u_2=-1$. So from now on we assume that $w\cdot f_2=-1$. The proof that $u_1\cdot w<0$ mirrors closely the proof of Lemma~\ref{Ch5:lem:tightedgedetect}.
\paragraph{}Since $w$ is non-zero, we may take $m$ minimal such that $w\cdot f_m>0$. We will prove the lemma for the cases $m\leq h$ and $m>h$ separately.

\paragraph{Case $m\leq h$:} First we suppose that $m\leq h$ holds. By Lemma~\ref{Ch5:lem:slacksum}, we may choose $A\subseteq\{1, \dots, m-1\}$ such that $1\in A$ and the vector
\[z=-f_m+\sum_{i\in A} f_i\]
is in $L$. By a direct computation we obtain the following four inequalities. If $2\in A$, then we have
\begin{equation}\label{Ch5:eq:slack1}
(u_1-z)\cdot z=-u_1\cdot f_m -1\geq -2,
\end{equation}
and
\begin{equation}\label{Ch5:eq:slack2}
w\cdot z\leq w\cdot f_1 + w\cdot f_2 - w\cdot f_m \leq -3.
\end{equation}
On the other hand, if $2\notin A$, then we have
\begin{equation}\label{Ch5:eq:slack3}
(u_1+u_2-z)\cdot z=(u_1-z)\cdot z+u_2\cdot z= -u_1\cdot f_m \geq -1,
\end{equation}
and
\begin{equation}\label{Ch5:eq:slack4}
w\cdot z\leq -w\cdot f_1 +w\cdot f_m \leq -2.
\end{equation}
By subtracting \eqref{Ch5:eq:slack1} from \eqref{Ch5:eq:slack2} or \eqref{Ch5:eq:slack3} from \eqref{Ch5:eq:slack4}, we see that there is a choice, $U=u_1$ or $U=u_1+u_2$ such that the inequality
\begin{equation}\label{Ch5:eq:slack5}
(U-z-w)\cdot z>0
\end{equation}
holds. Applying Lemma~\ref{Ch3:lem:usefulbound} to the sum of vertices $U+w$ yields the bound
\begin{equation}\label{Ch5:eq:slack6}
(U-z-w)\cdot z+w\cdot U=(U+w-(w+z))\cdot(w+z)\leq 0.
\end{equation}
By combining \eqref{Ch5:eq:slack5} and \eqref{Ch5:eq:slack6}, we obtain the bound $w\cdot U<0$. Since $u_2\cdot w=0$, this proves the lemma when $m\leq h$.

\paragraph{Case $m>h$:} Now suppose that $m>h$ holds. In this case, we have
\begin{equation*}
u_1\cdot w= - w\cdot f_h + \sum_{i=1}^{h-1} w\cdot f_i
\end{equation*}
and $w\cdot f_i\leq 0$ for $1\leq i\leq h$, so we need to bound the quantity $-w\cdot f_h$ above. Let $l>h$ be minimal such that $w\cdot f_l\geq 0$. Observe that such an $l$ exists and is at most $m$. By Lemma~\ref{Ch5:lem:slacksum}, we may pick $A\subseteq\{h+1, \dots , l-1\}$, $B\subseteq \{1, \dots , h-1\}$ and $\varepsilon \in \{0,1\}$ such that the vector
\begin{equation*}
x= f_l - \sum_{j\in A} f_j - \varepsilon f_h - \sum_{i\in B}f_i
\end{equation*}
is in $L$. Since $\sigma_2=1$, we may assume that we have $2\in B$.
Let $C$ be the set $C=\{1, \dots , h-1\}\setminus B$. By choice of $B$, we have then $2\notin C$ and
\begin{equation}\label{Ch5:eq:slack7}
x+u_1 = f_l - \sum_{j\in A} f_j - (\varepsilon +1) f_h + \sum_{i\in C}f_i \in L.
\end{equation}
Since $u_2+w+u_1$ is a sum of vertices, we can apply Lemma~\ref{Ch3:lem:usefulbound} to get
\begin{equation}\label{Ch5:eq:slack8}
0\geq (u_2+ w + u_1 -(x+u_1))\cdot(x+u_1)=(u_2 + w - x)\cdot(x+u_1).
\end{equation}
By combining \eqref{Ch5:eq:slack7} and \eqref{Ch5:eq:slack8}, we arrive at the inequality
\begin{align}
\begin{split}\label{Ch5:eq:slack9}
0 &\geq w\cdot f_l-1 - \sum_{j\in A}(w\cdot f_i+1) - (1+\varepsilon)(w\cdot f_h-\varepsilon) + \sum_{i\in C} (w\cdot f_i + u_2\cdot f_i)\\
&\geq -1 - (1+\varepsilon)(w\cdot f_h+\varepsilon) + \sum_{i\in C} (w\cdot f_i + u_2\cdot e_i),
\end{split}
\end{align}
where the second line follows from the fact that $w\cdot f_l\geq 0$ and that $w\cdot f_j \leq -1$ for all $j\in A$. Since $(w+u_2)\cdot f_1=0$, $(w+u_2)\cdot f_2=-2$ and $2\notin C$ all hold, we also get the bound
\begin{equation}\label{Ch5:eq:slack10}
\sum_{i\in C} (w\cdot f_i + u_2\cdot f_i) \geq \sum_{i=3}^{h-1} (w\cdot f_i + u_2\cdot f_i)=\sum_{i=1}^{h-1}w\cdot f_i+2.
\end{equation}
The inequalities \eqref{Ch5:eq:slack9} and \eqref{Ch5:eq:slack10} combine to give
\begin{equation*}
-w\cdot f_h\leq \frac{-1}{1+\varepsilon}(\sum_{i=1}^{h-1}w\cdot f_i +1)+\varepsilon.
\end{equation*}
This then allows us to compute
\begin{align*}
w\cdot u_1&=- w\cdot f_h + \sum_{i=1}^{h-1} w\cdot f_i
\\ &\leq \varepsilon + \frac{1}{1+\varepsilon}(\varepsilon \sum_{i=1}^{h-1}w\cdot f_i -1)
\\ &\leq \varepsilon + \frac{\varepsilon w\cdot (f_1+f_2) -1}{1+\varepsilon}
\\ &= \varepsilon - \frac{2\varepsilon + 1}{1+\varepsilon}.
\end{align*}
Since $\varepsilon \in \{0,1\}$, this gives $w\cdot u_1<0$, which is the required bound.
\end{proof}

It will be useful to subdivide Situation~A into two cases.
\begin{defn}
If $w\cdot f_2=0$, then we say we are in {\em Situation~A2}. Otherwise, we say that we are in {\em Situation~A1}.
\end{defn}
Note that in Situation~A2, $\sigma$ is necessarily slack and we must have $k=2$. Finally we are ready to accomplish the main aim of this section and prove that we can obtain a diagram with the required structure near the marked crossing.

\begin{prop}\label{Ch5:prop:singlecrossingsummary}
If $K$ has an alternating diagram $D$ with a single marked crossing $c$, then there is a sequence of flypes to a minimal diagram, $D'$, which takes the form shown in Figure~\ref{Ch5:fig:localsubgraphs} in the neighbourhood of $c$. Each of the flypes can be chosen to fix $c$ and hence commutes with changing $c$.
\end{prop}

\begin{figure}[ht]
  \centering
  \def\svgwidth{\columnwidth}
  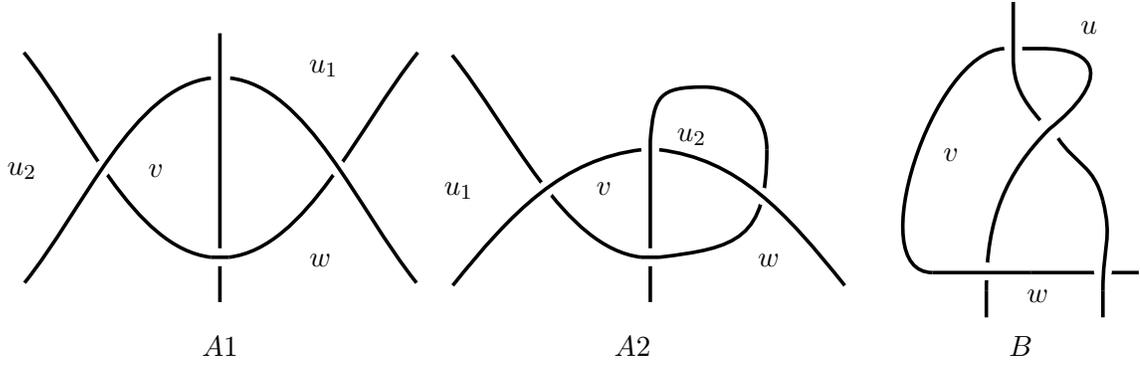
  \caption[The neighbourhood of a marked crossing]{In each of the Situations we may isotope to a diagram appearing as pictured in the neighbourhood of the marked crossing.}
  \label{Ch5:fig:localsubgraphs}
\end{figure}

\begin{proof}
By Lemma~\ref{Ch5:lem:tightflyping} and Lemma~\ref{Ch5:lem:slackflyping}, there is a sequence of flypes fixing $c$ to get $D''$ in standard form, with an adjacent vector of the form
\[\tilde{u}_1=-f_h+f_{h-1} + \dots + f_1,\]
if $\sigma$ is slack or a marker vertex of the form
\[\tilde{w}=f_g-f_{g-1} - \dots-f_1 -e_0 - e_{1},\]
if $\sigma$ is tight.
We prove the proposition for each of the three situations~A1, A2 and B separately, although the argument is similar in all cases.

\paragraph{Situation~A1} In this case we can assume $D''$ has marker vectors $\tilde{w}$ and $v=-f_1+e_0+e_{1}$ and adjacent vectors $\tilde{u}_1$ and $u_2$. Lemma~\ref{Ch5:lem:slackedgedetect} and Lemma~\ref{Ch5:lem:tightedgedetect} show that we may assume $\tilde{u}_1 \cdot \tilde{w}<0$. In particular, this means that $\Gamma_{D''}$ contains a triangle with vertices $v$, $\tilde{w}$ and $\tilde{u}_1$. Choosing such a triangle separates the plane into two regions. Let $R$ be the region not containing $u_2$. We may assume that we have chosen the triangle so that the interior of $R$ does not contain any edges between $u_1$ and $w$. If $R$ does not contain any of the vertices of $\Gamma_{D''}$, then $D''$ has form required by the proposition. So suppose that $R$ does contain a vertex $x$ of $\Gamma_{D''}$. Any such vertex $x$ must satisfy $v \cdot x=0$, so if we take $G_1$ to be the subgraph induced by the set of all vertices contained in $R$, then $G_1$ must be a union of connected components of $\Gamma_{D''} \setminus\{w,u_2\}$. Thus $\Gamma_{D''}$ and $D''$ must appear as in Figure~\ref{Ch5:fig:otherflypes}. Therefore, as in Figure~\ref{Ch4:fig:flype2}, we can perform a flype by twisting the tangle corresponding to $G_1$. Let $D'$ be the diagram obtained by performing this flype. If we let $[G_1]$ denote $[G_1]=\sum_{x\in G_1}x$, then the set of vertices $V_{D'}$ is obtained from $V_{D''}$ by replacing $\tilde{w}$ with $w=\tilde{w}+[G_1]$, $\tilde{u}_1$ with $u_1=\tilde{u}_1+[G_1]$ and $x$ with $-x$ for all $x\in G_1\subseteq V_{D''}$. Thus, the vertices $v$ and $w$ are the marker vertices for $D'$ and $u_1$ and $u_2$ are the adjacent vertices. Altogether this gives us $D'$ in the required form and proves the proposition in Situation~A1.

\begin{figure}[ht]
  \centering
  \def\svgwidth{\columnwidth}
  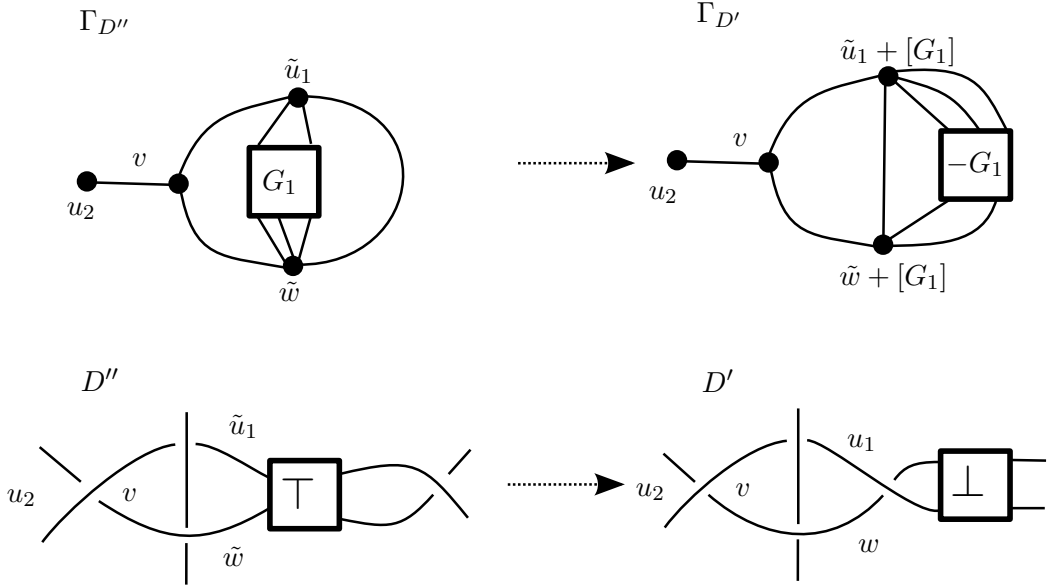
  \caption{The flype from $D''$ to $D'$ in Situation~A1.}
  \label{Ch5:fig:otherflypes}
\end{figure}

\paragraph{Situation~A2}In this case, we can assume $D''$ has marker vertices $v$ and $w$ and adjacent vertices $u_1$ and $u_2$, where $\norm{u_2}=2$ and $u_2\cdot w=u_2\cdot v=-1$. This means $v, \tilde{w}$ and $u_2$ form a triangle in $\Gamma_{D''}$ and there are no further edges from $u_2$. This triangle separates the plane into two regions. Let $R$ be the region not containing $u_1$. Since  $\Gamma_{D''}$ is 2-connected, and none of the vertices adjacent to $v$ and $u_2$ are contained in $R$, we see that the $R$ does not contain any vertex of $\Gamma_{D''}$ in its interior. Thus we may take $D'=D''$ and no further flypes are required to prove the proposition in this case.

\paragraph{Situation~B} In this case, $\sigma$ is necessarily tight and we can assume $D''$ has marker vectors
\[\tilde{w}=f_g-f_{g-1} - \dots -f_1 -e_0 - e_{1}\]
and $v=-f_1+e_0+e_{1}$ and adjacent vector $\tilde{u}$. Lemma~\ref{Ch5:lem:tightedgedetect} shows that $\tilde{u}\cdot \tilde{w}<0$. Thus $v, \tilde{w}$ and $\tilde{u}$ form a triangle in $\Gamma_{D''}$ with two edges between $u$ and $v$. Now let $R$ be a region bounded by this triangle. Any vertex $x$ which is contained in the interior of $R$ must satisfy $x \cdot v=0$, so we see that $x$ and $v$ must be contained in different connected components of $\Gamma_{D''} \setminus\{\tilde{w},\tilde{u}\}$. If we take $G_1$ to be the subgraph induced by the set of all vertices contained in $R$, then we see that $\Gamma_{D''}$ and $D''$ must appear as in Figure~\ref{Ch5:fig:otherflypessitB}. From Figure~\ref{Ch5:fig:otherflypessitB}, we see that, as in Figure~\ref{Ch4:fig:flype2}, there is a flype obtained by twisting the tangle corresponding to $G_1$. Let $D'$ be the the diagram obtained by performing this flype. If we let $[G_1]$ denote $[G_1]=\sum_{x\in G_1}x$, then the set of vertices $V_{D'}$ is obtained from $V_{D''}$ by replacing $\tilde{w}$ with $w=\tilde{w}+[G_1]$, $\tilde{u}$ with $u=\tilde{u}+[G_1]$ and $x$ with $-x$ for all $x\in G_1\subseteq V_{D''}$. Thus $v$ and $w$ are the marker vertices for $D'$ and $u$ is the unique adjacent vertex. Altogether this gives us $D'$ in the required form and proves the proposition in Situation~B.

\begin{figure}[ht]
  \centering
  \def\svgwidth{\columnwidth}
  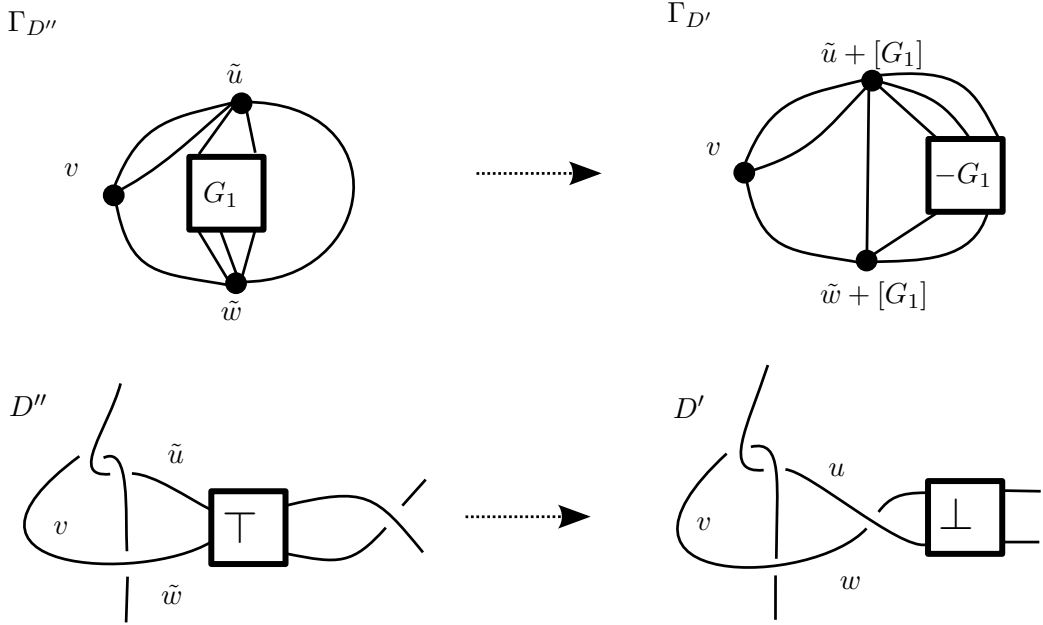
  \caption{The flype from $D''$ to $D'$ in Situation~B.}
  \label{Ch5:fig:otherflypessitB}
\end{figure}

\end{proof}

\section{Changing marked crossings}\label{sec:smallerdiagram}
Now we study the effects of changing a marked crossing. We have already considered the case of multiple marked crossings in Lemma~\ref{Ch5:lem:manymarkedcrossings}, so we suppose there is a single marked crossing, $c$. By Proposition~\ref{Ch5:prop:singlecrossingsummary}, we may assume that we have flyped so that the neighbourhood of the marked crossing is as shown in Figure \ref{Ch5:fig:localsubgraphs}. Let $K'$ be the knot obtained by changing the marked crossing. This crossing change gives an almost-alternating diagram $D'$. In Situation~A1, there is an obvious untongue move and in Situations~A2 and B there are obvious untwirl moves. Let $\widetilde{D}'$ be the new almost-alternating diagram obtained by performing these moves. There is a crossing, $\tilde{c}$, in $\widetilde{D}'$, such that changing $\tilde{c}$ gives an alternating diagram $\widetilde{D}$. In each case, the embedding of $\Lambda_{D}$ can be modified in a natural way to obtain an embedding of $\Lambda_{\widetilde{D}}$ into $\mathbb{Z}^{r+2}$. These operations are illustrated in Figure~\ref{Ch5:fig:sitA1embeddings}, Figure~\ref{Ch5:fig:sitA2embeddings} and Figure~\ref{Ch5:fig:sitBembeddings}. These new embeddings will be used to provide the induction step in a proof that $K'$ is the unknot.

\begin{lem}\label{Ch5:lem:inductstep}
The image of $\Lambda_{\widetilde{D}}$ under the given embedding is a changemaker lattice and $\tilde{c}$ is a marked crossing for this embedding.
\end{lem}
\begin{proof}
Let $S = \{\nu_1, \dots, \nu_r\} \subseteq \mathbb{Z}^{r+2}$ be a standard basis for $L$ as in Section~\ref{Ch3:sec:standardbases}. Since $V_D\setminus\{w\}=\{x_1, \dots, x_{r}\}$ is also basis for $L$, we can consider the change of basis matrix $(A_{ij})$ satisfying
\[\nu_i = \sum_{j} A_{ij}x_j.\]
In fact, since each element of $S$ is irreducible, Lemma~\ref{Ch3:lem:irreducible} implies that it can be written as a sum of vertices and hence that for each $i$ we have either $0\leq A_{ij}\leq 1$ for all $j$ or $-1\leq A_{ij}\leq 0$ for all $j$ .

\paragraph{} We consider each of the Situations~A1, A2, and B separately. Although in each case the proofs are virtually identical -- all that differs is the new embedding. In each case we use a minor of the matrix $A$ to construct a basis for $\Lambda_{\widetilde{D}}$ satisfying the hypotheses of Lemma~\ref{Ch3:lem:halfintidentification}.

\paragraph{Situation~A1} If we write the vertices $u_1,u_2$ and $w$ in the form $u_1=f_1 + \tilde{u}_1$, $u_2=f_1 + \tilde{u}_2$ and $w=-f_1 + \tilde{w}$, then $\Lambda_{\widetilde{D}}$ has an embedding into $\mathbb{Z}^{r+1} = \langle e_0, e_1, f_2, \dots, f_r \rangle$ with vertices $V_{\widetilde{D}}$ obtained by replacing $\{v,w,u_1,u_2\}$ in $V_D$ by $\{\tilde{u}_2+e_0+e_1, \tilde{w}, \tilde{u}_1\}$. This is shown in Figure~\ref{Ch5:fig:sitA1embeddings}.
\begin{figure}[ht]
  \centering
  \def\svgwidth{300pt}
  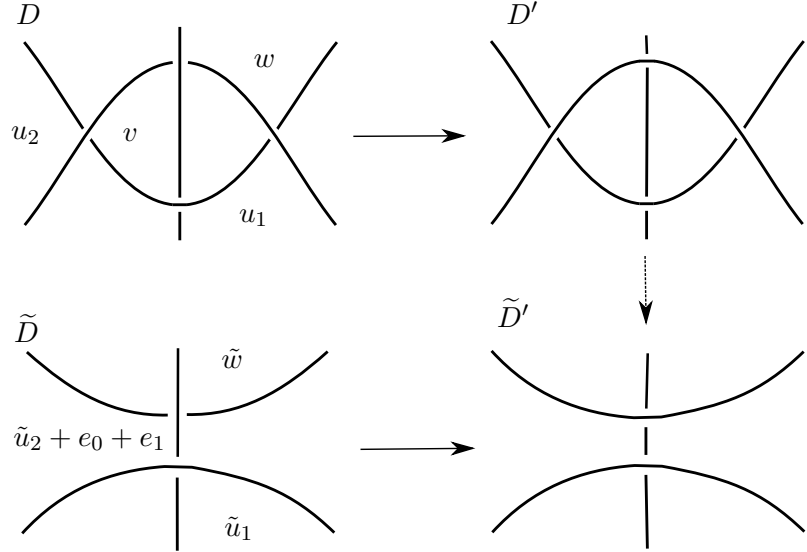
  \caption[The embedding for $\widetilde{D}$ in Situation~A1.]{The embedding for $\widetilde{D}$ in Situation~A1. The diagram $D'$ is obtained from $D$ by changing the marked crossing. The diagrams $D'$ and $\widetilde{D}'$ are almost-alternating and related by an tongue move.}
  \label{Ch5:fig:sitA1embeddings}
\end{figure}
In this case, we enumerate the elements of $V_D\setminus\{w\}$ as
\[V_D\setminus\{w\}= \{x_1, \dots, x_{r}\},\]
where $x_1=v, u_1=x_2$ and $u_2=x_3$. Since $\nu_1=v=x_1$, the matrix $A$ takes the form
\[A=\begin{pmatrix}
   1 & 0 & \dotsb & 0\\
   *  &  &\tilde{A} &
  \end{pmatrix}.\]
It follows that the bottom-right $(r-1)\times(r-1)$-minor satisfies $\det(\widetilde{A})=\det(A)= \pm1$. Thus if we set
\[\tilde{\nu}_i=A_{i2}\tilde{u}_1+ A_{i3}(\tilde{u}_2 + e_0+e_1) +\sum_{j=4}^r A_{ij}x_j\]
for $2\leq i \leq r$, then $\{\tilde{\nu}_2, \dots , \tilde{\nu}_r\}$ forms a basis for $\Lambda_{\widetilde{D}}$. We will show that this basis satisfies the hypotheses of Lemma~\ref{Ch3:lem:halfintidentification}. It is clear that $\tilde{\nu}_i \cdot f_j= \nu_i\cdot f_j$ for all $i,j\geq 2$. In particular $\tilde{\nu}_k$ takes the form
\[\tilde{\nu}_k=-f_k + f_{k-1} + A_{k3}(e_0+e_1) +\sum_{i\in A_k} f_i,\]
for some $A_k\subseteq\{2,\dots, k-2\}$ if $k\geq 3$ and
\[\tilde{\nu}_2=-f_2+A_{23}(e_0+e_1).\]
Observe that $\nu_i \cdot f_1= A_{i2}+A_{i3}-A_{i1} \in \{0,1\}$ and $\nu_i\cdot e_0 = A_{i1}\in \{0,1\}$. It follows that $A_{i3}\in\{0,1\}$, and hence that $\tilde{\nu}_i \cdot e_0=\tilde{\nu}_i \cdot e_1=A_{i3}\in \{0,1\}$. It remains to check that $\tilde{\nu}_2=-f_2+e_0+e_1$. If $\sigma_2=1$, then $u_2=\nu_2=f_1-f_2$, which implies that $\tilde{\nu}_2=-f_2+e_0+e_1$, as required. If $\sigma_2=2$, then $\nu_2=-f_2+f_1+e_0+e_1$, which implies that $A_{21}=A_{22}=A_{23}=1$ and hence that $\tilde{\nu}_2=-f_2+e_0+e_1$ as required.

Thus $\{\tilde{\nu}_2, \dots , \tilde{\nu}_r\}\subseteq \langle e_0, e_1, f_2, \dots, f_r \rangle$ satisfies the hypotheses of Lemma~\ref{Ch3:lem:halfintidentification} and hence that the image of $\Lambda_{\widetilde{D}}$ is a changemaker lattice in $\Z^{r+1}$. It is clear that that $\tilde{c}$ is a marked crossing for this embedding.

\paragraph{Situation~A2} If we write the vertices $u_1$ and $w$ in the form $u_1=f_1 + f_2 + \tilde{u}$ and $w=-f_1 + \tilde{w}$, then $\Lambda_{\widetilde{D}}$ has an embedding into $\mathbb{Z}^{r} = \langle e_0, e_1, f_3, \dots, f_r \rangle$ with vertices $V_{\widetilde{D}}$ obtained by replacing $\{v,w,u_1,u_2\}$ in $V_D$ by $\{\tilde{u}+e_0+e_1, \tilde{w}\}$. This is shown in Figure~\ref{Ch5:fig:sitA2embeddings}.
\begin{figure}[ht]
  \centering
  \def\svgwidth{350pt}
  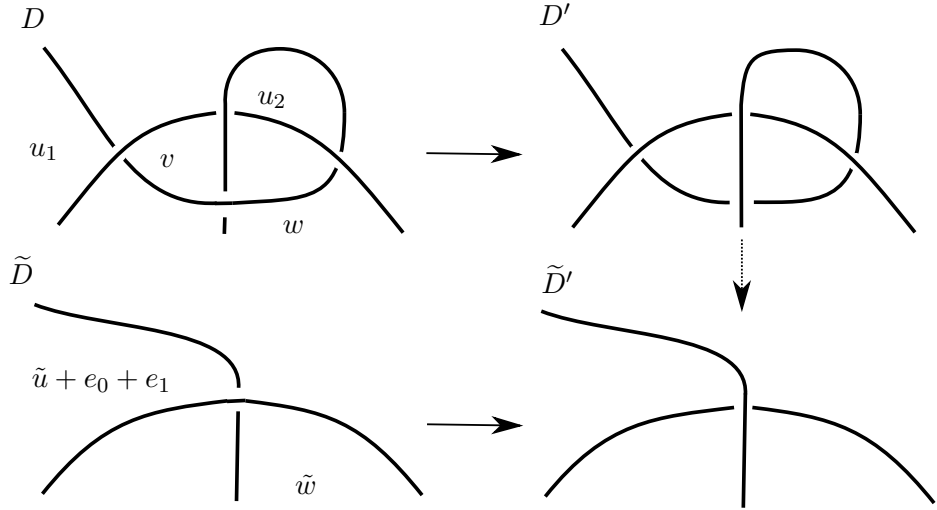
  \caption[The embeddings for $\widetilde{D}$ in Situation~A2]{The embedding for $\widetilde{D}$ in Situation~A2.}
  \label{Ch5:fig:sitA2embeddings}
\end{figure}

In this case enumerate the elements of $V_D\setminus\{w\}$ as
\[V_D\setminus\{w\}= \{x_1, \dots, x_{r}\},\]
where $x_1=v, u_2=x_2$ and $u_1=x_3$. Since $\nu_1=v=x_1$ and $\nu_2=u_2=x_2$, $A$ takes the form
\[A=\begin{pmatrix}
   1 & 0 & 0 & \dotsb & 0\\
   0 & 1 & 0 & \dotsb & 0 \\
   * & * &   &\tilde{A} &
  \end{pmatrix}.\]
It follows that the bottom-right $(r-2)\times(r-2)$-minor satisfies $\det(\widetilde{A})=\det(A)= \pm1$. Thus if we set
\[\tilde{\nu}_i=A_{i3}(\tilde{u}+e_0+e_1) +\sum_{j=4}^r A_{ij}x_j\]
for $3\leq i \leq r$, then $\{\tilde{\nu}_3, \dots , \tilde{\nu}_r\}$ forms a basis for $\Lambda_{\widetilde{D}}$. It is clear that $\tilde{\nu}_i \cdot f_j= \nu_i\cdot f_j$ for all $i,j\geq 3$. In particular $\tilde{\nu}_k$ takes the form
\[\tilde{\nu}_k=-f_k + f_{k-1} + A_{k3}(e_0+e_1) +\sum_{i\in A_k} f_i,\]
for some $A_k\subseteq\{3,\dots, k-2\}$ if $k\geq 3$ and
\[\tilde{\nu}_3=-f_3+A_{33}(e_0+e_1).\]
By proceeding as in Situation~A1, one can show that we always have $A_{k3}\in \{0,1\}$, and, furthermore that $\tilde{\nu}_3=-f_3+e_0+e_1$.

Thus $\{\tilde{\nu}_3, \dots , \tilde{\nu}_r\}\subseteq \langle e_0, e_1, f_3, \dots, f_r \rangle$ satisfies the hypotheses of Lemma~\ref{Ch3:lem:halfintidentification} and hence that the image of $\Lambda_{\widetilde{D}}$ is a changemaker lattice in $\Z^{r}$. It is clear that that $\tilde{c}$ is a marked crossing for this embedding.

\paragraph{Situation~B} If we write the vertices $u$ and $w$ in the form $u=2f_1 + \tilde{u}$ and $w=-f_1 + \tilde{w}$, then $\Lambda_{\widetilde{D}}$ has an embedding into $\mathbb{Z}^{r+1} = \langle e_0, e_1, f_2, \dots, f_r \rangle$ with vertices $V_{\widetilde{D}}$ obtained by replacing $\{v,w,u\}$ in $V_D$ by $\{\tilde{u}+e_0+e_1, \tilde{w}\}$. This is shown in Figure~\ref{Ch5:fig:sitBembeddings}.

\begin{figure}[ht]
  \centering
  \def\svgwidth{300pt}
  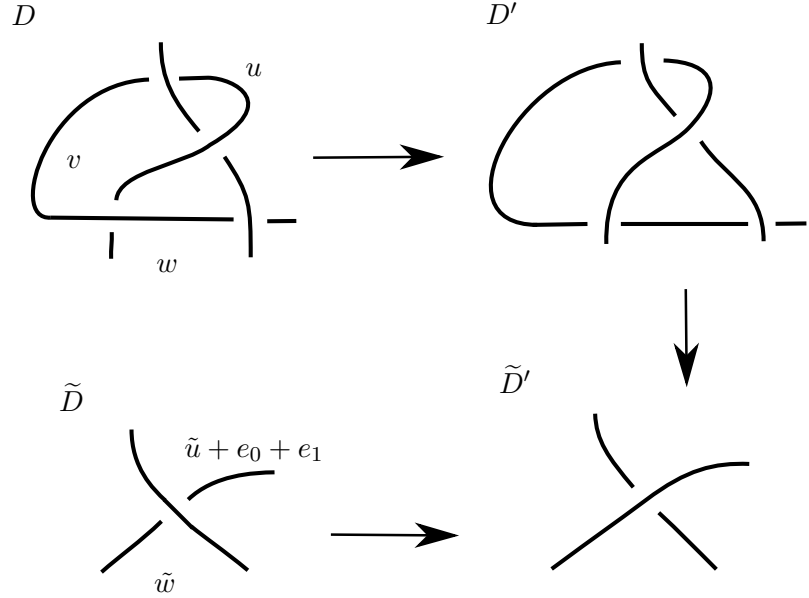
  \caption[The embeddings for $\widetilde{D}$ in Situation~B]{The embedding for $\widetilde{D}$ in Situation~B}
  \label{Ch5:fig:sitBembeddings}
\end{figure}
In this case, we enumerate the elements of $V_D\setminus\{w\}$ as
\[V_D\setminus\{w\}= \{x_1, \dots, x_{r}\},\]
where $x_1=v$ and $u=x_2$. Since $\nu_1=v=x_1$, the matrix $A$ takes the form
\[A=\begin{pmatrix}
   1 & 0 & \dotsb & 0\\
   *  &  &\tilde{A} &
  \end{pmatrix}.\]
It follows that the bottom-right $(r-1)\times(r-1)$-minor satisfies $\det(\widetilde{A})=\det(A)= \pm1$. Thus if we set
\[\tilde{\nu}_i=A_{i2}(\tilde{u}+ e_0+e_1) +\sum_{j=3}^r A_{ij}x_j\]
for $2\leq i \leq r$, then $\{\tilde{\nu}_2, \dots , \tilde{\nu}_r\}$ forms a basis for $\Lambda_{\widetilde{D}}$. Proceeding as in the preceding two situations, we find that $\tilde{\nu}_k$ takes the form
\[\tilde{\nu}_k=-f_k + f_{k-1} + A_{k2}(e_0+e_1) +\sum_{i\in A_k} f_i,\]
for some $A_k\subseteq\{2,\dots, k-2\}$ if $k\geq 3$ and
\[\tilde{\nu}_2=-f_2+A_{22}(e_0+e_1).\]
Furthermore, we find that we always have $A_{k2}\in \{0,1\}$ and that $\tilde{\nu}_2=-f_2+e_0+e_1$.

Thus $\{\tilde{\nu}_2, \dots , \tilde{\nu}_r\}\subseteq \langle e_0, e_1, f_2, \dots, f_r \rangle$ satisfies the hypotheses of Lemma~\ref{Ch3:lem:halfintidentification} and hence that the image of $\Lambda_{\widetilde{D}}$ is a changemaker lattice in $\Z^{r+1}$. It is clear that that $\tilde{c}$ is a marked crossing for this embedding.
\end{proof}

\section{The main results}\label{sec:mainresults}
Putting together the results of this chapter we obtain the following theorem.
\begin{thm}\label{Ch5:thm:markedmeansunknotting}
Suppose that $D$ is a reduced alternating diagram with $\Lambda_D$ is isomorphic to a half-integer changemaker lattice. Any marked crossing is an unknotting crossing which is negative if $\sigma(K)=0$ and positive if $\sigma(K)=-2$. Furthermore, if $D'$ is the almost-alternating diagram obtained by changing a marked crossing, then $D'$ can be reduced to one of the clasp diagrams $\mathcal{C}_m$ by a sequence of flypes, untongue and untwirl moves.
\end{thm}
\begin{proof}
We show by induction on the rank of $\Lambda_D$ that the diagram obtained by changing a marked crossing in $D$ can be reduced to one of the clasp diagrams $\mathcal{C}_m$ by a sequence of flypes, untongue and untwirl moves.
First suppose that $\Lambda_D$ has rank one. As shown in Example~\ref{Ch5:exam:trefoil}, this means $D$ is a minimal diagram of the right-handed trefoil. In particular, when $r=1$ every crossing in $D$ is both marked and an unknotting crossing. The diagram obtained by changing one of these crossings is clearly a clasp diagram.

Suppose now that $\Lambda_D$ has rank $r>1$. If there is more than one marked crossing, then Lemma~\ref{Ch5:lem:manymarkedcrossings} shows that any one of them is an unknotting crossing and that the diagram obtained by changing a marked crossing is a clasp diagram. Thus we need only consider the case that there is a single marked crossing, $c$. Let $D'$ be the knot obtained by changing $c$. By Proposition~\ref{Ch5:prop:singlecrossingsummary} we can flype $D'$ to an almost-alternating diagram which admits an untongue or an untwirl move. Let $\widetilde{D}'$ be the diagram obtained by performing this move. By Lemma~\ref{Ch5:lem:inductstep} we can assume $\widetilde{D}'$ is obtained by changing a marked crossing in a reduced alternating diagram $\widetilde{D}$ with the rank of $\Lambda_{\widetilde{D}}$ strictly less than $r$. Therefore, by the inductive hypothesis, $D'$ is an almost-alternating diagram of the unknot which can be reduced to a clasp diagram by a sequence of flypes, untongue and untwirl moves.

The statement about the sign of the marked crossing follows from Lemma~\ref{Ch5:lem:manymarkedcrossings} and the fact that the unknot has signature 0.
\end{proof}

\begin{proof}[Proof of Theorem~\ref{Ch5:thm:technical}]
The Montesinos trick, as stated in Proposition~\ref{Ch4:prop:Montesinos}, gives $(i) \Rightarrow (ii)$. Corollary~\ref{Ch4:cor:CMcondit} gives $(ii)\Rightarrow(iii)$. The implication $(iv)\Rightarrow (i)$ is trivial.
\paragraph{} The implication $(iii) \Rightarrow (iv)$ follows from Lemma~\ref{Ch5:lem:markedcexist}, which shows that the diagram $D$ contains a marked crossing and Theorem~\ref{Ch5:thm:markedmeansunknotting}, which shows that this marked crossing is an unknotting crossing of the required sign.
\end{proof}

\begin{proof}[Proof of Theorem~\ref{Intro:thm:unknotting}]
The implication $(iii)\Rightarrow (i)$ is trivial. The rest of the proof follows more or less immediately from Theorem~\ref{Ch5:thm:technical}. Since the unknotting number is preserved under reflection, and for any knot $K$ we $\Sigma(K)=S^3_{-d/2}(c)$ for some $c\subseteq S^3$ if and only if $-\Sigma(K)=\Sigma(\overline{K})=S^3_{d/2}(\overline{c})$,
Theorem~\ref{Ch5:thm:technical} gives the equivalence $(i)\Leftrightarrow (ii)$.

Similarly, Theorem~\ref{Ch5:thm:technical} shows that if an alternating knot $K$ has unknotting number one, then any reduced alternating diagram for $K$ contains an unknotting crossing. Since introducing nugatory crossings does not affect the property of having an unknotting crossing, this shows that every alternating diagram for $K$ must contain an unknotting crossing. This proves $(i)\Rightarrow (iii)$ and completes the proof.
\end{proof}

\chapter{Non-integer surgeries}\label{chap:nonint}
After using half-integer changemaker lattices to study the alternating knots whose double-branched covers arise as half-integer surgery in the preceding chapter, we now extend this analysis to arbitrary non-integer changemaker lattices to obtain our diagrammatic criterion for when the double branched cover of an alternating knot or link arises by non-integer surgery.

\begin{thm}\label{intro:thm:nonint}
Let $L$ be an alternating knot or link. For any $p/q$ with $|p|>q>1$, the double branched cover $\Sigma(L)$ arises as $p/q$-surgery on a knot in $S^3$ if and only if (1) $L$ possesses an alternating diagram obtained by rational tangle replacement from an almost-alternating diagram of the unknot and (2) the corresponding surgery slope $p/q$.
\end{thm}

This tangle replacement is described more precisely in the following statement.
\begin{thm} \label{Ch6:thm:rationalsurgery}
Let $L$ be an alternating link. Let $p>q>1$ be coprime integers and suppose that $p/q=n-r/q$, where $0<r<q$. The following are equivalent:
\begin{enumerate}[(i)]
\item there is $\kappa \subseteq S^3$ with $\Sigma(L)=S^3_{-p/q}(\kappa)$;
\item for any reduced alternating diagram $D$ of $L$ the Goeritz form $\Lambda_D$ of $D$ is isomorphic to a $p/q$-changemaker lattice;
\item $L$ has a reduced alternating diagram $D$ containing a rational tangle $T$ of slope $\frac{q-r}{r}$. such that if $D'$ is the alternating diagram obtained by replacing $T$ with a single crossing $c$, then either ${\sigma(D')=-2}$ and $c$ is a positive unknotting crossing or ${\sigma(D')=0}$ and $c$ is negative unknotting crossing.
\end{enumerate}
\end{thm}
The convention we are using for the slope of a rational tangle in an alternating diagram is set out in Section~\ref{Ch4:sec:tanglesindiagrams}. The condition that $p>q$ in Theorem~\ref{Ch6:thm:rationalsurgery} is necessary for $(iii)$ to hold. However, when $|p/q|\leq 1$ the situation is even simpler. The following is a straightforward consequence of known bounds on $L$-space surgeries.
\begin{prop}\label{Ch6:prop:smallsurgery}
If $S^3_{p/q}(\kappa)$ is an alternating surgery for some $|p/q|< 1$, then $\kappa$ is the unknot. In particular, $\Sigma(L)$ is a lens space and $L$ is a 2-bridge link.
\end{prop}
\begin{proof}
If $S^3_{p/q}(\kappa)$ is an alternating surgery, then $\kappa$ is an $L$-space knot \cite{ozsvath2005heegaard}. Therefore we must have have the bound \cite{ozsvath2011rationalsurgery}:
\[2g(\kappa)-1\leq |p/q|.\]
Therefore, $|p/q|<1$ implies that $\kappa$ is the unknot. The proposition follows since surgery on the unknot yields lens spaces, which are the branched double covers of 2-bridge links.
\end{proof}

\section{The fractional tangle}
The objective of this section is to show that if $D$ is a reduced alternating diagram with
\[\Lambda_D \cong \langle w_0, \dots, w_l\rangle^\bot \subseteq \mathbb{Z}^{t+s+1},\]
where $L=\langle w_0, \dots, w_l\rangle^\bot$ is a $p/q$-changemaker lattice with $p>q\geq 2$, then there is a sequence of flypes to obtain an alternating diagram which can be obtained by rational tangle replacement on an almost-alternating diagram of the unknot. We fix a choice of isomorphism,
\[\iota_D: \Lambda_D \longrightarrow L.\]
This gives a distinguished collection of vectors in $L$ given by the image of the vertices of $\Gamma_D$. We call this collection $V_D$, and, in an abuse of notation, we will fail to distinguish between a vertex of $\Gamma_D$ and the corresponding element in $V_D$.

\begin{rem}
Lemma~\ref{Ch3:lem:fracindecomp} implies that $\Lambda_D$ is indecomposable. Thus Lemma~\ref{Ch3:lem:2connectgraphlat} shows that $\Gamma_D$ is 2-connected and any $v \in V_D$ is irreducible. By Lemma~\ref{Ch3:lem:fracpartirred}, this shows for any $v\in V_D$ the fractional part $v_F\in L_F$ is also irreducible.
\end{rem}

It will be necessary for us to flype $D$ to obtain a new reduced alternating diagrams. As in the previous chapter, these flypes will be either an application of Lemma~\ref{Ch4:lem:flype1} or a flype as appearing in Figure~\ref{Ch4:fig:flype2}. In either case, if $D'$ is the diagram we obtain from such a flype, then we get a natural choice of $V_{D'}\subseteq L$ and hence an isomorphism
\[\iota_{D'}: \Lambda_{D'} \rightarrow L.\]
Whenever we flype, we will implicitly use these choices of isomorphism to speak of $V_{D'}$ without ambiguity.

\paragraph{}In order to prove Theorem~\ref{Ch6:thm:rationalsurgery}, we will study the vertices in $V_D$ which have a non-zero fractional part. This will allow us to obtain a diagram in which $L_F$ specifies a tangle. We will take $\mu_0, \dots, \mu_m$ to be the basis of $L_F$ as constructed in Section~\ref{Ch3:sec:fracparts}.
First we show that there is a sequence of flypes to a diagram in which $\mu_1,\dots , \mu_m$ are vertices.

\begin{lem}\label{Ch6:lem:flypetorational}
We may flype to obtain a diagram in which $\mu_1,\dots , \mu_m$ are vertices.
\end{lem}
\begin{proof}
Let $1\leq c\leq m$ be maximal such that $\mu_c$ is not a vertex. By Lemma~\ref{Ch3:lem:fracpartirred}, $\mu_c$ is irreducible. Therefore, Lemma~\ref{Ch3:lem:irreducible} shows that there is $R\subseteq V_D$ such that $\mu_c=\sum_{x\in R} x$. In particular, this implies that there exists some $u \in V_D$ with $u_F=\mu_a+ \dots + \mu_b$, for $a\leq c \leq b$. As $\mu_c$ is assumed not to be a vertex, we have $a<b$. Therefore, we have
\[\mu_b \cdot u = \mu_b \cdot u_F = \norm{\mu_b} -1>0,\]
which implies that $\mu_b$ is not a vertex. This shows $u_F$ takes the form $u_F=\mu_a+\dots + \mu_c$, for $a<c$.

Since $(u-\mu_c)\cdot \mu_c=\mu_{c-1}\cdot \mu_c =-1$, we may apply a flype as given by Lemma~\ref{Ch4:lem:flype1}. This yields a new diagram $D'$ with $\mu_c\in V_{D'}$. Since $\mu_c \cdot \mu_k \leq 0$ for all $k\ne c$, it follows that $\mu_{c+1}, \dots, \mu_m$ are still vertices. Iterating this process shows that we may flype to obtain a diagram $\widetilde{D}$ with $\mu_1,\dots , \mu_m$ in $V_{\widetilde{D}}$.
\end{proof}
From now on we will assume that we have performed flypes as in Lemma~\ref{Ch6:lem:flypetorational}, and have a diagram $D$ with $\mu_1,\dots , \mu_m\in V_D$.
\begin{lem}\label{Ch6:lem:markersunique}
There is a unique vertex $v$ with $v\cdot e_0>0$ and this has $v_F=\mu_0$. There is a unique vertex $w$ with $w\cdot e_0<0$ and this vertex has fractional part $w_F = -(\mu_0 + \dots + \mu_m)$.
\end{lem}
\begin{proof}
First we will show that any vertex $v\in V_D$ with $v\cdot e_0>0$ necessarily has $v_F=\mu_0$. Since $v_F$ is irreducible, it is of the form $v_F=\mu_0+\dots + \mu_b$. As $\mu_b \cdot v >0$, we have $\mu_b \not\in V_D$. This shows that $v_F=\mu_0$. Now suppose that $w$ is any vertex with $w\cdot e_0 <0$. By irreducibility of $w_F$, it follows that $w_F= -(\mu_0 + \dots + \mu_c)$. If $c<m$, then we have $\mu_{c+1}\in V_D$ and we get a contradiction from the fact that $w\cdot \mu_{c+1}=1>0$. This proves the statement about the fractional part of $w$.

Now we show the uniqueness of $v$. Since $\sum_{x \in V_D} x\cdot e_0 =0$, this is also implies the uniqueness of $w$. Suppose there are vertices $u_1$ and $u_2$, with $(u_1)_{F}=(u_2)_{F} =\mu_0$. To show they cannot be distinct, we consider the vector $U=u_1 + u_2$, which satisfies $U_F=2\mu_0$. Now let $g$ be minimal such that $U\cdot f_g \leq 0$. By Proposition~\ref{Ch3:prop:CMcondition}, there is $A\subseteq \{1, \dots , g-1 \}$, such that $\sigma_g-1 = \sum_{i \in A}\sigma_i$. Hence we may take $z=-f_g + \mu_0 + \sum_{i \in A} f_i \in L$.
 Now consider the inequality,
 \begin{align*}
 (U-z)\cdot z &= -U\cdot f_g -1 + \sum_{i \in A}(U\cdot f_i -1) + \norm{\mu_0} \\
    &\geq -1 + \norm{\mu_0}>0.
 \end{align*}
 Since this exceeds the bound in Lemma~\ref{Ch3:lem:usefulbound}, it follows that $U$ is not the sum of distinct vertices. In particular, this implies $u_1=u_2$, which gives the required uniqueness statement.
 \end{proof}

Now let $v$ and $w$ be the vertices as determined by Lemma~\ref{Ch6:lem:markersunique}. We wish to determine the number of edges between $v$ and $w$. These edges, along with the vertices $\mu_1,\dots, \mu_m$, will provide the rational tangle we are seeking. First we need some bounds on vectors in $L$.

\begin{lem}\label{Ch6:lem:coefbound}
If $x = \sum_{u\in R}u \in L$ for some $R\subseteq V_D$, then $|x\cdot f_1|\leq 2$. Furthermore, if $x$ is irreducible with $x\cdot e_0\ne 0$, then $|x\cdot f_1|\leq 1$.
\end{lem}
\begin{proof}
We may assume $x\cdot f_1\geq 0$. Let $g>1$ be minimal such that $x\cdot f_g\leq 0$. By Proposition~\ref{Ch3:prop:CMcondition}, we may write $\sigma_g -1 = \sum_{i \in A} \sigma_i$ for some $A \subseteq \{1, \dots , g-1\}$. So we have $z=-f_g + f_1 + \sum_{i \in A} f_i\in L$. Let $z\cdot f_1 =\varepsilon$ and observe that $\varepsilon \in \{1,2\}$. By Lemma~\ref{Ch3:lem:usefulbound}, $(x-z)\cdot z\leq 0$. This gives
\begin{align*}
0\geq (x-z)\cdot z&= -x\cdot f_g -1 + \sum_{i\in A\setminus \{1\}}(x\cdot f_i -1) + \varepsilon(x\cdot f_1-\varepsilon)\\
    &\geq -1 + \varepsilon(x\cdot f_1-\varepsilon).
\end{align*}
Therefore, $x\cdot f_1\leq \varepsilon + \frac{1}{\varepsilon}\leq \frac{5}{2}$, which gives the required bound.

\paragraph{}Suppose now that $x$ is irreducible, and $x\cdot e_0 > 0$. By Lemma~\ref{Ch3:lem:irredfracpart} and Lemma~\ref{Ch3:lem:fracpartirred}, this implies that $x\cdot e_0=1$. Let $g>0$ be minimal such that $x\cdot f_g\leq 0$. If $g=1$, then let $z=-f_1 + x_F$. By irreducibility of $x$, we have
\[(x-z)\cdot z= -(x\cdot f_1+1)\leq 0,\]
which implies $-1\leq x\cdot f_1 \leq 0$, as required.
Suppose $g>1$, we may write $\sigma_g -1 = \sum_{i \in A} \sigma_i$ for some $A \subseteq \{1, \dots , g-1\}$. If $1\in A$, set $z=-f_g + x_F + \sum_{i \in A} f_i$. If $1\notin A$, set $z=-f_g + f_{1} + \sum_{i \in A} f_i$. In either eventuality, we get $z\in L$ with $z\cdot f_1=1$. By irreducibility, it follows that either $x=z$ or
\[-1\geq (x-z)\cdot z \geq x\cdot f_1 - 1-(x\cdot f_g+1)\geq x \cdot f_1 -2,\]
which implies $x \cdot f_1 = 1$. Thus we get the necessary bounds on $x\cdot f_1$ in all cases.
\end{proof}

\begin{lem}\label{Ch6:lem:markerbounds} The vertices $v$ and $w$ satisfy the inequality
$v\cdot w \leq v_F \cdot w_F + 1$.
\end{lem}
\begin{proof}
By Lemma~\ref{Ch3:lem:generalirred}, the vector $z=-f_1 + \mu_0$ is irreducible. So, by Lemma~\ref{Ch3:lem:irreducible}, there is $R\subseteq V_D$, such that $z=\sum_{x \in R}x$. Since $z\cdot e_0=1$, it follows that $w\notin R$ and $v\in R$. Thus $z-v+w$ is also a sum of vertices. Applying Lemma~\ref{Ch6:lem:coefbound} gives
\[(z-v+w)\cdot f_1 = -1 -v\cdot f_1 + w\cdot f_1 \geq -2.\]
This implies $v\cdot f_1 \leq 0$ or $w \cdot f_1 \geq 0$. We will now prove the lemma for the case $v\cdot f_1 \leq 0$. The argument can easily be modified to treat the case $w \cdot f_1 \geq 0$.

\paragraph{}Suppose that $v\cdot f_1 \leq 0$ holds. By Lemma~\ref{Ch6:lem:coefbound}, this implies $v\cdot f_1 \in \{0,-1\}$. As before, we consider $z=-f_1 + \mu_0$. If $v\cdot f_1 = -1$, then the irreducibility of $v$ implies $v=z$ and we have
\[v\cdot w=-w\cdot f_1+w_F\cdot \mu_0\leq 1 + w_F\cdot \mu_0=1+w_F \cdot v_F,\]
which is the required bound.
If $v \cdot f_1 =0$, then $(v-z)\cdot z =-1$. Since
\[v\cdot w = (v-z)\cdot w + z\cdot w = (v-z)\cdot w - w\cdot f_1 + v_F \cdot w_F,\]
we are required to show $(v-z)\cdot w - w\cdot f_1 \leq 1$. However, we have $(v-z)\cdot w \leq 1$, by Lemma~\ref{Ch3:lem:cutedge} and $w\cdot f_1\geq -1$, by Lemma~\ref{Ch6:lem:coefbound}. Thus it suffices to show that $w\cdot (v-z)\leq 0$ or $w\cdot f_1\geq 0$. Suppose $(v-z)\cdot w = 1$. By Lemma~\ref{Ch3:lem:cutedge}, this implies the existence of a vertex $u\notin \{w,v\}$ such that $u\cdot z=1$ and $u+w$ is irreducible. The condition $u\cdot z =1$ implies that $u\notin \{\mu_1,\dots , \mu_m\}$, so such a $u$ satisfies $u_F=0$. Therefore we are required to have $u\cdot f_1=-1$. Using the irreducibility of $u+w$, Lemma~\ref{Ch6:lem:coefbound} implies
\[(u+w)\cdot f_1 = -1 + w\cdot f_1 \geq -1.\]
This implies that $w \cdot f_1 \geq 0$, which completes the proof for this case.
\paragraph{} If $w \cdot f_1 \geq 0$, then one can consider $z=f_1-(\mu_0 + \dots + \mu_m)$, which allows one to carry out an almost identical argument.
\end{proof}

The vertices $\mu_1, \dots , \mu_m$ form a path between $v$ and $w$ in $\Gamma_D$ and there are at least $|w_F \cdot v_F| -1$ edges between $v$ and $w$. Performing some flypes of the form given by Figure~\ref{Ch4:fig:flype2} if necessary, we may assume that there is a disk in the plane whose boundary intersects only the regions $v$ and $w$ and its interior contains precisely $|w_F \cdot v_F| -1$ of the crossings between $v$ and $w$, the regions $\mu_1, \dots, \mu_m$, and all crossings incident to them. In particular, this means that $v, w, \mu_1, \dots , \mu_m$ are the regions of a subtangle in $D$, and the white graph $\Gamma_D$ in the interior of the disk is as in Figure~\ref{Ch6:fig:fractangle}. We will call this tangle the {\em fractional tangle}.
\begin{figure}[h]
  \centering
  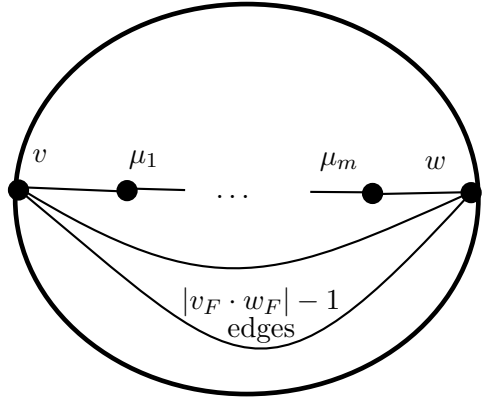
 \caption{The white graph of the fractional tangle.}
 \label{Ch6:fig:fractangle}
\end{figure}
\begin{lem}\label{Ch6:lem:fractionalslope}
The fractional tangle is a rational tangle of slope $\frac{q-r}{r}$.
\end{lem}
Here we are using the convention for slopes of rational tangles in alternating diagrams as described in Section~\ref{Ch4:sec:tanglesindiagrams}.
\begin{proof}
Observe that the fractional tangle satisfies the hypotheses of Proposition~\ref{Ch4:prop:tangledetection}. This shows that it is a rational tangle and its slope may be calculated by \eqref{Ch4:eq:tangleslope}. Observe that there are $\norm{v_F}-1=\norm{\mu_0}-1$ crossings incident to $v$ in the fractional tangle. Thus if the slope of the fractional tangle is $\frac{\alpha}{\beta}$, then \eqref{Ch4:eq:tangleslope} shows this is determined by the continued fraction
\begin{equation*}\label{Ch6:eq:betaalpha1}
\frac{\beta}{\alpha} =[\norm{\mu_0} -1, \norm{\mu_1}, \dots , \norm{\mu_m}]^-.
\end{equation*}
It follows from Lemma~\ref{Ch3:lem:contfractransform} that
\[\frac{\beta}{\alpha}= \frac{q}{q-r}-1=\frac{r}{q-r},\]
and hence that the slope of the fractional tangle is $\frac{\alpha}{\beta}=\frac{q-r}{r}$, as required.
\end{proof}

We are now in a position to prove the following proposition.
\begin{prop}\label{Ch6:prop:CMidentifiestangle}
Let $D$ be an alternating link diagram and suppose that $\Lambda_D$ is isomorphic to a $p/q$-changemaker lattice
\[L=\langle w_0, \dots, w_l\rangle^\bot \subseteq \Z^{t+s+1}\]
where $p/q$ can be written in the form $p/q=n-r/q$ for $p>q>r\geq 1$. Then there is a sequence of flypes to a diagram $\widetilde{D}$ which contains a rational tangle of slope $\frac{q-r}{r}$. Furthermore, if the fractional tangle is replaced by a single crossing $c$ to obtain a alternating diagram $\widetilde{D}'$, then $\Lambda_{\widetilde{D}'}$ is isomorphic to the half-integer changemaker lattice
\[\langle w_0, e_1-e_0 \rangle^\bot \subseteq \Z^{t+2}\]
 in such a way that $c$ is a marked crossing.
\end{prop}
\begin{proof}
Let $\mu_0,\dots, \mu_m$, be the basis for the $L_F$ as constructed in Section~\ref{Ch3:sec:fracparts}. By Lemma~\ref{Ch6:lem:flypetorational} we may flype to a diagram in which $\mu_1, \dots, \mu_m$ are vertices. Lemma~\ref{Ch6:lem:markersunique} shows that any such diagram contains vertices $v$ and $w$ with $v_F=\mu_0$ and $w_F=-(\mu_0+\dots + \mu_m)$. By Lemma~\ref{Ch6:lem:markerbounds} and discussion following it we may further flype to a diagram $\widetilde{D}$ in containing a fractional tangle which is the one determined by the regions $\mu_1, \dots, \mu_m$ and $|v_F\cdot w_F|-1$ crossings between $v$ and $w$. By Lemma~\ref{Ch6:lem:fractionalslope}, the fractional tangle is a rational of the required slope. Let $\widetilde{D}'$ be the alternating diagram obtained by replacing the fractional tangle with a single crossing $c$. Since $\mu_1, \dots, \mu_m,v$ and $w$ are the only regions in $\widetilde{D}$ with non-zero fractional part, we see that $\Lambda_{\widetilde{D}'}$ admits an embedding into $\langle f_1, \dots, f_t, e_0, e_1 \rangle \subseteq \mathbb{Z}^{t+2}$, where $V_{\widetilde{D}'}$ is obtained from $V_{\widetilde{D}}$ by deleting $\mu_1, \dots, \mu_m$ and replacing $v$ and $w$ by $\tilde{v} =v_I+e_1$ and $\tilde{w}=w_I - e_1$, respectively. By considering the image of this embedding we see that $\Lambda_{\widetilde{D}'}$ is isomorphic to the $(n-\frac{1}{2})$-changemaker lattice
\[\langle w_0, e_1-e_0 \rangle^\bot \subseteq \langle f_1, \dots, f_t, e_0, e_1 \rangle = \mathbb{Z}^{t+2}.\]
Since $\tilde{v}\cdot e_0=-\tilde{w}\cdot e_0 = 1$, it is clear from the definition that $c$ is a marked crossing for this embedding.
\end{proof}
Now we give an explicit example to show the tangle replacement of Proposition~\ref{Ch6:prop:CMidentifiestangle} in action.
\begin{example}Let $L$ be the $107/5$-changemaker lattice, given by
\[L=\langle 4f_3+2f_2+f_1+e_0, e_1-e_0, e_2+e_3-e_1 \rangle^\bot \subseteq \mathbb{Z}^{7}.\]
This is isomorphic to the Goeritz form of the alternating knot $11a_{15}$. Figure~\ref{Ch6:fig:11a15example} shows an alternating diagram $D$ of $11a_{15}$ with the induced labeling on the white regions. There is a fractional tangle in $D$ and, as expected, it is of slope $2/3$. Replacing the fractional tangle by a single crossing, $c$, we obtain an alternating diagram $D'$. As shown in Figure~\ref{Ch6:fig:11a15example}, there is an embedding of the Goeritz form $\Lambda_{D'}$ into $\mathbb{Z}^5$ which shows that $\Lambda_{D'}$ is isomorphic to the $43/2$-changemaker lattice
\[L'=\langle 4f_3+2f_2+f_1+e_0, e_1-e_0\rangle^\bot \subseteq \mathbb{Z}^{5}.\]
This isomorphism makes $c$ into a marked crossing. As one would expect from Theorem~\ref{Ch5:thm:markedmeansunknotting} and the fact that $\sigma(D')=-2$, $c$ is a positive unknotting crossing.
\end{example}
\begin{figure}[p!]
  \centering
  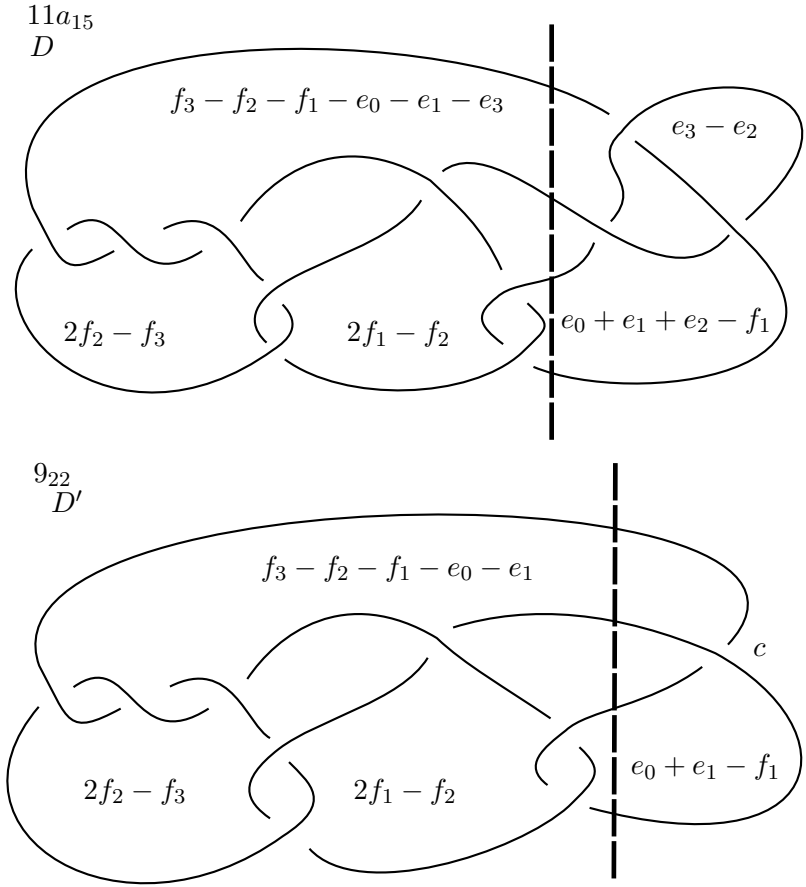
 \caption[Tangle replacement in $11a_{15}$]{A diagram of $11a_{15}$ with the structure of a $107/5$-changemaker lattice. The fractional tangle for this embedding is to the right of the dotted line. Replacing the fractional tangle with a single crossing, $c$, we obtain a diagram of the knot $9_{22}$. The embedding which makes $\Lambda_{D'}$ into a $43/2$-changemaker lattice is also illustrated.}
 \label{Ch6:fig:11a15example}
\end{figure}

This allows us to prove Theorem~\ref{Ch6:thm:rationalsurgery} and Theorem~\ref{intro:thm:nonint}.
\begin{proof}[Proof of Theorem~\ref{Ch6:thm:rationalsurgery}.]
Corollary~\ref{Ch4:cor:CMcondit} gives the implication $(i)\Rightarrow (ii)$. Applications of Proposition~\ref{Ch6:prop:CMidentifiestangle} and Theorem~\ref{Ch5:thm:markedmeansunknotting} prove $(ii) \Rightarrow (iii)$. The implication $(iii) \Rightarrow (i)$ is a consequence of the Montesinos trick as stated in Proposition~\ref{Ch4:prop:3implies1}.
\end{proof}
\begin{proof}[Proof of Theorem~\ref{intro:thm:nonint}]
After reflecting if necessary, this follows immediately from the equivalence $(i)\Leftrightarrow (iii)$ in Theorem~\ref{Ch6:thm:rationalsurgery}.
\end{proof}

\section{Knots in $\mathcal{D}$}
We show that the set $\mathcal{D}$ accounts for all non-integer alternating surgeries.
\begin{thm}\label{intro:thm:montyconverse}
If $S_{p/q}^3(\kappa)$ is an alternating surgery where $q>1$, then there is $\kappa'\in \mathcal{D}$ with
\[S_{p/q}^3(\kappa)\cong S_{p/q}^3(\kappa')
\quad \text{and}\quad
\Delta_\kappa(t)=\Delta_{\kappa'}(t).\]
\end{thm}
\begin{rem}
Since $\kappa$ and $\kappa'$ are both $L$-space knots, the fact that they have the same Alexander polynomials shows that they have isomorphic knot Floer homology.
\end{rem}
Although the existence of $\kappa'$ follows immediately from Theorem~\ref{Ch6:thm:rationalsurgery}, the statement regarding Alexander polynomials requires further work.

\begin{example}\label{Ch6:exam:UnknotinD}
Let $D$ to be one of the unknotted clasp diagrams $\mathcal{C}_m$ as in Figure~\ref{fig:claspdiagram}. If we remove a the interior of small ball about a non-alternating crossing, then the resulting tangle is a rational tangle. This shows that the knot complement of the corresponding knot in $\mathcal{D}$ is homeomorphic to the solid torus, i.e that the corresponding knot is the unknot.
\end{example}

Let $\kappa\in \mathcal{D}$ be obtained from an almost-alternating diagram obtained by changing a marked crossing. We will show how the Alexander polynomial of $\kappa$ can be determined from the corresponding changemaker lattice. First we require a lemma.

\begin{lem}\label{Ch6:lem:CMrigidity}
Let $L=\langle w_0, w_1 \rangle^\bot \subseteq \Z^{t+M}$ and $L'=\langle w_0', w_1' \rangle^\bot \subseteq \Z^{t+M}$ be isomorphic $(n-\frac{1}{M})$-changemaker lattices. If $M\geq n+10$, then $L$ and $L'$ have the same changemaker coefficients.
\end{lem}
\begin{proof}
Pick an orthonormal basis $f_1,\dots, f_t, e_0, \dots , e_{M-1}$ for $\Z^{t+M}$ such that
\[w_0=e_0+ \sigma_1 f_1 + \dots + \sigma_t f_t \quad\text{and}\quad w_1= -e_0 + e_1 + \dots + e_{M-1}.\]
Let $\nu_1, \dots, \nu_t, \mu_1, \dots, \mu_{M-2}$ be a standard basis for $L_M$, where $\mu_i=-e_i+e_{i+1}$ for $1\leq i \leq M-2$ and $\nu_k=-f_k + \sum_{i\in A_k} f_i + \varepsilon_k (e_0+e_1)$ for some $\varepsilon_k\in \{0,1\}$ and some $A_k \subseteq \{1, \dots, k-1\}$ such that $\sigma_k= \varepsilon_k + \sum_{i \in A_k} \sigma_i$.

\paragraph{}Suppose that we have an embedding $\phi\colon L \rightarrow \Z^{r+M}$ whose image is $L'$.
Consider first $u_i=\phi(\mu_i)$, for $1\leq i \leq M-2$. Since the $\mu_i$ satisfy
\[\mu_i \cdot \mu_j =
\begin{cases}
   2      & \text{if } i=j \\
   -1       & \text{if } |i-j|=1\\
   0        & \text{otherwise,}
\end{cases}
\]
there must be mutually orthogonal unit vectors $e_1',\dots, e_{M-1}'$ in $\Z^{r+M}$ for which $u_i=-e_i'+e_{i+1}'$. Since $w_0'\cdot u_i=w_1'\cdot u_i=0$ for all $i$, we must have
\[w_0'\cdot e_1' = \dots = w_0'\cdot e_{M-1}' \quad \text{and} \quad w_1'\cdot e_1' = \dots = w_1'\cdot e_{M-1}'.\]
It follows that $n=\norm{w_0}\geq(M-1)(w_0'\cdot e_1')^2$ and $M=\norm{w_1}\geq(M-1)(w_1'\cdot e_1')^2$. As $M > n+2$, we must have $w_0'\cdot e_1'=0$ and $|w_1'\cdot e_1'|\leq 1$. As $L'$ contains no unit vectors, we can assume that $w_1'\cdot e_1'=1$. As $\norm{w_1}=M$, there must be some other unit vector $e_0'$ such that $w_1'=-e_0'+e_1' + \dots + e_{M-1}'$. As $w_0'\cdot w_1'=-1$, we have $w_0'\cdot e_0'=1$.

Now consider $v_k=\phi(\nu_k)$ for any $1\leq k \leq t$. First observe that we have $\nu_k\cdot w_0'=0$, $\nu_k\cdot \mu_1=-\varepsilon_k$, and $\nu_k\cdot \mu_i=0$ for $i\geq 2$. Therefore, $v_k \cdot e_2' = \dots = v_k\cdot e_{M-1}'$ and $v_k\cdot e_0'=v_k\cdot e_1'=v_k\cdot e_2' + \varepsilon_k$. This gives
\[2(v_k\cdot e_2' + \varepsilon_k)^2 + (M-2)(v_k\cdot e_2')^2 \leq \norm{v_k}.\]
As $\norm{\nu_k}\leq n+1 <M-2$, it follows that
\[v_k\cdot e_1' = v_k\cdot e_0'= \varepsilon_k
\quad \text{and}\quad v_k \cdot e_2' = \dots = v_k\cdot e_{M-1}'=0\]
for all $k$.

As $\varepsilon_1=1$ and $\norm{v_1}=3$ it follows that there must be a unit vector $f_1'$ such that $v_1=-f_1'+e_0'+e_1'$. Hence we have $w_0'\cdot f_1'=\sigma_1 =1$.

We will complete the proof by inductively showing that there are orthonormal vectors $f_1', \dots, f_t'$ such that
\[v_k=-f_k' + \sum_{i\in A_k} f_i' + \varepsilon_k (e_0'+e_1'),\]
for all $k$, and
\[w_0'=e_0+ \sigma_1 f_1' + \dots + \sigma_t f_t'.\]

Now suppose that there are $f_1', \dots, f_{k-1}'$, such that
\[v_j=-f_j' + \sum_{i\in A_j} f_i' + \varepsilon_j (e_0'+e_1')
\quad \text{and} \quad
w_0'\cdot f_j=\sigma_j.\]
for all $j<k$. We have
\begin{align*}
v_k\cdot f_j'&=-v_k\cdot v_j + \sum_{i\in A_j} v_k\cdot f_i' + \varepsilon_k \varepsilon_j\\
&= -\nu_k\cdot \nu_j + \sum_{i\in A_j} \nu_k\cdot f_i + \varepsilon_k \varepsilon_j\\
&= \nu_k\cdot f_j
\end{align*}
for all $j<k$. As
\[\norm{v_k}-1=\norm{\nu_k}-1 = \sum_{i=1}^{k-1}\nu_k\cdot f_j + 2\varepsilon_k,\]
this shows that there is a unit vector $f_k'\in \Z^{t+M}$ such that
\[v_k=-f_k' + \sum_{i\in A_k} f_i' + \varepsilon_k (e_0'+e_1').\]
Since $v_k\cdot w_0'=0$, it follows that
\[w_0'\cdot f_k'=\sum_{i\in A_k} \sigma_i + \varepsilon_k=\sigma_k,\]
as required.
\end{proof}

\begin{lem}\label{Ch6:lem:kappaAlexpoly}
Let $D$ be a reduced alternating diagram with an isomorphism
\[\Lambda_D \rightarrow L=\langle w_0, w_1 \rangle^\bot \subseteq \Z^{t+2},\]
where $L$ is a $(n-\frac{1}{2})$-changemaker lattice. If $c$ is a marked crossing for this embedding and $\kappa \in \mathcal{D}$ is the knot arising from the almost-alternating diagram obtained by changing $c$, then the torsion coefficients of $\Delta_\kappa(t)$ satisfy
\[8t_{i}(\kappa) = \min_{ \substack{ |c'\cdot w_0|= n-2i \\ c' \in \Char(\Z^{t+2})}} \norm{c'} - (t+2),\]
for $0\leq i \leq n/2$ and $t_{i}(\kappa)=0$ for $i\geq n/2$.
\end{lem}
\begin{proof}
Let $D_M$ be the diagram obtained by replacing the marked crossing $c$ by a $(M-1)$-tangle $T$ for some $M\geq n+99$, as shown in Figure~\ref{Ch6:fig:twistintro2}.
\begin{figure}[h]
  \centering
  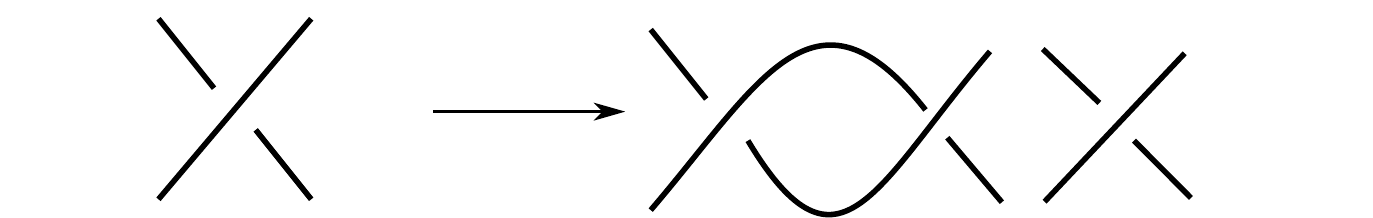
 \caption{The diagrams $D$ and $D_M$ arising in the proof of Lemma~\ref{Ch6:lem:kappaAlexpoly}.}
 \label{Ch6:fig:twistintro2}
\end{figure}
By Proposition~\ref{Ch4:prop:3implies1}, this is the link whose double branched cover arises by $-(n-\frac{1}{M})$-surgery on $\kappa$. Thus by Corollary~\ref{Ch4:cor:CMcondit}, $\Lambda_D$ is isomorphic to a $(n-\frac{1}{M})$-changemaker lattice
\[L'=\langle w_0', w_1'\rangle^\bot \subseteq \Z^{t+M},\]
where $w_0'$ computes the torsion coefficients of $\Delta_\kappa(t)$ by
\begin{equation}\label{Ch6:eq:kappatorsion}
8t_{i}(\kappa) = \min_{ \substack{ |c'\cdot w_0'|= n-2i \\ c' \in \Char(\Z^{t+M})}} \norm{c'} - (t+M),
\end{equation}
for $0\leq i \leq n/2$ and $t_i(K)=0$ for $i\geq n/2$.

However, by the embedding given in Figure~\ref{Ch6:fig:twistintro2}, we see that $\Lambda_{D_M}$ is isomorphic to the $(n-\frac{1}{M})$-changemaker lattice
\[L=\langle w_0, -e_0 + e_1 + \dots + e_{M-1} \rangle^\bot\subseteq \Z^{t+M}.\]
Thus the lattices $L$ and $L'$ are isomorphic. As we have chosen $M$ to be large, Lemma~\ref{Ch6:lem:CMrigidity} applies to show $L$ and $L'$ have the same changemaker coefficients. Consequently, for $0\leq i \leq n/2$,
\begin{equation*}
8t_{i}(\kappa) = \min_{ \substack{ |c'\cdot w_0|= n-2i \\ c' \in \Char(\Z^{t+M})}} \norm{c'} - (t+M)
=\min_{ \substack{ |c'\cdot w_0'|= n-2i \\ c' \in \Char(\Z^{r+M})}} \norm{c'} - (t+M).
\end{equation*}
It follows that
\begin{equation*}
8t_{i}(\kappa) = \min_{ \substack{ |c'\cdot w_0|= n-2i \\ c' \in \Char(\Z^{t+2})}} \norm{c'} - (t+2),
\end{equation*}
for $0\leq i \leq n/2$ (c.f Remark~\ref{Ch1:rem:removezeroes}). It also follows that $t_i(K)=0$ for $i\geq n/2$ (c.f Remark~\ref{Ch1:rem:CMgivesallVi}).
\end{proof}
We can now prove the main result of this section.
\begin{proof}[Proof of Theorem~\ref{intro:thm:montyconverse}]
Suppose that we have a knot $\kappa$ with an alternating surgery. By reflecting if necessary, suppose that $S_{-p/q}^3(\kappa)=\Sigma(L)$ for some alternating knot or link $L$ and some $p/q>0$.

If $p/q<1$, then Proposition~\ref{Ch6:prop:smallsurgery} shows that $c$ is the unknot. Since the unknot is in $\mathcal{D}$, there is nothing further to prove.

Now assume that $p/q>1$. Let $D$ be any reduced alternating diagram for $L$. Corollary~\ref{Ch4:cor:CMcondit} shows that $\Lambda_D$ is isomorphic to a $p/q$-changemaker lattice
\[\langle w_0, \dots, w_l \rangle^\bot \subseteq \Z^{t+s+1},\]
where the torsion coefficients of $\Delta_\kappa(t)$ satisfy
\begin{equation*}
8t_{i}(\kappa) = \min_{ \substack{ |c'\cdot w_0|= n-2i \\ c' \in \Char(\Z^{N})}} \norm{c'} - N,
\end{equation*}
for  $0\leq i\leq n/2$, where $n= \lceil p/q \rceil$ and $N=t+s+1$.

Proposition~\ref{Ch6:prop:CMidentifiestangle} shows that $L$ possesses an alternating diagram obtained by replacing the marked crossing in a diagram $D'$ with a rational tangle of slope $\frac{q-r}{r}$, where the marked crossing arises from an isomorphism
\[\Lambda_D \rightarrow \langle w_0, e_0-e_1 \rangle^\bot\subseteq \Z^{t+2}.\]
Theorem~\ref{Ch5:thm:markedmeansunknotting} and Proposition~\ref{Ch4:prop:3implies1} show that there is $\kappa' \in \mathcal{D}$ with
\[S_{-p/q}^3(\kappa')\cong \Sigma(L) \cong S_{-p/q}^3(\kappa).\]
By Lemma~\ref{Ch6:lem:kappaAlexpoly}, the torsion coefficients satisfy $t_i(\kappa)=t_i(\kappa')$ for all $i\geq 0$. As shown in Remark~\ref{Ch1:rem:torsiondeterminespoly}, the torsion coefficients determine the Alexander polynomial, so we have $\Delta_{\kappa'}(t)=\Delta_{\kappa}(t)$, as required.
\end{proof}

\section{Bounds on alternating surgeries}
Given a knot $\kappa\subseteq S^3$, we provide some restrictions on the surgery slopes for which $S_{p/q}^3(\kappa)$ can be an alternating surgery. As surgery on the unknot always results in a lens space, we see that every non-zero surgery on the unknot is an alternating surgery. This behaviour is not typical; all other knots admit alternating surgeries for only a restricted range of slopes.

\begin{thm}\label{intro:thm:widthbound}
Let $\kappa$ be a non-trivial knot admitting alternating surgeries. There is an integer $N$, which can be calculated from the Alexander polynomial of $\kappa$ such that if $S_{p/q}^3(\kappa)$ is an alternating surgery, then
\[N-1\leq p/q \leq N+1.\]
\end{thm}
Since it can be shown that the $N$ in Theorem~\ref{intro:thm:widthbound} satisfies $|N|\leq 4g(\kappa)+2$, we obtain a generalization of a bound on lens space surgeries originally due to Rasmussen \cite{Rasmussen04Goda}.
\begin{cor}\label{intro:cor:Rasmusbound}
Let $\kappa$ be a non-trivial knot. If $S_{p/q}^3(\kappa)$ is an alternating surgery, then
\[|p/q| \leq 4g(\kappa)+3.\]
\end{cor}

The work of Greene provides a lower bound on alternating surgery slopes in terms of the genus \cite[Theorem~1.1]{greene2010space}.
\begin{thm}[Greene]
If $S_{p/q}^3(\kappa)$ is an alternating surgery and $\kappa$ is non-trivial, then
\[|p/q|\geq 2g(\kappa)+ \frac{1+\sqrt{24g(\kappa)+1}}{2}\]
\end{thm}
\begin{proof}
Let $S_{p/q}^3(\kappa)$ be an alternating surgery on a non-trivial knot. By reflecting if necessary, we may assume that $p/q>1$. Corollary~\ref{Ch4:cor:CMcondit} shows that there is a collection of changemaker coefficients $\sigma_1\leq  \dots \leq  \sigma_t$ such that
\[2g(\kappa)=\sum_{i=1}^t \sigma_i (\sigma_i -1)
\quad \text{and}\quad
n=\lfloor \frac{p}{q} \rfloor =\sum_{i=1}^t \sigma_i^2.
\]
By Proposition~\ref{Ch3:prop:CMcondition}, $\sigma_i$ satisfies
\[ \sigma_i \leq 1+ \sigma_1 + \dots + \sigma_{i-1},\]
for all $i$, so we see that
\begin{align*}
(1+\sigma_1+ \dots + \sigma_t)^2 &= 1+ \sum_{i=1}^t \sigma_i^2 + 2\sigma_i(1+ \sigma_1 + \dots + \sigma_{i-1})\\
&\geq 1 + 3\sum_{i=1}^t \sigma_i^2=3n+1.
\end{align*}
This shows that the genus satisfies,
\[2g(\kappa)\leq n+1- \sqrt{3n+1}.\]
When $g(\kappa)\geq 1$, this implies that
\[\frac{p}{q}\geq n \geq 2g(\kappa)+ \frac{1+\sqrt{24g(\kappa)+1}}{2},\]
as required.
\end{proof}
\begin{rem}
This shows that the bound in Proposition~\ref{Ch6:prop:smallsurgery} can improved: if $S_{p/q}^3(\kappa)$ is an alternating surgery and $|p/q|<5$, then $\kappa$ is the unknot.
\end{rem}
It is also possible to say something about the knots which admit alternating surgeries in the full range allowed by Theorem~\ref{intro:thm:widthbound}.
\begin{thm}\label{ch6:thm:fullrange}
Suppose that $K$ is a nontrivial knot admitting alternating surgeries $S_{r}^3(K)$ for each of the slopes $r\in \{p'/q',p/q,N\}$, where $N$ is the integer appearing in Theorem~\ref{intro:thm:widthbound}. If $p/q$ and $p'/q'$ satisfy
\[N-1\leq p'/q'<N<p/q< N+1,\]
then $S_{N}^3(K)$ is a reducible surgery and $K$ is a torus knot or a cable.
\end{thm}

\subsection{Proving Theorem~\ref{intro:thm:widthbound} and Theorem~\ref{ch6:thm:fullrange}}
In this section we prove Theorem~\ref{intro:thm:widthbound}, Corollary~\ref{intro:cor:Rasmusbound} and Theorem~\ref{ch6:thm:fullrange}.
\begin{proof}[Proof of Theorem~\ref{intro:thm:widthbound} and Corollary~\ref{intro:cor:Rasmusbound}]
Let $\kappa$ be a nontrivial knot with an alternating surgery. Since a non-trivial knot cannot admit both positive and negative $L$-space surgeries, we may assume that all alternating surgeries on $\kappa$ are negative. If $S_{-p/q}^3(\kappa)\cong \Sigma(L)$ is an alternating surgery for some $p/q>0$, then Corollary~\ref{Ch4:cor:CMcondit} shows that there is a $p/q$-changemaker lattice $L_{p/q}$ such that any reduced alternating diagram $D'$ of $L$ satisfies $\Lambda_{D'}\cong L_{p/q}$. Let $(\sigma_1, \dots, \sigma_t)$ be the changemaker coefficients for $L_{p/q}$ and $2\leq \rho_1 \leq \dots \leq \rho_m$ be the corresponding stable coefficients. We will take
\[N=\rho_1 +\sum_{i=1}^m \rho_i^2.\]
Proposition ~\ref{Ch3:prop:CMcondition} shows that at least $\rho_1-1$ of the $\sigma_i$ are equal to one. Therefore, we have
\begin{align*}
\lfloor p/q \rfloor &= \sum_{i=1}^t \sigma_i^2 \\
&=\sum_{i=1}^m \rho_i^2 + |\{i \,|\, \sigma_i=1\}|\\
&\geq \sum_{i=1}^m \rho_i^2 + \rho_1-1=N-1.
\end{align*}
This proves the lower bound.

Let $n=\lceil p/q \rceil$ and $L_n$ be the $n$-changemaker lattice with the same stable coefficients as $L_{p/q}$.
\begin{claim}
There is an alternating diagram $D$, with $\Lambda_D\cong L_n$.
\end{claim}
\begin{proof}
If $p/q\not\in \Z$, then Proposition~\ref{Ch6:prop:CMidentifiestangle} shows that there is a diagram $D''$ such that $\Lambda_{D''}\cong L_{n-\frac{1}{2}}$, where $L_{n-\frac{1}{2}}$ is the $(n-\frac{1}{2})$-changemaker lattice with the same stable coefficients as $L_{p/q}$.
Let $V_{D''}\subseteq L_{n-\frac{1}{2}}$ be the corresponding set of vertices. If we take
\[L_{n-\frac{1}{2}}= \langle w_0, e_1- e_0 \rangle^\bot \subseteq \Z^{t+2},\]
then there are marker vertices of the form $v'+e_0+e_1$ and $w'-e_0-e_1$ in $V_{D''}$ where $v'\cdot e_0=w'\cdot e_0=0$. By Lemma~\ref{Ch5:lem:markedcexist}, this gives us a marked crossing in $D''$.
If we resolve this marked crossing as in Figure~\ref{Ch6:fig:resolveup}, then we obtain a diagram $D$ such that $\Lambda_D \cong L_n$, where
\[L_n=\langle w_0\rangle^\bot \subseteq \Z^{t+1}.\]
\begin{figure}[h]
  \centering
  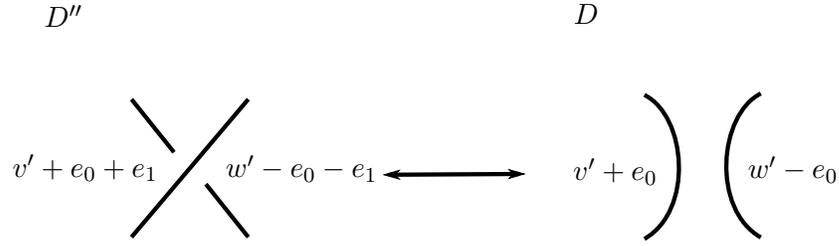
 \caption{The diagrams $D''$ and $D$ arising in the proof of Theorem~\ref{intro:thm:widthbound}.}
 \label{Ch6:fig:resolveup}
\end{figure}
The set of vertices $V_D\subseteq L_n$ for this isomorphism are obtained by replacing $v'+e_0+e_1$ and $w'-e_0-e_1$ in $V_{D''}$ with $v'+e_0$ and $w'-e_0$, respectively.
\end{proof}

Take $D$ to be a diagram such that $\Lambda_D \cong L_n$. Write $L_n$ in the form
\[L_n=\langle f_1 + \dots + f_{r}+ \rho_1 f_{r+1} + \dots + \rho_m f_{m+r} \rangle^\bot\subseteq \Z^{r+m},\]
where $n=r+ \sum_{i=1}^m \rho_i^2$. We will complete the proof of Theorem~\ref{intro:thm:widthbound} by showing that $r\leq \rho_1+1$.

Let $V_D$ be the set of vertices arising from the isomorphism $\Lambda_D\cong L_n$. Note that for $2\leq i\leq r$, we have $\nu_i=f_{i-1}-f_i\in L_n$. Assume that $r\geq \rho_1+1$.
\begin{claim}
We may assume $\{\nu_1, \dots, \nu_{m-1}\}\subseteq V_D$.
\end{claim}
\begin{proof}
By Lemma~\ref{Ch3:lem:reduciblebound}, $r\geq \rho_1+1$ shows that $L_n$ is indecomposable. By Lemma~\ref{Ch3:lem:2connectgraphlat}, this means $\Gamma_D$ is 2-connected and every vertex in $V_D$ is irreducible. This means we can apply Lemma~\ref{Ch4:lem:flype1} to flype $D$.

Let $k\leq r$, be minimal such that $\nu_k$ is not in $V_D$. Suppose first that $k=2$. Since $\nu_2$ is irreducible, Lemma~\ref{Ch3:lem:irreducible} implies that it can be written as a sum of elements of $V_D$. Hence, there is a vertex $u\in V_D$ with $u\cdot \nu_2>0$. As $u$ is irreducible and not equal to $\nu_2$, we have $(u-\nu_2)\cdot \nu_2=u\cdot \nu_2 -2 = -1$. Therefore via Lemma~\ref{Ch4:lem:flype1} we can flype to obtain $D'$ with $\nu_2\in V_{D'}$.

Now we suppose that $k>2$. Since $\nu_k$ is irreducible, it can be written as a sum of elements of $V_D$. Since $\nu_{k-1}$ is a vertex of $B$ and $\nu_k\cdot \nu_{k-1}=-1$, there is $u\in V_D$ with $u\cdot \nu_k = -u\cdot \nu_{k-1}=1$. This must satisfy $(u-\nu_k)\cdot \nu_k= -1$. By Lemma~\ref{Ch4:lem:flype1}, we can flype to obtain a diagram $D'$ with $\nu_k\in V_{D'}$. Moreover, since $(u-\nu_k)\cdot \nu_j \leq 0$ and $\nu_k\cdot \nu_j \leq 0$ for all $2\leq j<k$, we see that $\nu_2, \dots, \nu_{k-1}$ are in $V_{D'}$. Thus proceeding inductively, we see that we can assume that $D$ is chosen such that $\nu_2, \dots, \nu_{m-1}$ are all in $V_D$.
\end{proof}

Now consider the vector $z=-f_{r+1} + f_1 + \dots + f_{\rho_1}\in L_n$. This is irreducible so it can written as a sum $z=\sum_{x\in R}x$ for some subset $R\subseteq V_D$. As $z\cdot \nu_{\rho_1 +1}=1$, we must have  $\nu_{\rho_1 +1}\in R$ and there must be a vertex $v\in V_D$ with $v\cdot z \leq -1$ and $v\cdot \nu_{\rho_1 +1}=-1$. As $\norm{\nu_{\rho_1 +1}}=2$, there are precisely two vertices that pair non-trivially with $\nu_{\rho_1 +1}$. One of these is $\nu_{\rho_1}$ which has $\nu_{\rho_1}\cdot z=0$. If $r>\rho_1+1$, then we can take $\nu_{\rho_1 +2}\in L_n$. As this pairs non-trivially with $\nu_{\rho_1 +1}$ and also satisfies $\nu_{\rho_1+2}\cdot z=0$, it follows that $r= \rho_1+1$. This shows $p/q\leq n\leq N+1$, completing the proof of Theorem~\ref{intro:thm:widthbound}.

Now we deduce Corollary~\ref{intro:cor:Rasmusbound}. As  $\rho_i \geq 2 $ for all $i$ we have
\[\rho_i^2 \leq 2 \rho_i (\rho_i -1)
\quad\text{and}\quad
\rho_1^2 + \rho_1 \leq 2\rho_1 (\rho_1 -1) +2.\]
Combining these in equalities with $g(\kappa)=\sum_{i=1}^m \rho_i(\rho_i-1)$, we have
\begin{align*}
p/q \leq N+1 &=\rho_1^2 + \rho_1 +1+\sum_{i=2}^m \rho_i^2 \\
    &\leq 2\sum_{i=1}^m \rho_i(\rho_i-1) +3 \\
    &=4g(\kappa) + 3,
\end{align*}
as required.
\end{proof}

\begin{proof}[Proof of Theorem~\ref{ch6:thm:fullrange}]
Suppose that $S_{-p/q}^3(\kappa)$ is an alternating surgery for some $N<p/q<N+1$. There is an alternating diagram $D_{p/q}$ and a $p/q$-changemaker lattice
\[L_{p/q}=\langle w_0,\dots, w_l \rangle^\bot\subseteq \Z^{t+s+1},\]
such that $\Lambda_{D_{p/q}}\cong L_{p/q}$. Write $w_0$ in the form
\[w_0=e_0 + f_1 + \dots + f_{m-1} + \sigma_m f_m + \dots + \sigma_t f_t,\]
where $2\leq \sigma_m \leq \dots \leq \sigma_t$ are the stable coefficients.

By Proposition~\ref{Ch6:prop:CMidentifiestangle}, there is an alternating diagram $D_{N+\frac{1}{2}}$, such that
\[\Lambda_{D_{N+\frac{1}{2}}}\cong L_{N+\frac{1}{2}}=\langle w_0, e_1- e_0 \rangle^\bot \subset \Z^{t+2}.\]
If $S_{-p'/q'}^3(\kappa)$ is an alternating surgery for some $N-1 \leq p'/q' <N$, then
\[\langle f_2+ \dots + f_{m-1} + \sigma_m f_m + \dots + \sigma_t f_t \rangle^\bot \subseteq \Z^{t-1}\]
must also be a changemaker lattice. In particular is shows that $L_{N+\frac{1}{2}}$ is slack.

As the set of stable coefficients is non-trivial, $D_{N+\frac{1}{2}}$ is not a clasp diagram. Therefore Lemma~\ref{Ch5:lem:markedcexist} and Lemma~\ref{Ch5:lem:manymarkedcrossings} show that $D_{N+\frac{1}{2}}$ contains a unique marked crossing $c$. Let $v$ and $w$, be the marker vertices for this marked crossing. As $L_{N+\frac{1}{2}}$ is slack, Lemma~\ref{Ch5:lem:slackvtight} shows that we can assume that $v$ and $w$ form a cut set in $\Gamma_{D_{N+\frac{1}{2}}}$. Thus if $D_N$ is the diagram obtained by resolving $c$ so that the regions corresponding to $v$ and $w$ merge, then we can assume that $\Gamma_{D_N}$ is not 2-connected and, in particular, that $D_N$ is not prime. However $\Lambda_{D_N}$ is isomorphic to the $N$-changemaker lattice
\[L_N=\langle f_1+ \dots + f_{m-1} + \sigma_m f_m + \dots + \sigma_t f_t \rangle^\bot \subseteq \Z^{t}.\]
By Lemma~\ref{Ch3:lem:2connectgraphlat}, this shows that $L_N$ is decomposable. However, if $S_N(\kappa)^3 \cong \Sigma(L)$ is an alternating surgery, then for any reduced alternating diagram $D$ of $L$ satisfies $\Lambda_D\cong L_N$. This shows that any such $D$ is not prime. Therefore, $S_N^3(\kappa)$ is reducible.

Since $N>2g(\kappa)-1$, it follows from the work of Hoffman and Matignon-Sayari on the cabling conjecture that $\kappa$ is either a cable or a torus knot \cite{hoffman98reducing, Matignon03longitudinal}.
\end{proof}

\section{Tsukamoto's theorem}
Finally, we complete our proof of Tsukamoto's theorem on the almost-alternating diagrams of the unknot. \cite[Corollary 1.1]{tsukamoto2009almost}.
\begin{thm}[Tsukamoto]\label{intro:thm:tsukamoto}
Any reduced almost-alternating diagram of the unknot can be obtained from $\mathcal{C}_m$, for some non-zero integer $m$, by a sequence of flypes, tongue moves and twirl moves. For each $m$, $\mathcal{C}_m$ is the almost-alternating diagram shown in Figure~\ref{fig:claspdiagram}.
\end{thm}

The proof is given by combining Theorem~\ref{Ch5:thm:markedmeansunknotting} with the following lemma.
\begin{lem}\label{Ch6:lem:everymarked}
Let $D$ be a reduced alternating diagram with an unknotting crossing $c$ such that either $c$ is positive and $\sigma(D)=-2$ or $c$ is negative and $\sigma(D)=0$, then there is an isomorphism from $\Lambda_D$ to a half-integer changemaker lattice for which $c$ is the marked crossing.
\end{lem}
\begin{proof}
Construct a new alternating diagram $D_M$ by replacing $c$ by a $\frac{1}{M-1}$-tangle $T$ for some $M$ greater than the number of crossings in $D$. As shown in Figure~\ref{Ch6:fig:twistintro}, this means replacing $c$ by a twist containing $M-1$ crossings.
\begin{figure}[h]
  \centering
  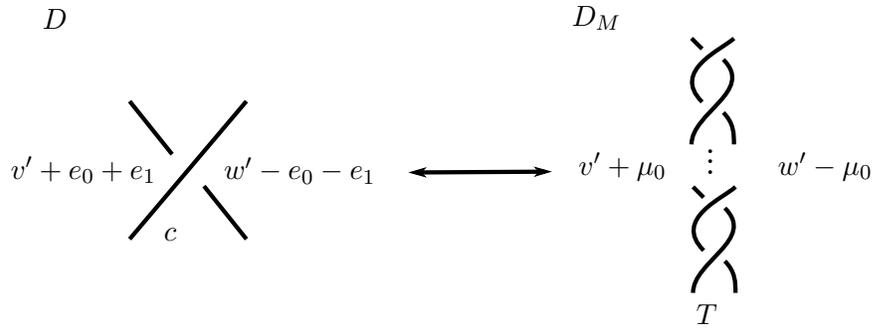
 \caption{The diagrams $D$ and $D_M$ arising in the proof of Lemma~\ref{Ch6:lem:everymarked}.}
 \label{Ch6:fig:twistintro}
\end{figure}
By Proposition~\ref{Ch4:prop:3implies1}, this means that $\Sigma(D_M)$ arises by $-p/M$-surgery where $p/M=n- \frac{M-1}{M}$ for some $n>0$. By Theorem~\ref{Ch6:thm:rationalsurgery}, we know that $\Lambda_{D_M}$ is isomorphic to a $p/M$-changemaker lattice $L$, and we get a corresponding labeling of regions of $D_M$ by elements of $L$. The continued fraction for $n- \frac{M-1}{M}$ is $[n,\underbrace{2, \dots, 2}_{M-1}]^-$. So $L$ takes the form
\[L=\langle e_0+ \sigma_1 f_1 + \dots + \sigma_t f_t, e_1-e_0, \dots, e_{M-1} - e_{M-2} \rangle^\bot \subseteq \Z^{t+M}.\]
As the fractional part $L_F$ is spanned by $\mu_0=e_0+ \dots +e_{M-1}$, any vector in $x\in L$ with non-zero fractional part satisfies $\norm{x}\geq M$. Since the image of $\Lambda_{D'}$ must span $L$, there are regions $v$ and $w$ with non-zero fractional part and these must be adjacent to at least $M$ crossings. However, as $M$ exceeds the crossing number of $D$, we see that the only regions of $D_M$ which adjacent to more than $M$ crossings are the regions adjacent to the tangle $T$ that we introduced. Therefore the regions adjacent to $T$ are a labeled by $v=v'+\mu_0$ and $w=w'-\mu_0$, for some $v'$ and $w'$ contained in $\langle f_1,\dots, f_t \rangle$ and $\mu_0$ does not appear in any other region of $D_M$. By replacing $v$ and $w$ by $v'+e_0+e_1$ and $w'-e_0-e_1$ respectively, we obtain an embedding of $\Lambda_D$ into $\Z^{t+2}$ whose image is the half-integer changemaker lattice
\[L'=\langle e_0+ \sigma_1 f_1 + \dots + \sigma_t f_t, e_1-e_0 \rangle^\bot.\]
It is clear that $c$ is the marked crossing for this embedding.
\end{proof}

\begin{proof}[Proof of Theorem~\ref{intro:thm:tsukamoto}]
Up to reflection, Lemma~\ref{Ch6:lem:everymarked} shows that every almost-alternating diagram of the unknot can be obtained by changing a marked crossing in an alternating diagram. However, Theorem~\ref{Ch5:thm:markedmeansunknotting} shows that any diagram obtained by changing a marked crossing can be obtained from $\mathcal{C}_m$, for some non-zero integer $m$, by a sequence of flypes, tongue moves and twirl moves.
\end{proof}

\bibliographystyle{alpha}
\bibliography{thesis}	
\end{document}